\documentclass[10pt]{article}
\usepackage{amssymb,latexsym,amsmath,enumerate,verbatim,amsfonts,amsthm,cite,lscape,algorithm,algorithmic}
\usepackage{color} % added by TK
\usepackage[active]{srcltx}
\usepackage{microtype}

\newtheorem{lemma}{Lemma}[section]
\newtheorem{definition}{Definition}[section]

\newtheorem{theorem}{Theorem}[section]
\newtheorem{corollary}{Corollary}[section]

\newtheorem{remark}{Remark}[section]
\newtheorem{example}{Example}[section]
\newtheorem{assumption}{Assumption}[section]

\numberwithin{equation}{section}

\def\R{{\mathbb{R}}}
\def\D{{\frak D}}
\def\B{{\frak B}}

\def\dist{{\rm dist}}
\def\dom{{\rm dom}\,}
\def\Argmin{\mathop{\rm Arg\,min}}

\definecolor{blue}{rgb}{0,0,0} % uncomment for ArXiv versions
\newcommand{\revise}[1]{\textcolor{blue}{\ignorespaces#1\ignorespaces}}   % REVISE
\newcommand{\reviser}[1]{\textcolor{blue}{\ignorespaces#1\ignorespaces}}

\title{\sf Kurdyka-{\L}ojasiewicz exponent via inf-projection}

\author{
Peiran Yu \thanks{Department of Applied
Mathematics, The Hong Kong Polytechnic University, Hong Kong, China. E-mail: {peiran.yu@connect.polyu.hk}.}
\and Guoyin Li \thanks{Department of Applied Mathematics, University of New South Wales, Sydney, Australia. This author is partially supported by a Future fellowship from Australian Research Council (FT130100038) and a discovery project from Australian Research Council (DP190100555). E-mail: {g.li@unsw.edu.au}.}
\and Ting Kei Pong \thanks{Department of Applied Mathematics, The Hong Kong Polytechnic University, Hong Kong, China.
This author was supported partly by Hong Kong Research Grants Council PolyU153005/17p. E-mail: {tk.pong@polyu.edu.hk}.}
}

\date{Revised version: January 13, 2021}

\begin{document}
  \maketitle

\begin{abstract}
Kurdyka-{\L}ojasiewicz (KL) exponent plays an important role in estimating the convergence rate of many contemporary first-order methods. In particular, a KL exponent of $\frac12$ \color{blue}for a suitable potential function \color{black} is related to local linear convergence. Nevertheless, KL exponent is in general extremely hard to estimate. In this paper, we show under mild assumptions that KL exponent is preserved via inf-projection.
Inf-projection is a fundamental operation that is ubiquitous when reformulating optimization problems via the lift-and-project approach. By studying its operation on KL exponent, we show that the KL exponent is $\frac12$ for several important convex optimization models, including some semidefinite-programming-representable functions and \revise{some} functions that involve $C^2$-cone reducible structures, under conditions such as strict complementarity. Our results are applicable to concrete optimization models such as group fused Lasso and overlapping group Lasso. In addition, for nonconvex models, we show that the KL exponent of many difference-of-convex functions can be derived from that of their natural majorant functions, and the KL exponent of the Bregman envelope of a function is the same as that of the function itself. Finally, we estimate the KL exponent of the sum of the least squares function and the indicator function of the set of matrices of rank at most $k$.
\end{abstract}

\section{Introduction}

Many problems in machine learning, signal processing and data analysis involve large-scale nonsmooth nonconvex optimization problems. These problems are typically solved using first-order methods, which are noted for their scalability and ease of implementation. Commonly used first-order methods include the proximal gradient method and its variants, and splitting methods such as Douglas-Rachford splitting method and its variants; see the recent expositions \cite{BoPaChPeEc10,PaBo2012} and references therein for more detail. In the general nonconvex \color{blue}nonsmooth \color{black} setting, convergence properties of the sequences generated by these algorithms are typically analyzed by assuming a certain potential function to have the so-called Kurdyka-{\L}ojasiewicz (KL) property.

%Loosely speaking, the KL property holds when it is possible to bound the function value deviation in terms of a generalized notion of ``gradient"; see Definition~\ref{DefKL} below for the precise definition. \revise{The idea of generalizing the KL property to nonsmooth case and using it to deduce convergence rate of subgradient dynamic system was  proposed earlier  in \cite{BoDaLe07}. }
\color{blue}The KL property originates from the seminal {\L}ojasiewicz inequality that bounds  the function value deviation  of a  real-analytic function in terms of  its gradient; see \cite{Loja63}. This inequality was extended to the case of $C^1$ subanalytic functions by Kurdyka in  \cite{Kur98} using the notion of desingularizing function.  An important breakthrough was made in \cite{BoDaLe07,BoDaLeSh07}, where the {\L}ojasiewicz  inequality was further generalized to nonsmooth cases by using tools of modern variational  analysis and semialgebraic geometry. This generalization significantly broadened the applicability of the aforementioned KL inequality to nonconvex settings, and it allowed us to perform convergence rate
analysis for various important algorithms in nonsmooth optimization and subgradient dynamical systems. \color{black}

The KL property\footnote{\revise{See Definition~\ref{DefKL} for the precise definition.}} is satisfied by a large class of functions such as proper closed semi-algebraic functions; see, for example, \cite{AtBoReSo10}.
It has been the main workhorse for establishing convergence of sequences generated by various first-order methods, especially in nonconvex settings \cite{AtBo09,AtBoReSo10,AtBoSv13,BoSaTe14}. Moreover, when it comes to estimating {\em local convergence rate}, the so-called KL exponent plays a key role; see, for example, \cite[Theorem~2]{AtBo09}, \cite[Theorem~3.4]{FranGarPey15} and \cite[Theorem~3]{LiPong16}. Roughly speaking, an exponent of \revise{$\alpha\in (0,\frac12]$} of a suitable potential function corresponds to a linear convergence rate, while an exponent of $\alpha \in (\frac12,1)$ corresponds to a sublinear convergence rate. However, as noted in
\cite[Page~63, Section~2.1]{LuoPangRalph96}, explicit estimation of KL exponent for a given function is difficult in general. Nevertheless, due to its significance in convergence rate analysis, KL exponent computation has become an important research topic in recent years and some positive results have been obtained. For instance, we now know the KL exponent of the maximum of finitely many polynomials \cite[Theorem~3.3]{LiMorPham15} and the KL exponent of a class of quadratic optimization problems with matrix variables satisfying orthogonality constraints \cite{LiuWuSo16}. In addition, it has been shown that the KL exponent is closely related to several existing and widely-studied error bound concepts such as the H\"{o}lder growth condition and \color{blue}the first-order error bound mentioned in \cite{BoNgPeSu17,LuTs93,TsYu09};\footnote{\color{blue}This type of first-order error bound is sometimes called the Luo-Tseng error bound; see \cite{LiPong17,YuZhSo19}.\color{black}} \color{black} see for example, \cite[Theorem~5]{BoNgPeSu17}, \cite[Theorem~3.7]{DrIoLewis16}, \cite[Proposition~3.8]{DrIoLewis16}, \cite[Corollary~3.6]{DrLewis18} and \cite[Theorem~4.1]{LiPong17}. Taking advantage of these connections, we now also know that convex models that satisfy the second-order growth condition have KL exponent $\frac12$, so do models that satisfy the \color{blue} first-order error bound condition \color{black} together with a mild assumption on the separation of stationary values; see the recent work \cite{Cui_Ding_Zhao,LiPong17,ZhSo2017} for concrete examples. This sets the stage for developing calculus rules for KL exponent in \cite{LiPong17} to deduce the KL exponent of a function from functions with known KL exponents. For example, it was shown in \cite[Corollary~3.1]{LiPong17} that under mild conditions, if $f_i$ is a KL function with exponent $\alpha_i\in [0,1)$, $1\le i\le m$, then the KL exponent of $\min_{1\le i\le m}f_i$ is given by $\max_{1\le i\le m}\alpha_i$. This was then used in \cite[Section~5.2]{LiPong17} for showing that the least squares loss with smoothly clipped absolute deviation (SCAD) \cite{Fan97} or minimax concave penalty (MCP) regularization \cite{Zhang10} has KL exponent $\frac12$.

In this paper, we will further explore this line of research and study how KL exponent behaves under the {\em inf-projection} operation: this is a significant generalization of the operation of taking the minimum of finitely many functions. Precisely, let $\mathbb{X}$ and $\mathbb{Y}$ be two finite dimensional Hilbert spaces and let \color{blue} $F:\mathbb{X}\times \mathbb{Y}\to  \R \cup \{\infty\}$ \color{black} be a proper closed function,\footnote{We refer the readers to Section~\ref{sec2} for relevant definitions.} we call the function $f(x) := \inf_{y\in \mathbb{Y}}F(x,y)$ for $x\in \mathbb{X}$ an inf-projection of $F$. The name comes from the fact that the strict epigraph of $f$, defined as $\{(x,r)\in \mathbb{X}\times \R:\; f(x) < r\}$, is equal to the projection of the strict epigraph of $F$ onto $\mathbb{X}\times \R$. Functions represented in terms of inf-projections arise naturally in sensitivity analysis as {\em value functions}; see, for example, \cite[Chapter~3.2]{BoLe06}. Inf-projection also appears when representing functions as optimal values of linear programming problems, or more generally, semidefinite programming (SDP) problems; see \cite{HeNi10} for \color{blue} semidefinite-programming-representable (SDP-representable) \color{black} functions. It is known that inf-projection preserves nice properties of $F$ such as convexity \cite[Proposition~2.22(a)]{RocWets98}. In this paper, we show that, under mild assumptions, the KL exponent is also preserved under inf-projection. Based on this result and the ubiquity of inf-projection, we are then able to \revise{study} KL exponents of various important convex and nonconvex models that were out of reach in \revise{previous studies. These include convex models such as a large class of SDP-representable functions, and some functions with $C^2$-cone reducible structures, as well as nonconvex models such as difference-of-convex functions and Bregman envelopes. These models are discussed in details in Section~\ref{sec3.1} with the general strategy for deducing their KL exponents outlined.}

The rest of the paper is organized as follows. We present necessary notation and preliminary materials in Section~\ref{sec2}. The KL exponent under inf-projection is studied in Section~\ref{sec3}, \revise{and we outline how the results can be applied to deducing KL exponents of some optimization models in Section~\ref{sec3.1}. Section~\ref{sec4} is devoted to deriving KL exponents for various structured convex models, and in Section~\ref{sec5}, we study KL exponents for several nonconvex models.} Finally, some concluding remarks are given in Section~\ref{sec6}.

\section{Notation and preliminaries}\label{sec2}
In this paper, we use $\mathbb{X}$ and $\mathbb{Y}$ to denote two finite dimensional Hilbert spaces. We use $\langle \cdot,\cdot\rangle$ to denote the inner product of the underlying Hilbert space and use $\|\cdot\|$ to denote the associated norm. Moreover, for a linear map ${\cal A}:\mathbb{X}\to\mathbb{Y}$, we use ${\cal A}^*$ to denote its adjoint. Next, we let $\R$ denote the set of real numbers \revise{and let $\R^n$ denote the set of $n$-tuples of real numbers}. We also let $\R^{m\times n}$ denote the set of all $m\times n$ matrices. The (trace) inner product of two matrices $A$ and $B\in \R^{m\times n}$ is defined as $\langle A,B\rangle := {\rm tr}(A^TB)$, where ${\rm tr}$ denotes the trace of a square matrix. The Fr\"{o}benius norm of a matrix $A\in \R^{m\times n}$ is denoted by $\|A\|_F$, which is defined as $\|A\|_F:= \sqrt{{\rm tr}(A^TA)}$. Finally, the space of $n\times n$ symmetric matrices is denoted by ${\cal S}^n$, the cone of $n\times n$ positive semidefinite matrices is denoted by ${\cal S}^n_+$, and we write $X\succeq 0$ (resp., $X\succ 0$) to mean $X\in {\cal S}^n_+$ \color{blue}(resp., $X\in {\rm int\,}{\cal S}^n_+$, where ${\rm int\,}{\cal S}^n_+$ is the interior of ${\cal S}^n_+$). \color{black}

\color{blue}For a set ${\frak D}\subseteq \mathbb{X}$, we denote the distance from an $x\in \mathbb{X}$ to $\mathfrak{D}$ as $\dist(x,\mathfrak{D}):=\inf_{y\in \mathfrak{D}}\|x-y\|$.
The closure (resp., interior) of ${\frak D}$ is denoted by \revise{${\rm cl\,}{\frak D}$} (resp., \revise{${\rm int\,} \mathfrak{D}$}), and we use $B(x,r)$ to denote the closed ball centered at $x\in \mathbb{X}$ with radius $r > 0$, i.e., $B(x,r) := \{u\in \mathbb{X}:\; \|u-x\|\le r\}$. For a convex set ${\frak C}\subseteq \mathbb{X}$, we denote its relative interior by \revise{${\rm ri\,}{\frak C}$}, and use ${\frak C}^\circ$ to denote its polar, which is defined as
\[
{\frak C}^\circ := \{z\in \mathbb{X}:\; \langle x,z\rangle\le 1\ \mbox{for all}\  x\in {\frak C}\}.
\]
Finally, the indicator function of a nonempty set ${\frak D}\subseteq \mathbb{X}$ is denoted by $\delta_{\frak D}$, which equals zero in ${\frak D}$ and is infinity otherwise. We use $\sigma_{\frak D}$ to denote its support function, which is defined as $\sigma_{\frak D}(x) := \sup_{z\in {\frak D}}\langle x,z\rangle$ for $x\in \mathbb{X}$.

For a mapping $\Theta:\mathbb{X}\to \mathbb{Y}$ that is continuously differentiable on $\mathbb{X}$, we use $D\Theta(x)$ to denote the derivative mapping of $\Theta$ at $x\in \mathbb{X}$: this is the linear map defined by
\[
[D\Theta(x)]h := \lim_{t\to 0}\frac{\Theta(x+th) - \Theta(x)}{t}\ \ \ \mbox{for all $h\in \mathbb{X}$.}
\]
We denote the adjoint of the derivative mapping by $\nabla \Theta(x)$. This latter mapping is referred to as the gradient mapping of $\Theta$ at $x$.
Then, following \cite[Definition~3.1]{Shapiro}, we say that a closed set ${\frak D} \subseteq \mathbb{X}$ is $C^2$-cone reducible at $\bar w \in {\frak D}$ if there exist a closed convex pointed cone $K \subseteq \mathbb{Y}$, $\rho>0$ and a mapping $\Theta:\mathbb{X} \rightarrow \mathbb{Y}$ that maps $\bar w$ to 0 and is twice continuously differentiable in $B(\bar w, \rho)$ with $D \Theta(\bar w)$ being onto, such that
\[
{\frak D} \cap B(\bar w, \rho) = \{w: \Theta(w) \in K\} \cap B(\bar w, \rho).
\]
We say that the set ${\frak D}$ is $C^2$-cone reducible if, for all $\bar w \in {\frak D}$, ${\frak D}$ is $C^2$-cone reducible at $\bar w$. It is known that convex polyhedral sets, the positive semidefinite cone and the second-order cone are all $C^2$-cone reducible; see, for example, the discussion following \cite[Definition~3.1]{Shapiro}. Finally, following the discussion right after \cite[Definition~6]{Cui_Ding_Zhao}, we say that an extended-real-valued function is $C^2$-cone reducible if
its epigraph is a $C^2$-cone reducible set, where the epigraph of an extended-real-valued function $f: \mathbb{X}\to [-\infty,\infty]$ is defined as ${\rm epi\,}f:= \{(x,t)\in \mathbb{X}\times \R:\; f(x)\le t\}$.

An extended-real-valued function $f:\mathbb{X}\to [-\infty,\infty]$ \color{black} is said to be proper if its domain \color{blue}${\rm dom\,} f:=\{x\in \mathbb{X}:\; f(x)<\infty\}\neq \emptyset$ and it is never $-\infty$. \color{black} A proper function is closed if it is lower semicontinuous. For a proper function $f$, its regular subdifferential at \color{blue}$x\in {\rm dom\,} f$ \color{black} is defined in \cite[Definition~8.3]{RocWets98} by
\[
\hat{\partial} f(x):=\left\{\zeta\in \mathbb{X} :\; \liminf\limits_{z\rightarrow x,z\neq x}\frac{f(z)-f(x)-\langle \zeta,z-x\rangle }{\|z-x\|}\ge 0 \right\}.
\]
The subdifferential of $f$ at \color{blue}$x\in {\rm dom\,} f$ \color{black} (which is also called the limiting subdifferential) is defined in \cite[Definition~8.3]{RocWets98} by
\begin{equation*}
\partial f(x):=\left\{\zeta\in \mathbb{X} :\; \exists x^{k}\stackrel{f}{\rightarrow} x,\ \zeta^{k}\rightarrow \zeta \ {\rm with\ } \zeta^k \in \hat{\partial} f(x^k)\  {\rm for\  each}\  k \right\};
\end{equation*}
here, $x^{k}\stackrel{f}{\rightarrow} x$ means both $x^{k}\to x$ and $f(x^{k})\to f(x)$. Moreover, we set $\partial f(x) = \hat \partial f(x) = \emptyset$ for \color{blue}$x\notin {\rm dom\,} f$ \color{black} by convention, and write \color{blue}${\rm dom\,}\partial f:= \{x\in \mathbb{X}:\; \partial f(x)\neq \emptyset\}$. \color{black} It is known in \cite[Exercise~8.8]{RocWets98} that $\partial f(x) = \{\nabla f(x)\}$ if $f$ is continuously differentiable at $x$. Moreover, when $f$ is proper convex,
the limiting subdifferential reduces to the classical subdifferential in convex analysis; see \cite[Proposition~8.12]{RocWets98}. Finally, for a nonempty closed set ${\frak D}$, we define its normal cone at an $x\in {\frak D}$ by $N_{\frak D}(x) := \partial \delta_{\frak D}(x)$. If ${\frak D}$ is in addition convex, we define its tangent cone at $x\in {\frak D}$ by $T_{\frak D}(x) := [N_{\frak D}(x)]^\circ$.

For a proper convex function $f$, its Fenchel conjugate is
\[
f^*(u):=\sup_{x}\left\{\langle u,x\rangle -f(x)\right\};
\]
moreover, it is known that the following equivalence holds (see \cite[Theorem~23.5]{Roc70}):
\begin{equation}\label{Young}
u\in \partial f(x) \ \ \Longleftrightarrow \ \ f(x) + f^*(u) = \langle x,u\rangle \ \ \Longleftrightarrow \ \ f(x) + f^*(u) \le \langle x,u\rangle.
\end{equation}
For a proper closed convex function $f$, its asymptotic (or recession) function $f^{\infty}$ is defined by
$f^\infty(d):=\liminf_{t \rightarrow \infty, d' \rightarrow d} \frac{f(td')}{t}$; see \cite[Theorem~2.5.1]{AuTe2003}.
Finally, for a proper function $f$, we say that it is level-bounded if, for each $\alpha\in \R$, the set $\{x:\; f(x)\le \alpha\}$ is bounded.

\color{blue}For a proper function $F:\mathbb{X}\times \mathbb{Y}\to \R \cup \{\infty\}$, following \cite[definition~1.16]{RocWets98}, we say that $F$ is level-bounded in $y$ locally uniformly in $x$ if for each $\bar x\in\mathbb{X}$ and $\alpha\in\R$ there is a neighborhood $V $ of $\bar x$ such that the set $\{(x,y) \in \mathbb{X}\times \mathbb{Y}:\; x\in V {\rm \ and\ }F(x,y)\le\alpha\}$ is bounded. When a function  $F$ is level-bounded in $y$ locally uniformly in $x$, its inf-projection $f(x): = \inf_{y}F(x,y)$ has the following properties, which can be found in \cite{RocWets98}. We include the proof for the convenience of the readers.
\begin{lemma}\label{Yx}
   Let $F:\mathbb{X}\times\mathbb{Y}\to \mathbb{R}\cup\{\infty\}$ be a proper closed function and define $f(x):= \inf_{y\in \mathbb{Y}} F(x,y)$ and $Y(x):=\Argmin_{y\in \mathbb{Y}} F (x,y)$ for $x\in \mathbb{X}$. Suppose  $F$ is level-bounded in $y$ locally uniformly in $x$. Then the following statements hold:
  \begin{enumerate}[{\rm (i)}]
   \item The function $f$ is proper and closed, and the set $Y(x)$ is nonempty and compact for any $x\in\dom \partial f$.
    \item For any $x\in\dom \partial f$, it holds that
     \begin{align}\label{dfg}
    \partial f (x)\subseteq \bigcup_{y\in Y(x)}\{\xi\in \mathbb{X}\ :\ (\xi,0)\in \partial F (x,y)\}.
\end{align}
   \item For any $\bar x\in\dom \partial f$, it holds that
     \begin{equation}\label{limsup}
\limsup_{\revise{{\rm dom\,}\partial f} \ni x\overset{f}\rightarrow \bar x}Y(x)\subseteq Y(\bar x);
\end{equation}
     \item For any $\bar x\in\dom \partial f$ and any $\nu>0$, there exists $\epsilon>0$ such that
     \begin{equation*}%\label{distineq}
 {\rm dist}(y,Y(\bar x))\le\frac{\nu}2
\end{equation*}
whenever $y\in Y(x)$ with $x\in B(\bar x,\epsilon)\cap \revise{{\rm dom\,}\partial f}$ and $|f(x) - f(\bar x)| < \epsilon$.
   \end{enumerate}
\end{lemma}
\begin{proof}
Since $F $ is proper, closed and level-bounded in $y$ locally uniformly in $x$, we have from \cite[Theorem~1.17]{RocWets98} that $f$ is proper and closed, and $Y(x)$ is a nonempty compact set whenever $x\in \revise{{\rm dom\,}\partial f}$.
Applying \cite[Theorem~10.13]{RocWets98}, we conclude that \eqref{dfg} holds for any $x\in{\rm dom\,}\partial f $.

We now prove (iii) and (iv) respectively.
For (iii),  fix any $\bar x\in\dom\partial f$ and any $y^*$ satisfying $y^*\in \limsup_{{\rm dom\,}\partial f \ni x\overset{f}\rightarrow \bar x}Y(x)$ and recall from \cite[Section~5B]{RocWets98} that
\[
\limsup_{{\rm dom\,}\partial f \ni x\overset{f}\rightarrow \bar x}Y(x):=\left\{y:\;
\exists\, x^k\overset{f}\to \bar x,\ y^k\to y\ {\rm with}\ y^k \in Y(x^k)\ {\rm and}\ x^k\in \revise{{\rm dom\,}\partial f}\ \mbox{for each }k\right\}.
\]
So, there exist $x^k\overset{f}\to \bar x$ with $x^k\in \revise{{\rm dom\,}\partial f} $ and $y^k\to y^* $ such that $y^k\in Y(x^k)$ for all $k$. Then we have
 \[
 F (\bar x,y^*)\overset{\rm (a)}\le \liminf_k F (x^{k},y^{k}) \overset{\rm (b)}= \liminf_k f (x^k) \overset{\rm (c)}= f (\bar x),
 \]
 where (a) is due to the closedness of $F$, (b) holds because $y^k\in Y(x^k)$, and (c) holds because $x^k\overset{f}\to \bar x$.
 The above relation implies that $y^* \in Y({\bar x})$. This proves \eqref{limsup}.

Finally, for (iv), fix any $\bar x\in\dom\partial f$ and any $\nu > 0$. Since $F$ is level-bounded in $y$ locally uniformly in $x$, there exist $\tilde \epsilon > 0$ and a bounded set $D$ so that whenever $x\in B(\bar x,\tilde\epsilon)\cap \revise{{\rm dom\,} \partial f}$, we have $\{y : F (x,y)\le f (\bar x) + 1\}\subseteq D$. Thus, for any $x$ satisfying $x\in B(\bar x,\tilde\epsilon)\cap \revise{{\rm dom\,}\partial f}$ and $f(x) < f(\bar x) + 1$, we obtain
\begin{equation}\label{YinD}
Y(x) = \{y : F (x,y)\le f (x)\}\subseteq \{y : F (x,y)\le f (\bar x) + 1\}\subseteq D.
\end{equation}
Since \eqref{limsup} holds, by picking $\eta > 0$ so that $D\subseteq B(0,\eta)$ and following the proof of \cite[\reviser{Proposition~5.12(a)}]{RocWets98}, we see that for this $\eta$, there exists $\epsilon\in (0,\min\{\tilde\epsilon,1\})$ such that
\begin{equation*}
Y(x) = Y(x)\cap D \subseteq Y(x)\cap B(0,\eta)\subseteq Y(\bar x) + B(0,\nu/2),
\end{equation*}
whenever $x\in B(\bar x,\epsilon)\cap \revise{{\rm dom\,}\partial f}$ and $|f(x) - f(\bar x)| < \epsilon$,
where the first equality follows from \eqref{YinD} and the facts that $\epsilon < \tilde \epsilon$ and $\epsilon  < 1$.
This further implies that
\begin{equation*}
 {\rm dist}(y,Y(\bar x))\le\frac{\nu}2.
\end{equation*}
for any $y\in Y(x)$ with $x\in B(\bar x,\epsilon)\cap \revise{{\rm dom\,}\partial f}$ and $|f(x) - f(\bar x)| < \epsilon$.
\end{proof}

\color{black}

We next recall the Kurdyka-{\L}ojasiewicz (KL) property and the notion of KL exponent; see \cite{Loja63,Kur98,AtBo09,AtBoReSo10,AtBoSv13,LiPong17}. This property has been used extensively in analyzing convergence of first-order methods; see, for example, \cite{AtBo09,AtBoReSo10,AtBoSv13,BoSaTe14,WCP18}.

\begin{definition}[{{\bf Kurdyka-{\L}ojasiewicz property and exponent}}]\label{DefKL}
  We say that a proper closed function \color{blue}$h:\mathbb{X}\to\mathbb{ R}\cup\{\infty\}$ \color{black} satisfies the Kurdyka-{\L}ojasiewicz (KL) property at \color{blue}$\hat x\in {\rm dom\,} \partial h$ \color{black} if there are $a\in (0,\infty]$, a neighborhood $V$ of $\hat{x}$ and a continuous concave function $\varphi: [0,a)\rightarrow [0,\infty) $ with $\varphi(0)=0$ such that
  \begin{enumerate}[{\rm (i)}]
    \item $\varphi$ is continuously differentiable on $(0, a)$ with $\varphi'>0$ on $(0,a)$;
    \item For any $x\in V$ with $h(\hat{x})<h(x)<h(\hat{x})+a$, it holds that
     \begin{equation}\label{phichoice}
     \varphi'(h(x)-h(\hat{x})){\rm dist}(0,\partial h(x))\ge 1.
     \end{equation}
  \end{enumerate}
  If $h$ satisfies the KL property at \color{blue}$\hat x\in {\rm dom\,}\partial h$ \color{black} and the $\varphi(s)$ in \eqref{phichoice} can be chosen as $\bar c \, s^{1-\alpha}$ for some $\bar c>0$ and $\alpha\in[0,1)$, then we say that $h$ satisfies the KL property at $\hat x$ with exponent $\alpha$.

  A proper closed function $h$ satisfying the KL property at every point in \color{blue}${\rm dom\,} \partial h$ \color{black} is said to be a KL function, and a proper closed function $h$ satisfying the KL property with exponent $\alpha\in [0,1)$ at every point in \color{blue}${\rm dom\,} \partial h$ \color{black} is said to be a KL function with exponent $\alpha$.
\end{definition}

KL functions is a broad class of functions which arise naturally in many applications. For instance, it is known that proper closed semi-algebraic functions are KL functions with exponent $\alpha\in [0,1)$; see, for example, \cite{AtBoReSo10}. KL property is a key ingredient in many contemporary convergence analysis for first-order methods, and the KL exponent plays an important role in identifying {\em local convergence rate}; see, for example, \cite[Theorem~2]{AtBo09}, \cite[Theorem~3.4]{FranGarPey15} and \cite[Theorem~3]{LiPong16}. In this paper, we will study how the KL exponent behaves under inf-projection, and use the rules developed to compute the KL exponents of various functions and to derive new calculus rules for KL exponent.

Before ending this section, we present two auxiliary lemmas. The first lemma concerns the uniformized KL property. It is a specialization of \cite[Lemma~6]{BoSaTe14} and explicitly involves the KL exponent.

\begin{lemma}[{{\bf Uniformized KL property with exponent}}]\label{UKL}
Suppose that \color{blue}$h:\mathbb{X}\to\mathbb{R}\cup\{\infty\}$ \color{black} is a proper closed function and let $\Omega$ be a \revise{nonempty} compact set \revise{with $\Omega\subseteq {\rm dom}\,\partial h$}. If $h$ \color{blue} takes a constant value \color{black} on $\Omega$ and satisfies the KL property at each point of $\Omega$ with exponent $\alpha$, then there exist $\epsilon, a ,c>0$  such that
\[
{\rm dist}\left(0,\partial h(x) \right)\ge c\left(h(x)- h(\bar{x})\right)^{\alpha}
\]
for any $\bar x\in \Omega$ and any $x$ satisfying $ h(\bar{x})<h(x)< h(\bar{x})+a$ and ${\rm dist}(x,\Omega)<\epsilon$.
\end{lemma}
\begin{proof}
  Replace the $\varphi_i(t)$ in the proof of \cite[Lemma~6]{BoSaTe14} by $c_i t^{1-\alpha}$ for some $c_i > 0$. The desired conclusion can then be proved analogously as in \cite[Lemma~6]{BoSaTe14}.
\end{proof}

The next lemma is a direct consequence of results in \cite{Sturm00}; see \cite[Theorem~3.3]{Sturm00} and the discussion following \cite[Eq.~(1.4)]{Sturm00} concerning the degree of singularity \reviser{for semidefinite feasibility system}.
\begin{lemma}[{{\bf Error bound for standard SDP problems under strict complementarity}}]\label{KLG}
  Let $C\in\mathcal{S}^d$, ${\cal A}:\mathcal{S}^d\to \R^m$ be a linear map, \color{blue}$b\in {\rm Range\,}({\cal A})$ \color{black} and define the function \color{blue}$G:\mathcal{S}^d\to\mathbb{R}\cup\{\infty\}$ \color{black} by
  \[
  G(X):= \langle C,X\rangle  +\delta_{\mathfrak{L}}(X),
  \]
  where $\mathfrak{L}={\cal A}^{-1}\{b\}\cap \mathcal{S}^d_+$. Suppose that ${\cal A}^{-1}\{b\}\cap\color{blue}{\rm int\,}\mathcal{S}^d_+\color{black}\neq\emptyset$ and there exists $\bar X\in\mathfrak{L}$ satisfying $0\in\color{blue}{\rm ri\,}\partial G(\bar X)\color{black}$. Then for any bounded neighborhood $\mathfrak{U}$ of $\bar X$, there exists $c>0$ such that for any $X\in\mathfrak{U}\cap\mathfrak{L}$,
\[
\dist(X,\Argmin G)\le c \left(G(X)- G(\bar X)\right)^{\frac12}.
\]
\end{lemma}
\begin{proof}
  Observe that
  \begin{align}\label{riri}
  \begin{aligned}
  0\in \color{blue}{\rm ri\,}\partial G(\bar X)\color{black}\overset{\rm (a)}=C + \color{blue}{\rm ri\,}N_{\mathfrak{L}}(\bar X)\color{black}
  &\overset{\rm (b)}=C + {\rm ri\,}\left(N_{{\cal A}^{-1}\{b\}}(\bar X)+N_{\mathcal{S}^d_+}(\bar X)\right)\\
  &\overset{\rm (c)}=C + \color{blue}{\rm ri\,}N_{{\cal A}^{-1}\{b\}}(\bar X)+ {\rm ri\,}N_{\mathcal{S}^d_+}(\bar X),\color{black}
  \end{aligned}
  \end{align}
  where (a) follows from \cite[Exercise~8.8]{RocWets98}, (b) follows from \cite[Theorem~23.8]{Roc70} and the assumption ${\cal A}^{-1}\{b\}\cap\color{blue}{\rm int\,}\mathcal{S}^d_+\color{black}\neq\emptyset$, and (c) follows from \cite[Corollary~6.6.2]{Roc70}.
  Since $N_{{\cal A}^{-1}\{b\}}(\bar X)=\color{blue}{\rm Range\,}({\cal A}^*)\color{black}$, we deduce further from \eqref{riri} the existence of $\bar y$ satisfying
 \begin{align}\label{strict}
   \color{blue}{\cal A}^*\bar y-C \in {\rm ri\,}N_{\mathcal{S}^d_+}(\bar X).\color{black}
  \end{align}
Next, since $0\in\partial G(\bar X)$, we have that $\bar X\in\Argmin{G}$ and thus
\[
\Argmin G=\left\{W:\;{\cal A}W=b\right\}\cap \left\{W:\;\langle C,W\rangle=\inf G\right\}\cap \mathcal{S}^d_+\neq \emptyset.
\]
\color{blue}This together with  \eqref{strict} implies that  the singularity degree of the semidefinite feasibility system $\left(\left\{W:\;{\cal A}W=b\right\}\cap \left\{W:\;\langle C,W\rangle=\inf G\right\},\mathcal{S}^d_+\right)$ is one. Combining this with \cite[Theorem~2.3]{DrLiWo17}, \color{black}%Sturm's error bound \cite[Theorem~3.3]{Sturm00} and \cite[Lemma~3.6]{Sturm00},}
  we conclude that for any bounded neighborhood $\mathfrak{U}$ of $\bar X$, there exists \color{blue}$c_1>0$ \color{black} such that for any $X\in\mathfrak{U}\cap\mathfrak{L}$,
\begin{align*}
\dist(X,\Argmin{G})&\le \color{blue}c_1\color{black}\ \sqrt{{\rm dist}\left(X,\left\{W:\;{\cal A}W=b\right\}\cap\left\{W:\langle C,W\rangle =\inf G\;\right\}\right)}\\
&\le \color{blue}c\color{black} \left(\langle C,X\rangle -\inf G\right)^{\frac12}= \color{blue}c\color{black} \left(G(X)- G(\bar X)\right)^{\frac12},
\end{align*}
where the second inequality holds for some \color{blue}$c > 0$ \color{black} thanks to the Hoffman error bound \color{blue}\cite[Lemma~3.2.3]{FaPa03}. \color{black} This completes the proof.
\end{proof}

\begin{remark}
 \color{blue}In the above lemma, the Slater's  condition ${\cal A}^{-1}\{b\}\cap{\rm int\,}\mathcal{S}^d_+\neq\emptyset$  together with the relative interior (ri) condition \reviser{$0\in {\rm ri\,}\partial G(\bar X)$} implies that \eqref{strict}  holds. \reviser{The condition  \eqref{strict}} is widely used in the  SDP literature and is often referred to as the strict complementarity condition; see  \cite{ShSc00,TuWo12,Pa00} for detailed discussions. \reviser{In particular, it is known that if strict complementarity condition \eqref{strict} holds , then the singular degree of the associated semidefinite feasibility system is one (see \cite[Proposition 7]{Bruno18} or the discussion following \cite[Eq.~(1.4)]{Sturm00})}.

 As we shall see in Section \ref{sec4}, \reviser{this strict complementarity condition}  is crucial for  deriving a KL exponent of $\frac12$ for some SDP representable functions. \color{black}
\end{remark}
\section{KL exponent via inf-projection}\label{sec3}

In this section, we study how the KL exponent behaves under inf-projection. Specifically, given a proper closed function \color{blue}$F:\mathbb{X}\times \mathbb{Y}\to \mathbb{R}\cup\{\infty\}$ \color{black} with known KL exponent, we would like to deduce the KL exponent of $\inf_{y\in \mathbb{Y}} F (\cdot,y)$ under suitable assumptions.
\begin{theorem}[{{\bf KL exponent via inf-projection}}]\label{general}
 Let \color{blue}$F:\mathbb{X}\times\mathbb{Y}\to \mathbb{R}\cup\{\infty\}$ \color{black} be a proper closed function and define $f(x):= \inf_{y\in \mathbb{Y}} F(x,y)$ and $Y(x):=\Argmin_{y\in \mathbb{Y}} F (x,y)$ for $x\in \mathbb{X}$. \color{blue}Suppose that the function  $F$ is level-bounded in $y$ locally uniformly in $x$. \color{black} Let $\alpha \in [0,1)$ and $\bar x\in \color{blue}{\rm dom\,}\partial f\color{black}$.\footnote{\color{blue} Here, $f$ is a proper closed function, thanks to Lemma \ref{Yx}(i).\color{black}} Suppose in addition the following conditions hold:
 \begin{enumerate}[{\rm (i)}]
   \item It holds that $\partial F (\bar x,\bar y) \neq \emptyset$ for all $\bar y\in Y(\bar x)$.
   \item The function $F$ satisfies the KL property with exponent $\alpha$ at every point in $\{\bar x\}\times Y(\bar x)$.
 \end{enumerate}
 Then $f$ satisfies the KL property at $\bar{x}$ with exponent $\alpha$.
\end{theorem}
\begin{proof}
\color{blue}Using the nonemptiness and compactness of $Y(\bar x)$ given by Lemma \ref{Yx}(i), and the facts that \color{black} $F(x,y)\equiv f (\bar x)$ on $\Omega := \{\bar x\}\times Y(\bar x)\revise{ \subseteq {\rm dom}\,\partial F}$ and  $F$ satisfies the KL property with exponent $\alpha$ at every point in $\Omega$, we deduce from Lemma \ref{UKL} that there exist $\nu, a ,c>0$  such that
\begin{align}\label{UniKL}
\dist \left(0,\partial F (x,y) \right)\ge c\left(F (x,y)- f (\bar{x})\right)^{\alpha}
\end{align}
for any $(x,y)$ satisfying
\begin{align}\label{ass22}
f (\bar x)<F (x,y)<f (\bar x)+a\ \ {\rm and}\ \
{\rm dist}((x,y),\Omega)<\nu.
\end{align}
\reviser{By decreasing $a$ if necessary,} without loss of generality, we may assume $a \in (0,1)$.

\color{blue}Next, using Lemma~\ref{Yx}(iv), we see that there exists $\epsilon\in(0,\min\{\nu/2,a\})$ such that
\begin{equation*}
 {\rm dist}(y,Y(\bar x))\le\frac{\nu}2
\end{equation*}
whenever $y\in Y(x)$ with $x\in B(\bar x,\epsilon)\cap {\rm dom\,}\partial f$ and $f (\bar x) < f (x) < f (\bar x) + \epsilon$.
Hence, for any $x \in B(\bar x,\epsilon)\cap \revise{{\rm dom\,}\partial f} $ with $f (\bar x) < f (x) < f (\bar x) + \epsilon$ and any $y\in Y(x)$, we have
\[
{\rm dist}((x,y),\Omega)\le \|x - \bar x\| + {\rm dist}(y,Y(\bar x)) \le \epsilon + \frac{\nu}2 < \nu,
\]
where the last inequality follows from the choice of $\epsilon$.
The above relation  together with the fact that  $\epsilon < a$ shows that
the relation \eqref{ass22} holds for any such $x$ and any $y\in Y(x)$.
Thus, \color{black} using \eqref{UniKL} we conclude that \color{blue}for any such $x$ and any $y\in Y(x)$, \color{black}
  \begin{align*}
        &\dist (0,\partial f (x)) = \dist \left(0,\begin{bmatrix}
                                        \partial f (x)\\
                                        0
                                       \end{bmatrix}\right)\ge \inf_{y\in Y(x)} \dist \left(0,\partial F (x,y)\right)\\
                        &\ge \inf_{y\in Y(x)} c\left(F (x,y)- f (\bar x)\right)^{\alpha} = c\left(f (x)- f (\bar{x})\right)^{\alpha},
  \end{align*}
  where the first inequality follows from \eqref{dfg} and the last equality follows from the definition of $Y(x)$. This completes the proof.
\end{proof}

Theorem~\ref{general} can be viewed as a generalization of \cite[Theorem~3.1]{LiPong17}, which \color{blue}studies \color{black} the KL exponent of the minimum of finitely many proper closed functions with known KL exponents. Indeed, let $f_i$, $1\le i\le m$, be \reviser{proper closed functions}. If we let $\mathbb{Y} = \R$ and define \color{blue}$F:\mathbb{X}\times \R\to \mathbb{R}\cup\{\infty\}$ \color{black} by
\begin{equation}\label{F_finite}
F(x,y) = \begin{cases}
  f_y(x) & {\rm if}\ y = 1,2,\ldots,m,\\
  \infty & {\rm otherwise},
\end{cases}
\end{equation}
then it is not hard to see that this $F$ is a proper closed function, and $\inf_{y\in \R} F(x,y) = \min_{1\le i\le m}f_i(x)$ for all $x\in \mathbb{X}$. Moreover, one can check directly from the definition that
\begin{equation}\label{dF_finite}
\partial F(x,y) = \begin{cases}
  \partial f_y(x)\times \R & {\rm if}\ y = 1,2,\ldots,m,\\
  \emptyset & {\rm otherwise}.
\end{cases}
\end{equation}
Thus, we have the following immediate corollary of Theorem~\ref{general}, which is a slight generalization of \cite[Theorem~3.1]{LiPong17} by dropping the continuity assumption on $\min_{1\le i\le m}f_i$.
\begin{corollary}[{{\bf KL exponent for minimum of finitely many functions}}]\label{cor:finitef}
  Let $f_i$, $1\le i\le m$, be proper closed functions, and define $f:= \min_{1\le i\le m}f_i$. Let $\bar x\in \color{blue}{\rm dom\,}\partial f\color{black}\cap \bigcap_{i\in I(\bar x)}\color{blue}{\rm dom\,}\partial f_i\color{black}$, where $I(\bar x):= \{i:f_i(\bar x) = f(\bar x)\}$. Suppose that for each $i\in I(\bar x)$, the function $f_i$ satisfies the KL property at $\bar x$ with exponent $\alpha_i\in [0,1)$. Then $f$ satisfies the KL property at $\bar x$ with exponent $\alpha = \max\{\alpha_i:i\in I(\bar x)\}$.
\end{corollary}
\begin{proof}
  Define $F$ as in \eqref{F_finite}. Then $F$ is proper and closed, and $f(x) = \inf_{y\in \R} F(x,y)$. Moreover, $I(x) = Y(x) := \Argmin_{y\in \R} F(x,y)$. It is clear that this $F$ is level-bounded in $y$ locally uniformly in $x$. Moreover, in view of \eqref{dF_finite} and the assumption that $\bar x\in \bigcap_{i\in I(\bar x)}\color{blue}{\rm dom\,}\partial f_i\color{black}$, we see that $\partial F(\bar x,\bar y)\neq \emptyset$ whenever $\bar y\in Y(\bar x)$. Finally, it is routine to show that $F$ satisfies the KL property with exponent $\alpha_i$ at $(\bar x,i)$ for $i \in I(\bar x)$. Thus, $F$ satisfies the KL property with exponent $\alpha = \max\{\alpha_i:i\in I(\bar x)\}$ on $\{\bar x\}\times I(\bar x)$. The desired conclusion now follows from Theorem~\ref{general}.
\end{proof}

The next corollary can be proved similarly as \cite[Corollary~3.1]{LiPong17} by using Corollary~\ref{cor:finitef} in place of \cite[Theorem~3.1]{LiPong17}.
\begin{corollary}%\label{cor:finitef2}
  Let $f_i$, $1\le i\le m$, be proper closed functions with $\color{blue}{\rm dom\,}f_i={\rm dom\,}\partial f_i\color{black}$ for all $i$, and define $f:= \min_{1\le i\le m}f_i$. Suppose that for each $i$, the function $f_i$ is a KL function with exponent $\alpha_i\in [0,1)$. Then $f$ is a KL function with exponent $\alpha = \max\{\alpha_i:1\le i\le m\}$.
\end{corollary}

Finally, we show in the next corollary that one can relax some conditions of Theorem~\ref{general} when $F$ is in addition convex.\color{blue}

 \begin{corollary}[{{\bf KL exponent via inf-projections under convexity}}]\label{convex_general}
 Let $F:\mathbb{X}\times\mathbb{Y}\to \mathbb{R}\cup\{\infty\}$ be a proper closed convex function and define $f(x):= \inf_{y\in \mathbb{Y}} F(x,y)$ and $Y(x):=\Argmin_{y\in \mathbb{Y}} F (x,y)$ for $x\in \mathbb{X}$. Suppose there exists $\bar u$ such that $f(\bar u)\in \R$ and $Y(\bar u)$ is nonempty and compact. Then the following statements hold:
 \begin{enumerate}[{\rm (i)}]
   \item The function $f$ is proper and closed, and $Y(x)$ is nonempty and compact for any $x\in {\rm dom}\,\partial f$.
   \item It holds that $\partial F(x,y)\neq \emptyset$ for all $x\in {\rm dom\,}\partial f$ and $ y\in Y( x)$.
   \item If $\bar x\in \revise{{\rm dom\,}\partial f}$, $\alpha \in [0,1)$ and   the function $F$ satisfies the KL property with exponent $\alpha$ at every point in $\{\bar x\}\times Y(\bar x)$, then $f$ satisfies the KL property at $\bar{x}$ with exponent $\alpha$.
 \end{enumerate}
 \end{corollary}
 \begin{proof}
   For (i), we first show that $F$ is level-bounded in $y$ locally uniformly in $x$. \color{black} Suppose to the contrary that there exist $x_0\in \mathbb{X}$ and $\beta\in \R$ so that ${\frak C} := \{(x,y): x\in B(x_0,1)\ {\rm and} \ F(x,y)\le \beta\}$ is unbounded. Then
   there exists $\{(x^k,y^k)\}\subset {\frak C}$ with $\|y^k\|\to \infty$. By passing to a subsequence if necessary, we may assume $\lim_{k\to\infty}\frac{y^k}{\|y^k\|} = d$ for some $d$ with $\|d\|=1$. Since $F(x^k,y^k)\le \beta$ and $\{x^k\}\subset B(x_0,1)$ is bounded, we have
   \[
   F^\infty(0,d) \le \liminf_{k\to \infty}\frac{F(x^k,y^k)}{\|(x^k,y^k)\|} \le  \liminf_{k\to \infty}\frac{\beta}{\|(x^k,y^k)\|} = 0,
   \]
   where $F^\infty$ is the asymptotic function of $F$ and the first inequality follows from \cite[Theorem~2.5.1]{AuTe2003}. This together with the convexity of $F$ and \cite[Proposition~2.5.2]{AuTe2003} shows that
   \[
   F(x,y+td)\le F(x,y)\ \ \mbox{for all}\ t > 0 \ \mbox{and for all}\ (x,y)\in \color{blue}{\rm dom\,}F. \color{black}
   \]
   \color{blue}Since $Y(\bar u)\neq \emptyset$ and $f(\bar u) \in \R$, we have $\{\bar u\}\times Y(\bar u)\subseteq {\rm dom}\,F$. Hence, we can take $\bar v\in Y(\bar u)$ and set $x = \bar u$ and $y = \bar v$ in the above display to conclude that
   $F(\bar u,\bar v + td)\le F(\bar u,\bar v)$ for all $t > 0$. This further implies that $\bar v + t d \in Y(\bar u)$ for all $t > 0$, which contradicts the compactness of $Y(\bar u)$. Thus, for any $x_0\in \mathbb{X}$ and $\beta\in \R$, the set $\{(x,y): x\in B(x_0,1)\ {\rm and} \ F(x,y)\le \beta\}$ is bounded. Using Lemma~\ref{Yx}(i), we see that (i) holds.

   Next, we prove (ii). To this end, fix any $u\in{\rm dom\,}\partial f $ and $ v\in Y(u)$. Note that the function $f$ is convex as inf-projection of the convex function $F$; see \cite[Proposition~2.22(a)]{RocWets98}. Now, for the proper convex function $f$, we have from the definition that $f^*(w) = \sup_x\{\langle w,x\rangle - f(x)\}= \sup_{x,y}\{\langle w,x\rangle - F(x,y)\} = F^*(w,0)$ for any $w\in \mathbb{X}$. Taking a $\bar w\in \partial f(u)$ and using \eqref{Young}, we see further that for any $v \in Y(u)$,
   \[
   F(u,v) + F^*(\bar w,0) = f(u) + f^*(\bar w) = \langle u,\bar w\rangle,
   \]
   where the equality $F(u,v) = f(u)$ holds because $v\in Y(u)$. In view of \eqref{Young}, the above relation further implies that $(\bar w,0)\in \partial F(u,v)$. This proves (ii).

 Now, suppose in addition that  $ \bar x\in {\rm dom\,}\partial f$, $\alpha \in [0,1)$ and the function $F$ satisfies the KL property with exponent $\alpha$ at every point in $\{\bar x\}\times Y(\bar x)$. Recall that we have shown that $F$ is level-bounded in $y$ locally uniformly in $x$ in the proof of item (i) and we have $\{\bar x\}\times Y(\bar x)\subseteq {\rm dom}\,\partial F$ from item (ii). The conclusion (iii) now follows by applying Theorem~\ref{general}.
\end{proof}

\begin{remark}\label{newremark3.1}
In addition to the inf-projection, another closely related operation, which appears frequently in optimization, would be taking the supremum over a family of functions. However, we would like to point out that, as opposed to the inf-projection, the supremum operation may not preserve KL exponents. For example, consider $F:\R^2\to \R$ defined by $F = \max\{f_1,f_2\}$ with $f_1(x) = x_1^2$ and $f_2(x) = (x_1+1)^2 + x_2^2 - 1$. Clearly, $f_1$ and $f_2$ are both quadratic and are KL functions with exponent $\frac12$. On the other hand, it was shown in \cite[Page 1617]{JiangLi} that $F$ has an optimal solution at $(0,0)$ and the KL exponent of $F$ at $(0,0)$ is $\frac34$ and cannot be $\frac12$. It would be of interest to see, under what additional conditions, the supremum operation can preserve the KL exponents. This could be one interesting future research direction.
\end{remark}

\subsection{Optimization models that can be written as inf-projections}\label{sec3.1}
\color{blue}Inf-projection is ubiquitous in optimization.  In this section, we present some commonly encountered models that can be written as inf-projections. This includes a large class of semidefinite-programming-representable (SDP-representable) functions, rank constrained least squares problems, and Bregman envelopes. These are important convex and nonconvex models whose explicit KL exponents were out of reach in previous studies. In Sections \ref{sec4} and \ref{sec5}, we will study their KL exponents based on their inf-projection representations, Theorem \ref{general} and Corollary \ref{convex_general}.
\subsubsection{Convex models that can be written as inf-projections}\label{Coninf}
\begin{description}
\item[(i) SDP-representable functions] Following \cite[Eq.~(1.3)]{HeNi10}, we say that a function $f:\R^n\to \mathbb{R}\cup\{\infty\}$, is semidefinite-programming-representable (SDP-representable) if its epigraph can be expressed as the feasible region of some SDP problems, i.e.,
\begin{align}\label{epiSDP}
\revise{{\rm epi\,}f} =\left\{(x,t)\in\R^n\times\R:\;\exists u\in \R^N\, {\rm\ s.t.\ }A_{00} + A_0t + \sum_{i=1}^{n}A_ix_i + \sum_{j=1}^NB_j u_j\succeq 0\right\}
\end{align}
for some $\{A_{00},A_0, A_1,\dots,A_n,B_1,\dots,B_N\}\subset\mathcal{S}^d$, \reviser{$d \ge 1$} and $N \ge 1$. These functions arise in various applications and include important examples such as least squares loss functions, $\ell_1$ norm, and nuclear norm, etc; see, for example, \cite[Section~4.2]{BenTalNemi01} for more discussions.  Using the symmetric matrices in \eqref{epiSDP}, we define a linear map
$\mathcal{A}:\mathcal{S}^d\to\R^{n+N+1}$ as
\begin{align}\label{mathcalA}
\mathcal{A}(W):= \left[\langle A_1,W\rangle \ \cdots\ \langle  A_n,W\rangle \ \langle B_1,W\rangle \ \cdots\ \langle  B_N,W\rangle \ \langle A_0,W\rangle \right]^T.
\end{align}
Then it is routine to show that $\mathcal{A}^*:\R^{n+N+1}\to \mathcal{S}^d$ is given by
$\mathcal{A}^*(x,u,t)= A_0t + \sum_{i=1}^{n}A_ix_i + \sum_{j=1}^NB_j u_j$ for $(x,u,t)\in \R^n \times \R^N \times \R$.
Now, if we define
\begin{align}\label{F}
F(x,u,t):=t + \delta_{\mathfrak{D}}(x,u,t)\ \ {\rm with}\ \
\mathfrak{D}=\left\{(x,u,t):\; A_{00} + \mathcal{A}^*(x,u,t)\succeq 0\right\},
\end{align}
then it holds that $f(x) = \inf_{u,t}F(x,u,t)$ for all $x\in \R^n$. We will show in Theorem \ref{SDP}  (using Corollary~\ref{convex_general}) that a proper closed SDP-representable  function   has KL property with exponent $\frac12$ at points satisfying  suitable assumptions on the SDP representation of $F$ in \eqref{F}.

\item[(ii) Sum of LMI-representable functions] We say that a function $h:\R^n\to\mathbb{R}\cup\{\infty\}$, is LMI-representable (see \cite[Eq.~(1.1)]{HeNi10}) if there exist symmetric matrices  $A_{00}$, $A_j$, $j=0,\ldots,n$, such that
  \[
  \revise{{\rm epi\,}h}=\left\{(x,t)\in \R^n\times \R:\;A_{00} +\sum_{j=1}^nA_jx_j + A_0t\succeq 0\right\}.
  \]
It is clear that   LMI-representable functions form a special class of  SDP-representable functions.
Many commonly used functions are LMI-representable such as the least squares loss function, the $\ell_1$, $\ell_2$, $\ell_\infty$ norm functions, the indicator functions of their corresponding norm balls, and the indicator function of the matrix operator norm ball, etc.

Let $f = \sum_{i=1}^mf_i$ be the sum of $m$ proper closed LMI-representable functions.
In Theorem~\ref{corSDP}, we show that    $f$ has KL property with exponent  $\frac12$ at points under suitable assumptions. Different from  Theorem~\ref{SDP}, which imposes the ``strict complementarity condition" on the corresponding $F$ in  \eqref{F},   Theorem~\ref{corSDP}   {\em directly} imposes such kind of condition  on the original function $f$. Explicit optimization models which can be written as sum of LMI-representable functions include (non-overlapping) group Lasso and group fused Lasso, and are discussed in Example \ref{appLMI}.

\item[(iii) Sum of LMI-representable functions and the nuclear norm] The nuclear norm  has been used for inducing low rank of solutions in various applications; see, for example, \cite{ReFaPa10} for more discussions.  Noticing that the nuclear norm is a  special  SDP-representable function, we further consider the sum of LMI-representable functions and the nuclear norm:
  \begin{align}\label{fnuclear}
  f(X):= \sum_{k=1}^pf_k(X) + \|X\|_*,
  \end{align}
  where $X\in\R^{m\times n}$, $\|X\|_*$ denotes the nuclear norm of $X$ (the sum of all singular values of $X$) and each \revise{$f_k:\R^{m\times n}\to\mathbb{R}\cup\{\infty\}$} is a proper closed LMI-representable function.  Define a function \revise{$F:{\cal S}^{n+m}\to \mathbb{R}\cup\{\infty\}$} by
  \begin{equation}\label{defF}
  F(Z):=\sum_{k=1}^pf_k(X)+\frac12({\rm tr}(U)+{\rm tr }(V))+ \delta_{\mathcal{S}_+^{m+n}}(Z);
  \end{equation}
  here, we partition the matrix variable $Z\in {\cal S}^{n+m}$ as follows:
  \begin{align}\label{Z}
  Z = \begin{bmatrix}
    U & X\\ X^T& V
  \end{bmatrix},
  \end{align}
  where $U\in {\cal S}^m$, $V\in {\cal S}^n$ and $X\in \R^{m\times n}$.
  Then one can show that $  f(X) = \inf_{U,V}F(Z)$; see \eqref{nucnorm} below. In Theorem \ref{nuclearLMI}, we will show that  $f$ in   \eqref{fnuclear} satisfies KL property with exponent $\frac12$ at points $\bar X$ such that  $0\in{\rm ri\,}\partial f(\bar X)$, under mild conditions. Explicit optimization models of the form \eqref{fnuclear} are introduced in Remark \ref{appLMInu}.
  \item[(iv) Convex models with $C^2$-cone reducible structure ] SDP representable functions are all semi-algebraic. As an attempt to go beyond semi-algebraicity,  we  analyze functions involving $C^2$-cone reducible structure. Specifically, we consider the following function $f:\mathbb{X}\to \R\cup\{\infty\}$:
\begin{align}\label{C2h}
f(x):=\ell(\mathcal{\mathcal{A}} x)+\langle v,x \rangle+ \gamma(x),
\end{align}
where $\gamma$ is a closed gauge\footnote{\revise{A gauge is a nonnegative positively homogeneous convex function that vanishes at the origin.}} whose polar gauge\footnote{\color{blue}See \cite[Proposition~2.1(iii)]{FriMaPong14}.\color{black}} is $C^2$-cone reducible, the function $\ell:\mathbb{Y}\to \R$ is strongly convex on any compact convex set
and has locally Lipschitz gradient, $\mathcal{A}:\mathbb{X} \to \mathbb{Y}$ is a linear map, and $v\in \mathbb{X}$.

Notice that $f(x) = \inf_t F(x,t)$, where
\begin{align}\label{C2F}
F(x,t):=\ell(\mathcal{A}x)+\langle v, x \rangle +  t +\delta_{\mathfrak{D}}(x,t),
\end{align}
with $\mathfrak{D}=\{(x,t)\in \mathbb{X}\times \R:  \gamma(x) \le t\}$. In Section \ref{sec4.2}, we will deduce that $f$ in \eqref{C2h} has KL property with exponent $\frac12$ at points satisfying assumptions involving relative interior of some subdifferential sets; see Corollary~\ref{C2gauge}. Optimization models in the form of \eqref{C2h} are presented in Example~\ref{ExC2}.
\end{description}

 \subsubsection{Nonconvex optimization models that can be written as inf-projections}\label{nonconinf}
 \begin{description}
   \item[(i) Difference-of-convex functions]We consider difference-of-convex (DC) functions of the following form:
  \begin{align}\label{DC}
    f (x) = P_1(x) - P_2({\cal A}x),
  \end{align}
  where $P_1:\mathbb{X}\to \mathbb{R}\cup\{\infty\}$ is a proper closed convex function, $P_2:\mathbb{Y}\to \R$ is a continuous convex function and ${\cal A}:\mathbb{X}\to \mathbb{Y}$ is a linear map. These functions arise in many contemporary applications including compressed sensing; see, for example, \cite{APX17,Tuy16,WCP18,YLHX15} and references therein. In the literature, the following function is a typically used majorant for designing and analyzing algorithms for minimizing DC functions. It is obtained from \eqref{DC} by majorizing the concave function $-P_2$ using the Fenchel conjugate $P_2^*$ of $P_2$:
  \begin{align}\label{liftedDC}
    F (x,y) = P_1(x) - \langle {\cal A}x,y\rangle  + P_2^*(y).
  \end{align}
  Note that $f(x) = \inf_yF(x,y)$ thanks to the definition of Fenchel conjugate and \cite[Theorem~12.2]{Roc70}. In Theorem~\ref{DCth}, we will deduce the KL exponent of  $f$ in \eqref{DC} from that of $F$ in \eqref{liftedDC}.

  \item[(ii) Bregman envelope] The Bregman envelope of a proper closed function \revise{$f:\mathbb{X}\to \mathbb{R}\cup\{\infty\}$},  is defined in \cite{BaCoNo2004} as follows:
\begin{align}\label{Benve}
  F_\phi(x):=\inf_y\{f(y) + \B_\phi(y,x)\}
\end{align}
where $\phi:\mathbb{X}\to \R$ is a differentiable convex function and
\begin{align}\label{Bphi}
 \B_\phi(y,x)= \phi(y)-\phi(x)-\langle \nabla \phi(x),y-x\rangle
\end{align}
is the Bregman distance.
Note that $F_\phi$ is an inf-projection by definition. In Section~\ref{BreE}, we will show that if $\phi$ satisfies Assumption~\ref{Ass11} and $f$ is a KL function with exponent $\alpha\in(0,1]$ and satisfies $\inf f > -\infty$, then  $F_\phi$ in \eqref{Benve} is also a KL function with exponent $\alpha\in(0,1]$. As we shall see in Remark \ref{remarkEnvelope}, the $F_\phi$ with $\phi$ satisfying   Assumption~\ref{Ass11} covers the widely studied Moreau envelope (see, for example, \cite[Section~1G]{RocWets98}) and the recently proposed forward-backward envelope \cite{StThPa17}.

\item[(iii) Least squares loss function with rank constraint ] Consider the following least squares loss function with rank constraint:
\begin{align}\label{leastRank}
f (X):=\frac{1}{2}\|{\cal A}X-b\|^2 + \delta_{{\rm rank}(\cdot)\le k}(X),
\end{align}
where $X\in\R^{m\times n}$, ${\cal A}:\R^{m\times n}\to \R^p$ is a linear map, $b\in\R^p$ and \reviser{$k$ is an integer between $1$ and $\min\{m,n\}-1$}. The model above is considered in many applications such as principal components analysis (PCA); see \cite{UdHoZaBo16} for more details.  Notice  that $f$ in \eqref{leastRank} is an inf-projection in the following form:
\begin{equation}\label{f4=hatf4}
f (X)=\inf_{U}\bigg\{\frac{1}{2}\|{\cal A}X-b\|^2 + \delta_{\mathfrak{\widehat D}}(X,U)\bigg\},
\end{equation}
where
\[
\mathfrak{\widehat D}:=\{(X,U)\in\R^{m\times n}\times\R^{m\times (m-k)}:\;U^TX=0 {\rm\ and\ }U^TU=I_{m-k}\},
\]
and $I_{m-k}$ is the identity matrix of size $m-k$. In Section \ref{Lr}, we first establish an auxiliary KL calculus rule concerning Lagrangian in Theorem~\ref{equalityconstrainLL}. Then, using this result together with Theorem~\ref{general}, we give an explicit KL exponent (dependent on $n$, $m$ and $k$) of $f$ in \eqref{leastRank} in Theorem~\ref{lstsqRankKL}.
 \end{description}
\color{black}

\section{KL exponents for some convex models}\label{sec4}

%In this section, we demonstrate how Theorem~\ref{general} can be applied to deducing the KL exponent of functions commonly encountered in applications.

\subsection{Convex models with SDP-representable structure}\label{sec4.1}

In this section, we explore the KL exponent of SDP-representable functions \color{blue} introduced in  Section~\ref{Coninf}(i). More specifically, we will deduce the KL exponent of a proper closed function $f$ with its epigraph represented as in \eqref{epiSDP}, under suitable conditions on $F$ in \eqref{F}. To this end, we collect the  $u$ components in $\mathfrak{D}$ in \eqref{F} for each fixed $x\in {\rm dom\,}\partial f$ and define the following set:
\begin{align}\label{Dbarx}
\mathfrak{D}_x=\left\{u\in \R^N:\;(x,u,f(x))\in\mathfrak{D}\right\}.
\end{align}
Roughly speaking, these are extra variables that correspond to the ``$x$-slice" in the ``lifted" SDP representation.  As we shall  see in the proof of Theorem~\ref{corSDP}, when $f$ is the sum of LMI-representable functions (which is SDP-representable), one can have $\mathfrak{D}_{x }=\{(f_1( x),\dots,f_m( x))\}$.

 We begin with three auxiliary lemmas. The first one relates the KL exponent of $f$, whose epigraph is represented as in \eqref{epiSDP}, to that of $F$  in \eqref{F}.
\begin{lemma}\label{appTh1}
  Let $f:\R^n\to\mathbb{R}\cup\{\infty\}$ be a proper closed SDP-representable function with its epigraph represented as in \eqref{epiSDP}. Then the function $F$ defined in \eqref{F} is proper, closed and convex.

  Next, suppose in addition that $\bar x\in{\rm dom\,}\partial f$, $\alpha\in [0,1)$, and that the following conditions hold:
  \begin{enumerate}[{\rm (i)}]
    \item The set $\mathfrak{D}_{\bar x}$ defined as in \eqref{Dbarx} is nonempty and compact.
    \item The function $F$ defined in \eqref{F} satisfies the KL property with exponent $\alpha$ at every point in $\{\bar x\}\times \mathfrak{D}_{\bar x}\times \{f(\bar x)\}$.
  \end{enumerate}
  Then $f$ satisfies the KL property at $\bar x$ with exponent $\alpha$.
\end{lemma}\color{black}
\begin{proof}
  Observe from the definition that
   \[
  f(x)=\inf_{u,t} F(x,u,t).
  \]
  First, note that $\mathfrak{D}\neq\emptyset$ because $f$ is proper. Since $\mathfrak{D}$ is clearly closed and convex, we conclude that $F$ is \revise{proper, closed and convex}.
  \revise{We will now check the conditions in Corollary~\ref{convex_general} and apply the corollary to deduce the KL property of $f$ from that of $F$.}

  \revise{To this end}, by assumption, we see that $F$ satisfies the KL property with exponent $\alpha$ on $\{\bar x\}\times \mathfrak{D}_{\bar x}\times \{f(\bar x)\} = \{\bar x\}\times\Argmin_{u,t} F(\bar x,u,t)$ and that $\mathfrak{D}_{\bar x}$ is nonempty and compact. \color{blue}The desired conclusion now follows \color{black} from a direct application of Corollary~\ref{convex_general}. This completes the proof.
\end{proof}

\color{blue}The second lemma relates  the KL exponent of $F$ in \eqref{F} to that of  another SDP-representable function with carefully constructed matrices involved in its representation.
 \begin{lemma}\label{FF1}
 Let $f$ be a proper closed function and $\bar x\in {\rm dom}\, f$. Suppose that $f$ is SDP-representable with its epigraph represented as in \eqref{epiSDP}, and that there exists $(x^s,u^s,t^s)$ such that $A_{00}+\mathcal{A}^*(x^s,u^s,t^s)\succ 0$, where $A_{00}$ and $\mathcal{A}$ are given in \eqref{epiSDP} and \eqref{mathcalA} respectively. Let $F$ be defined as in \eqref{F} and $\mathfrak{D}_{\bar x}$ be defined as in \eqref{Dbarx}.\footnote{\revise{Notice that $F$ is proper and closed thanks to the existence of the Slater point $(x^s,u^s,t^s)$.}} Let $\bar u\in\mathfrak{D}_{\bar x}$ and suppose that $0\in\partial F(\bar x,\bar u, f(\bar x))$.
Then the following statements hold:
\begin{enumerate}[{\rm (i)}]
 \item It holds that $A_0\neq 0$. Moreover, the set ${\rm span\,}\{A_1,\dots,A_n,B_1,\dots,B_N,A_0\}$ has an orthogonal basis  $\{\hat A_0,\dots,\hat A_{p}\}$, where $p\ge 0$ and $\hat A_0\neq 0$, such that
  \[
  \begin{bmatrix}{\bf a}_1\ \dots\ {\bf a}_n\ {\bf b}_1\ \dots\ {\bf b}_N\ {\bf a}_0\end{bmatrix} = \begin{bmatrix}\hat {\bf a}_1 \dots \hat {\bf a}_p \ \hat {\bf a}_0\end{bmatrix}U
   \]
  for some $U\in\mathbb{R}^{(p+1)\times (n+N+1)}$ having full row rank and the entries of the $(p+1)^{\rm th}$ row of $U$ are $0$ except for $U_{p+1,n+N+1}=1$; here, ${\bf a}_i$, ${\bf b}_j$ and $\hat {\bf a}_k\in \R^{d^2}$ are the columnwise vectorization of the matrices $A_i$, $B_j$ and $\hat A_k$, respectively.
  \item Define $F_1:\R^{p+1}\to \R\cup\{\infty\}$ by
\begin{align}\label{D_3}
F_1(z,t):=t + \delta_{\mathfrak{D}_1}(z,t){\ with\ }\mathfrak{D}_1=\left\{(z,t):\;A_{00} + \hat A_0t +\sum_{w=1}^p\hat A_wz_w \succeq 0\right\},
\end{align}
 where $p\ge 0$ and $\{\hat A_0,\dots,\hat A_{p}\}$ is the orthogonal basis constructed in {\rm (i)}.\footnote{\revise{Note that $F_1$ is proper and closed thanks to the existence of the Slater point $(x^s,u^s,t^s)$.}}
Suppose that $U(\bar x,\bar u,f(\bar x))\in {\rm dom}\,\partial F_1$ and $F_1$ satisfies the KL property at $U(\bar x,\bar u,f(\bar x))$ with exponent $\alpha\in[0,1)$, where $U$ is the same as in {\rm (i)}.\footnote{\revise{Here and henceforth, $U(\bar x,\bar u,f(\bar x))$ is a short-hand notation for the matrix vector product $U\begin{bmatrix}
  \bar x\\ \bar u\\ f(\bar x)
\end{bmatrix}$.}}
Then $F$ satisfies the KL property at $(\bar x,\bar u,f(\bar x))$ with exponent $\alpha$.
  \end{enumerate}
\end{lemma}\color{black}
 \begin{proof}
 Since $0\in\partial F(\bar x,\bar u, f(\bar x))$, we have in view of \cite[Exercise~8.8]{RocWets98} that
 \begin{align}\label{step1}
 0_{n+N+1}\in(0_n,0_N,1)+N_{\mathfrak{D}}(\bar x,\bar u, f(\bar x)),
 \end{align}
where $\mathfrak{D}$ is defined as in \eqref{F}, and $0_k$ is the zero vector of dimension $k$. Next, since $\delta_{\mathfrak{D}}(x,u,t)=[\delta_{\mathcal{S}^d_+-A_{00}}\circ\mathcal{A}^*](x,u,t)$
and we have $\mathcal{A}^*(x^s,u^s,t^s)\succ-A_{00}$ by assumption, using \cite[Theorem~23.9]{Roc70}, we deduce that
\begin{align*}
&N_{\mathfrak{D}}(\bar x,\bar u,f(\bar x))=\partial{\left[\delta_{\mathcal{S}^d_+-A_{00}}\circ\mathcal{A}^*\right]}(\bar x,\bar u,f(\bar x))=\mathcal{A}N_{\mathcal{S}^d_+-A_{00}}(\mathcal{A}^*(\bar x,\bar u,f(\bar x))).
\end{align*}
This together with \eqref{step1} implies that there exists $Y\in N_{\mathcal{S}^d_+-A_{00}}(\mathcal{A}^*(\bar x,\bar u,f(\bar x)))$ such that
\[
\langle A_1,Y\rangle =\dots=\langle  A_n,Y\rangle =\langle B_1,Y\rangle =\dots=\langle  B_N,Y\rangle =0\ \ {\rm but}\ \ \langle A_0,Y\rangle=-1;
\]
in particular, \revise{$A_0\not\in{\rm span\,}\{A_1,\dots,A_n,B_1,\dots,B_N\}$ and hence $A_0\neq 0$. }

\color{blue} If ${\rm span\,}\{A_1,\dots,A_n,B_1,\dots,B_N\} = \{0\}$, then $A_i = B_j=0$ for $i=1,\ldots,n$ and $j = 1,\ldots,N$. In this case, set $\hat A_0 = A_0$. We see that $\{\hat A_0\}$ is an orthogonal set and we have
\[
\begin{bmatrix}{\bf a}_1\ \dots\ {\bf a}_n\ {\bf b}_1\ \dots\ {\bf b}_N\ {\bf a}_0\end{bmatrix} = \hat {\bf a}_0\begin{bmatrix}0_{n+N}^T &1\end{bmatrix},
\]
where $0_{n+N}$ is the zero vector of dimension $n+N$. Thus, the conclusion in (i) holds in this case.

Otherwise, ${\rm span\,}\{A_1,\dots,A_n,B_1,\dots,B_N\} \neq \{0\}$ and we let $\{\bar A_1,\dots, \bar A_p\}$ be a maximal linearly independent subset of $\{A_1,\dots,A_n,B_1,\dots,B_N\}$. Then there exists $M_0\in\mathbb{R}^{p\times (n+N)}$ with full row rank such that $\left[{\bf a}_1\ \dots\ {\bf a}_n\ {\bf b}_1\ \dots\ {\bf b}_N\right] = \left[\bar {\bf a}_1\ \dots\ \bar {\bf a}_p\right]M_0$, where $\bar {\bf a}_i\in \R^{d^2}$ is the columnwise vectorization of $\bar A_i$. Thus
\begin{align}\label{nonlin}
\left[{\bf a}_1\ \dots\ {\bf a}_n\ {\bf b}_1\ \dots\ {\bf b}_N\ {\bf a}_0\right] = \left[\bar {\bf a}_1\ \dots\ \bar {\bf a}_p\ {\bf a}_0\right]\begin{bmatrix}
  M_0 & 0\\
  0   & 1
\end{bmatrix}.
\end{align}
Using Gram-Schmidt process followed by a suitable scaling  to $\{\bar A_1,\dots,\bar A_p,A_0\}$, there exists an invertible upper triangle  matrix $U_0\in\R^{(p+1)\times (p+1)}$ with the $(U_0)_{p+1,p+1}=1$ and an orthogonal basis $\{\hat A_1,\dots,\hat A_p,\hat A_0\}$ of ${\rm span\,}\{\bar A_1,\dots, \bar A_p,A_0\}$ such that $\left[\bar {\bf a}_1\ \dots\ \bar {\bf a}_p\ {\bf a}_0\right] = \begin{bmatrix}\hat {\bf a}_1\ \dots\ \hat {\bf a}_p \ \hat {\bf a}_0\end{bmatrix}U_0$, where $\hat {\bf a}_i$ is the columnwise vectorization of $\hat A_i$. This together with \eqref{nonlin} shows that
 \[
  \begin{bmatrix}{\bf a}_1\ \dots\ {\bf a}_n\ {\bf b}_1\ \dots\ {\bf b}_N\ {\bf a}_0\end{bmatrix} =\left[\bar {\bf a}_1\ \dots\ \bar {\bf a}_p\ {\bf a}_0\right]\begin{bmatrix} M_0 & 0\\
  0   & 1
\end{bmatrix}= \begin{bmatrix}\hat {\bf a}_1\ \dots\ \hat {\bf a}_p \ \hat {\bf a}_0\end{bmatrix}U,
   \]
   where $U := U_0\begin{bmatrix} M_0 & 0\\
  0   & 1
\end{bmatrix}$ has full row rank and the entries of the $(p+1)^{\rm th}$ row of $U$ are $0$ except for $U_{p+1,n+N+1}=1$. This proves (i).

Now, using the definition of $F_1$ in \eqref{D_3}, we have $F(x,u,t) = F_1(U(x,u,t))$.
Since $U$ is surjective and the KL exponent of $F_1$ is $\alpha$ at $U(\bar x,\bar u,f(\bar x))$, using a similar argument as in \cite[Theorem~3.2]{LiPong17}, the KL exponent of $F$ at $(\bar x,\bar u,f(\bar x))$ equals $\alpha$. This completes the proof. \color{black}
 \end{proof}

\color{blue}Finally, we rewrite $F_1$ in \eqref{D_3} suitably as a function on ${\cal S}^d$ that satisfies a certain ``strict complementarity" condition so that Lemma~\ref{KLG} can be \reviser{readily} applied to deducing the KL exponent of $F_1$ explicitly.
\begin{lemma}\label{FF2}
 Let $f$ be a proper closed function and $\bar x\in{\rm dom}\, f$. Suppose in addition that $f$ is SDP-representable with its epigraph represented as in \eqref{epiSDP}. Let $F$ be defined as in \eqref{F},  $\mathfrak{D}_{\bar x}$ be defined as in \eqref{Dbarx}, and $\bar u\in\mathfrak{D}_{\bar x}$. Suppose  that the following conditions hold:
  \begin{enumerate}[{\rm (i)}]
    \item{\bf(Slater's condition)} There exists $(x^s,u^s,t^s)$ such that $A_{00}+\mathcal{A}^*(x^s,u^s,t^s)\succ 0$, where $A_{00}$ and $\mathcal{A}$ are given in \eqref{epiSDP} and \eqref{mathcalA} respectively.\footnote{\revise{Note that this condition implies that both $F$ in \eqref{F} and $F_1$ in \eqref{D_3} are proper and closed.}}
    \item{\bf(Strict complementarity)} It holds that $0\in\revise{{\rm ri\,}\partial F(\bar x,\bar u, f(\bar x))}$.
  \end{enumerate}
 Let $F_1$ be defined as in \eqref{D_3}. Then $U(\bar x,\bar u,f(\bar x))\in {\rm dom}\,\partial F_1$ and $F_1$ satisfies the KL property at $U(\bar x,\bar u,f(\bar x))$ with exponent $\frac12$, where $U$ is given in Lemma \ref{FF1}{\rm (i)}.
\end{lemma}\color{black}
\begin{proof}
Define $\mathcal{\bar A}:\mathcal{S}^d\to\R^{p+1}$ by
\begin{align*}
 \mathcal{\bar A}(W):= \left[\langle \hat A_1,W\rangle \ \dots\ \langle \hat A_p,W\rangle \ \langle \hat A_0,W\rangle \right]^T,
\end{align*}
\color{blue}where $\{\hat A_0,\dots,\hat A_p\}$ is given by Lemma~\ref{FF1}(i). \color{black}
Since \revise{$\{\hat A_0,\dots,\hat A_p\}$} is orthogonal, we see that $\mathcal{\bar A}$ is surjective and
$\mathcal{\bar A}^*:\R^{p+1}\to\mathcal{S}^d$ with $\mathcal{\bar A}^*(z,t):=  \hat A_0t + \sum_{w=1}^p\hat A_wz_w$
 is injective. Also, for any $(z,t)\in\R^{p+1}$, by orthogonality,
\[
\mathcal{\bar A}\mathcal{\bar A}^*(z,t)=\mathcal{\bar A}\left(\hat A_0t +\sum_{w=1}^p\hat A_wz_w\right)=\left(\|\hat A_1\|_F^2z_1,\dots,\|\hat A_p\|_F^2z_p,\|\hat A_0\|_F^2t\right).
\]

Choose a basis $\{H_1,H_2,\dots,H_r\}$ of $\ker\mathcal{\bar A}$ \revise{and define} a linear map $\mathcal H:{\cal S}^d\to \R^{r}$ by\footnote{\revise{In the case when $\ker\mathcal{\bar A}=\{0\}$ so that the basis is empty (i.e., $r = 0$), we define ${\cal H}$ to be the {\em unique} linear map that maps ${\cal S}^d$ onto the {\em zero vector space}.}}
\begin{equation}\label{trans1.7}
\begin{aligned}
\mathcal H(W):= \left[\langle H_1,W\rangle \ \cdots\ \langle H_r,W\rangle \right]^T.
\end{aligned}
\end{equation}
\revise{Define} a proper closed function \reviser{$F_2: \mathcal{S}^d \rightarrow \mathbb{R} \cup \{\infty\}$} by
 \begin{equation}\label{F3}
 F_2(X):={\|\hat A_0\|_F^{-2}}\langle\hat A_0,X\rangle +\delta_{\mathfrak{D}_2}(X){\rm \ with\ } \mathfrak{D}_2:=\left\{X\in {\cal S}^d_+:\;\mathcal HX=\mathcal H A_{00}\right\}.
 \end{equation}
\color{blue}Thanks to the identity $(\ker \bar {\cal A})^\perp = {\rm Range\,}(\bar {\cal A}^*)$ and the fact that ${\cal H}X = {\cal H}A_{00}$ if and only if $X - A_{00}\in (\ker \bar {\cal A})^\perp$, \color{black}  we have the following relations concerning ${\frak D}_2$ and the ${\frak D}_1$ defined in \eqref{D_3}:\color{black}
\begin{equation}\label{eqqui}
  \begin{aligned}
  &\qquad\qquad\quad\quad (z,t)\in\mathfrak{D}_1 \Longrightarrow A_{00}+\mathcal{\bar A}^*(z,t)\in \mathfrak{D}_2,\\
  &X \in \mathfrak{D}_2 \Rightarrow \mbox{$\exists$ unique $(z,t)$ s.t. $A_{00}+\mathcal{\bar A}^*(z,t) = X$, and $(z,t)\in {\frak D}_1$},
  \end{aligned}
\end{equation}
\color{blue}where the second implication also makes use of  the injectivity of ${\bar {\cal A}}^*$. \color{black}
We then deduce further that for any $(z,t)\in \R^{p+1}$,
 \begin{equation}\label{F2val}
 \begin{aligned}
&F_2(A_{00}+\mathcal{\bar A}^*(z,t))-{\|\hat A_0\|_F^{-2}}\langle \hat A_0,A_{00}\rangle \\
&=\langle \mathcal{\bar A}\left({\|\hat A_0\|_F^{-2}} \hat A_0\right),(z,t)\rangle +\delta_{\mathfrak{D}_2}(A_{00}+\mathcal{\bar A}^*(z,t))\\
&=\revise{t} +\delta_{\mathfrak{D}_2}(A_{00}+\mathcal{\bar A}^*(z,t))=F_1(z,t),
\end{aligned}
\end{equation}
where \revise{the last equality follows from \eqref{eqqui}.}

\color{blue}Next, let $U$ be as in Lemma \ref{FF1}(i). Since the entries in the $(p+1)^{\rm th}$ row of $U$ are $0$ except for $U_{p+1,n+N+1}=1$, there exists $\bar z\in\mathbb{R}^p$ such that\footnote{\revise{Recall that $p\ge 0$. When $p=0$, we interpret $\bar z$ as a null vector so that $U(\bar x,\bar u,f(\bar x)) = f(\bar x)$.}}
\begin{equation}\label{barz}
U(\bar x,\bar u,f(\bar x)) = (\bar z,f(\bar x)).
\end{equation}
Now, define
\begin{equation}\label{barXnumber}
\bar X:=A_{00}+\mathcal{\bar A}^*(\bar z,f(\bar x)).
\end{equation}
\color{black}
We claim that $0\in\color{blue}{\rm ri\,}\partial F_2(\bar X)\color{black}$. We first show that
   \begin{align}\label{riF_1}
  0\in \color{blue}{\rm ri\,}\partial F_1(\bar z,f(\bar x))\color{black}.
  \end{align}
 \color{blue} In fact, using \cite[Theorem~23.9]{Roc70} (note that $U(x^s,u^s,t^s)\in {\rm int\,}{\frak D}_1$ thanks to assumption (i)) together with the assumption (ii), we have
\begin{align*}%\label{riF_U}
0 \in {\rm ri\,}\partial F(\bar x,\bar u,f(\bar x))={\rm ri\,}\left[U^T\partial F_1\left(U(\bar x,\bar u,f(\bar x))\right)\right]=U^T{\rm ri\,}\partial F_1\left(U(\bar x,\bar u,f(\bar x))\right),
\end{align*}
\color{black}where the second equality follows from \cite[Theorem~6.6]{Roc70}. Since $U$ has full row rank and thus $U^T$ is injective, \color{blue}recalling the definition of $\bar z$ in \eqref{barz}\color{black}, we deduce further that \eqref{riF_1} holds. Now, using this and \cite[Exercise~8.8]{RocWets98}, we have
   \begin{align}\label{step4}
 0\in \color{blue}{\rm ri\,}\partial F_1(\bar z,f(\bar x))\color{black}=\revise{(\underbrace{0,\ldots,0}_{p\ {\rm entries}},1)}  + \color{blue}{\rm ri\,}N_{\mathfrak{D}_1}(\bar z,f(\bar x))\color{black}.
   \end{align}
Now, notice that $\delta_{\mathfrak{D}_1}(z,t)=\left[\delta_{S_+^d-A_{00}}\circ\mathcal{\bar A}^*\right](z,t)$  and \color{blue}
\begin{equation}\label{Slater}
\mathfrak{D}_2 \ni X^s:= A_{00}+\mathcal{\bar A}^*(z^s,t^s)=A_{00}+\mathcal{A}^*(x^s,u^s,t^s)\succ0
\end{equation}\color{black}
with $(z^s,t^s)=U( x^s, u^s, t^s)$\color{blue}, where the inclusion holds thanks to \eqref{eqqui}\color{black}. Using these and \cite[Theorem~23.9]{Roc70}, we see that
\begin{align*}
\color{blue}{\rm ri\,}N_{\mathfrak{D}_1}(\bar z,f(\bar x))={\rm ri\,}{\partial\left[{\delta_{S_+^d-A_{00}}}\circ\mathcal{\bar A}^*\right]}(\bar z,f(\bar x))={\rm ri\,}\mathcal{\bar A}N_{S_+^d}(\bar X)=\mathcal{\bar A}\,{\rm ri\,}N_{S_+^d}(\bar X),\color{black}
\end{align*}
where the last equality follows from \cite[Theorem~6.6]{Roc70}. This together with \eqref{step4} implies that there exists $\tilde Y\in \color{blue}{\rm ri\,}N_{S_+^d}(\bar X)\color{black}$ such that
\begin{equation}\label{display}
\langle\hat A_1,\tilde Y\rangle=\dots=\langle \hat A_p,\tilde Y\rangle=0 \ {\rm and}\ \langle\hat A_0,\tilde Y\rangle=-1.
\end{equation}
The second relation in \eqref{display} gives $\langle \hat A_0,\tilde Y+{\|\hat A_0\|_F^{-2}}\hat A_0\rangle =\langle \hat A_0,\tilde Y\rangle +1=0$.
In addition, in view of the first relation in \eqref{display} and the orthogonality of \revise{$\{\hat A_0,\ldots,\hat A_p\}$}, we have
$
\langle \hat A_i,\tilde Y+{\|\hat A_0\|_F^{-2}}\hat A_0\rangle =\langle \hat A_i,\tilde Y\rangle +\langle \hat A_i,{\|\hat A_0\|_F^{-2}}\hat A_0\rangle =0$ for all $i=1,\dots,p$. Thus, it holds that
$
\tilde Y+{\|\hat A_0\|_F^{-2}}\hat A_0\in \color{blue}{\rm ker\,}\mathcal{\bar A}\color{black}$.
Hence, there exists $\omega\in\R^r$ such that
\begin{align}\label{YY}
\tilde Y+{\|\hat A_0\|_F^{-2}}\hat A_0=\sum_{i=1}^rH_i\omega_i
\end{align}
with $r$ and $H_i$ defined as in \eqref{trans1.7}.\footnote{\revise{In the case when $\ker \bar{\cal A} = \{0\}$ (i.e., $r = 0$), we have $\tilde Y+{\|\hat A_0\|_F^{-2}}\hat A_0 = 0$. In this case, we interpret $\omega$ as a null vector.}}
Using \eqref{YY} and the definition of $\tilde Y$, we have further that
\begin{align}\label{riF_3}
0=\tilde Y+{\|\hat A_0\|_F^{-2}}\hat A_0-\sum_{i=1}^rH_i\omega_i\in\color{blue}{\rm ri\,}N_{\mathcal{S}^d_+}(\bar X)\color{black}+{\|\hat A_0\|_F^{-2}}\hat A_0+\color{blue}{\rm Range\,}\mathcal{H}^*\color{black}.
\end{align}

On the other hand, using the definition of $F_2$ in \eqref{F3}, we have
 \begin{align*}
& \color{blue}{\rm ri\,}\partial F_2(\bar X)\color{black}={\|\hat A_0\|_F^{-2}}\hat A_0 + \color{blue}{\rm ri\,}\partial\delta_{\mathfrak{D}_2}(\bar X)\color{black}
={\|\hat A_0\|_F^{-2}}\hat A_0 + \color{blue}{\rm ri\,}\left(N_{{\cal H}^{-1}\{{\cal H}A_{00}\}}(\bar X)+N_{\mathcal{S}^d_+}(\bar X)\right)\color{black}\\
 &={\|\hat A_0\|_F^{-2}}\hat A_0 + \color{blue}{\rm ri\,}N_{{\cal H}^{-1}\{{\cal H}A_{00}\}}(\bar X)\color{black}+ \color{blue}{\rm ri\,}N_{\mathcal{S}^d_+}(\bar X)\color{black}
 ={\|\hat A_0\|_F^{-2}}\hat A_0 + \color{blue}{\rm Range\,}\mathcal{H}^*+ {\rm ri\,}N_{\mathcal{S}^d_+}(\bar X)\color{black},
  \end{align*}
  where the second equality follows from \revise{\cite[Theorem~23.8]{Roc70} and \eqref{Slater}}, and the third  equality follows from \cite[Corollary~6.6.2]{Roc70}. This together with \eqref{riF_3} shows
  \begin{equation}\label{F3haha}
  0\in\color{blue}{\rm ri\,}\partial F_2(\bar X)\color{black}.
  \end{equation}
 In view of \eqref{Slater} and \eqref{F3haha}, we \revise{can} now apply Lemma \ref{KLG} \revise{and} deduce that, for a given compact neighborhood $\mathfrak{U}$ of $\bar X$, there exists $c>0$ such that for any $X\in\mathfrak{U}\cap{\mathfrak{D}}_2$,
\begin{align}\label{USD+}
\dist(X,\Argmin F_2)&\le  c \left(F_2(X)- F_2(\bar X)\right)^{\frac12}.
\end{align}
Thus, fix an $\epsilon > 0$ so that $A_{00} + \bar {\cal A}^*(z,t) \in \frak U$ whenever $(z,t) \in B((\bar z,f(\bar x)),\epsilon)$; \color{blue}such an $\epsilon$ exists thanks to the definitions of $\bar z$ in \eqref{barz} and $\bar X$ in \eqref{barXnumber}. \color{black} Now, consider any $(z,t)$ satisfying $(z,t) \in B((\bar z,f(\bar x)),\epsilon)$ and $F_1(\bar z,f(\bar x)) < F_1(z,t) < F_1(\bar z,f(\bar x)) + \epsilon$. Then $(z,t)\in \color{blue}{\rm dom\,}F_1\color{black}$, which means $A_{00} + \bar {\cal A}^*(z,t)\in {\frak D}_2$ according to \eqref{eqqui}. Hence, using \eqref{USD+}, we have
\begin{align*}
& \dist^2((z,t),\Argmin F_1)\le  \|(z,t)-(z^*,t^*)\|^2\overset{\rm (a)}\le c_1\left\|\bar{\mathcal{A}}^*(z,t)-\bar{\mathcal{A}}^*(z^*,t^*)\right\|^2_F\\
&=c_1\|A_{00}+\bar{\mathcal{A}}^*(z,t)-X^*\|^2_F=c_1\dist^2(A_{00}+\bar{\mathcal{A}}^*(z,t),\Argmin F_2)\\
&\le c^2c_1 \left(F_2(A_{00}+\bar{\mathcal{A}}^*(z,t))- F_2(\bar X)\right)\overset{\rm (b)}=c^2c_1\left(F_1(z,t)- F_1(\bar z,f(\bar x))\right),
\end{align*}
where $X^*$ denotes the projection of $A_{00}+\bar{\mathcal{A}}^*(z,t)$ on $\Argmin F_2$ and $(z^*,t^*)$ is the corresponding element in $\Argmin F_1$ such that $X^*=A_{00}+\bar{\mathcal{A}}^*(z^*,t^*)$ (the existence of $(z^*,t^*)$ follows from \eqref{eqqui} and \eqref{F2val}), (a) holds for some $c_1>0$ because $\bar{\cal A}^*$ is injective, and (b) follows from \eqref{F2val}. Combining this with \cite[Theorem~5]{BoNgPeSu17}, we conclude that $F_1$ satisfies the KL property with exponent $\frac12$ at $(\bar z,f(\bar x)) = U(\bar x,\bar u,f(\bar x))$.
\end{proof}
We are now ready to state and prove our main result in this section.
\begin{theorem}[{{\bf KL exponent of SDP-representable functions}}]\label{SDP}
   Let $f$ be a proper closed function and \color{blue}$\bar x\in {\rm dom}\,\partial f$\color{black}. Suppose in addition that $f$ is SDP-representable with its epigraph represented as in \eqref{epiSDP} and that the following conditions hold:
  \begin{enumerate}[{\rm (i)}]
    \item{\bf(Slater's condition)} There exists $(x^s,u^s,t^s)$ such that $A_{00}+\mathcal{A}^*(x^s,u^s,t^s)\succ 0$, where $A_{00}$ and $\mathcal{A}$ are given in \eqref{epiSDP} and \eqref{mathcalA} respectively.
    \item{\bf(Compactness)} The set $\mathfrak{D}_{\bar x}$ defined as in \eqref{Dbarx} is nonempty and compact.
    \item{\bf(Strict complementarity)} It holds that $0\in\color{blue}{\rm ri\,}\partial F(\bar x,u, f(\bar x))\color{black}$ for all $u\in\mathfrak{D}_{\bar x}$, where $F$ is defined as in \eqref{F} and $\mathfrak{D}_{\bar x}$ is defined as in \eqref{Dbarx}.\footnote{\revise{We note that because of the Slater's condition, the function $F$ in \eqref{F} is proper and closed.}}
  \end{enumerate}
  Then $f$ satisfies the KL property at $\bar x$ with exponent $\frac12$.
\end{theorem}
\begin{remark}
  In Theorem~\ref{SDP}, we require $0\in\color{blue}{\rm ri\,}\partial F(\bar x,u, f(\bar x))\color{black}$ for all $u\in\mathfrak{D}_{\bar x}$ with $\mathfrak{D}_{\bar x}$ defined as in \eqref{Dbarx}. This can be hard to check in practice. In Sections~\ref{sec4.3.1} and \ref{sec4.3.2}, we will impose additional assumptions on $f$ so that this condition can be replaced by $0 \in \color{blue}{\rm ri\,}\partial f(\bar x)\color{black}$, which is a form of strict complementarity condition \color{blue} imposed on the original function $f$ (rather than  the representation $F$ in the lifted space).\color{black}
\end{remark}
\begin{proof}
 In view of Lemma \ref{appTh1}, it suffices to show that $F$ satisfies the KL property with exponent $\frac12$ at every point in $\{\bar x\}\times \mathfrak{D}_{\bar x}\times \{f(\bar x)\}$.
\color{blue} Fix any $\bar u\in\mathfrak{D}_{\bar x}$. From Lemma \ref{FF2}, we know that $F_1$ defined as in \eqref{D_3} has KL property with exponent $\frac12$ at $U(\bar x,\bar u,f(\bar x))\in {\rm dom}\,\partial F_1$, where $U$ is given in Lemma \ref{FF1}(i). Using this together with Lemma \ref{FF1}, we know that $F$ satisfies the KL property with exponent $\frac12$ at  $(\bar x,\bar u,f(\bar x))$. This completes the proof. \color{black}
\end{proof}

We would like to point out that the third condition in Theorem \ref{SDP} cannot be replaced by ``$0\in\color{blue}{\rm ri\,}\partial f(\bar x)\color{black}$" in general. One concrete counter-example is $f(x)=x^4$. Indeed, for this function, the global minimizer is $0$ and we have $\partial f(0)=\{\nabla f(0)\}=\{0\}$, which implies that $0\in\color{blue}{\rm ri\,}\partial f(0)\color{black}$. Moreover, this function is SDP-representable:
  \begin{align*}%\label{x4}
 \color{blue}{\rm epi\,}f\color{black}=\left\{(x,t)\revise{\in \R^n\times \R}:\; \left[\begin{array}{cccc}
    1 & y & 0 & 0\\
    y &t  & 0 & 0\\
    0 & 0 & 1 & x\\
    0 & 0 & x & y \end{array}
  \right] \succeq0 \mbox{ for some }y\right\}.
  \end{align*}
  It is easy to check that the first two conditions of Theorem \ref{SDP} are satisfied for $\bar x = 0$.
  However, \color{blue}it can be directly verified that this $f$ does not have KL property with exponent $\frac12$ at $0$. \color{black} This concrete example suggests that the third condition in Theorem~\ref{SDP} cannot be replaced by $0\in\color{blue}{\rm ri\,}\partial f(\bar x)\color{black}$ in general.

Next, in Sections~\ref{sec4.3.1} and \ref{sec4.3.2}, we will look at special SDP-representable functions and show that the third condition in Theorem~\ref{SDP} can indeed be replaced by $0\in\color{blue}{\rm ri\,}\partial f(\bar x)\color{black}$ in those cases.

\subsection{Sum of LMI-representable functions}\label{sec4.3.1}
\color{blue}In this section, we discuss how the KL exponent of the sum of finitely many proper closed LMI-representable functions as defined in Section \ref{Coninf}(ii) can be deduced through  Theorem~\ref{SDP}. Compared with Theorem~\ref{SDP}, the strict complementarity condition in this section is now imposed {\em directly} on the original function.\color{black}

\begin{theorem}[{{\bf KL exponent of sum of LMI-representable functions}}]\label{corSDP}
  Let $f=\sum_{i=1}^mf_i$, where each \color{blue}$f_i:\R^n\to\mathbb{R}\cup\{\infty\}$ \color{black} is proper and closed. Suppose that each $f_i$ is LMI-representable, i.e., there exist \reviser{$d_i \ge 1$}  and matrices $\{A_{00}^i,A_0^i,A_1^i,\ldots,A_n^i\}\subset {\cal S}^{d_i}$ such that
  \[
  \color{blue}{\rm epi\,}f_i\color{black}=\left\{(x,t)\in \R^n\times \R:\;A_{00}^i +\sum_{j=1}^nA^i_jx_j + A_0^it\succeq 0\right\}.
  \]
  Suppose in addition that there exist {$x^{s}\in \R^n$ and $s^s\in \R^m$} such that for $i=1,\ldots,m$,
  \[
  A_{00}^i + \sum_{j=1}^nA^i_jx^{s}_j + A_0^is^s_i\succ 0.
  \]
  If $\bar x\in\color{blue}{\rm dom\,} \partial f\color{black}$ satisfies $0\in\color{blue}{\rm ri\,}\partial f(\bar x)\color{black}$, then $f$ satisfies the KL property at $\bar x$ with exponent $\frac12$.
\end{theorem}
\begin{proof}
  We first derive an SDP representation of $\color{blue}{\rm epi\,}f\color{black}$. To this end, define
  \[
  \mathfrak{\hat{D}}:=\left\{(x,s,t):\; t\ge \sum_{i=1}^ms_i {\rm \ and\ }s_i\ge f_i(x),\ \forall i=1,\dots,m\right\}.
  \]
  Then it holds that $(x,s,t)\in \mathfrak{\hat D}$ if and only if
  \begin{equation}\label{Dhat}
  \begin{bmatrix}
    t-\sum_{i=1}^ms_i & 0 & \cdots &0 \\
  0& A_{00}^1 +\sum_{j=1}^nA^1_jx_j + A_0^1s_1& & \vdots \\
  \vdots& &\ddots & \\
  0& \cdots & &A^m_{00}+\sum_{j=1}^nA^m_jx_j + A_0^ms_m
  \end{bmatrix}\succeq 0.
  \end{equation}
  Since
  \begin{equation}\label{eqquiCor}
  \begin{aligned}
  &(x,t)\in\color{blue}{\rm epi\,}f\color{black}\ \ \Longleftrightarrow\ \ t\ge \sum_{i=1}^mf_i(x)\ \ \Longleftrightarrow\ \  \exists s\in\R^m\ {\rm s.t.}\ (x,s,t)\in\mathfrak{\hat D},
  \end{aligned}
  \end{equation}
  we see that $f$ is SDP-representable. Moreover, if we define
  \begin{equation}\label{LMI_f}
  F(x,s,t):=t +\delta_{\mathfrak{\hat D}}(x,s,t),
  \end{equation}
  then it holds that $f(x) = \inf_{s,t}F(x,s,t)$ for all $x\in \R^n$. We next show that $f$ and the $F$ defined in \eqref{LMI_f} satisfy the conditions required in Theorem~\ref{SDP}.

  First, from the definition of $x^{s}\in \R^n$ and $s^s\in \R^m$, we have
  \[
  \begin{bmatrix}
    t^s-\sum_{i=1}^ms^s_i & 0 & \cdots &0 \\
  0& A_{00}^1 +\sum_{j=1}^nA^1_jx^{s}_j + A_0^1s^s_1& & \vdots \\
  \vdots& &\ddots & \\
  0& \cdots & &A^m_{00}+\sum_{j=1}^nA^m_jx^{s}_j + A_0^ms^s_m
  \end{bmatrix}\succ 0,
  \]
  where $t^s := \sum_{i=1}^ms^s_i + 1$. This together with \eqref{Dhat} and \eqref{eqquiCor} shows that condition (i) in Theorem~\ref{SDP} holds.

  Next, note that the set $\{s:\; (\bar x,s,f(\bar x))\in \revise{\hat {\frak D}}\} = \{(f_1(\bar x),\ldots,f_m(\bar x))\}$, which is clearly nonempty and compact. In view of this and \eqref{LMI_f}, we conclude that
  condition (ii) in Theorem~\ref{SDP} is satisfied.

  Finally, we look at the strict complementarity condition, i.e., condition (iii) in Theorem~\ref{SDP}. Notice that the definition of $x^s\in \R^n$ implies
  \begin{equation}\label{Jan27}
    x^s\in\bigcap_{i=1}^m \color{blue}{\rm int\,}{\rm dom\,} f_i\color{black}.
  \end{equation}
  Write $\bar s := (f_1(\bar x),\cdots,f_m(\bar x))$ for notational simplicity.
  Define
  \[
  \mathfrak{C}_0=\left\{(x,s,t):\; t\ge \sum_{i=1}^ms_i\right\} {\rm \ and\ } \reviser{\mathfrak{C}_i=\{(x,s,t):\;s_i\ge f_i(x)\}},\ \forall \reviser{i=1,\dots,m}.
  \]
  Then $\mathfrak{\hat D}=\reviser{\bigcap_{i=0}^m\mathfrak{C}_i}$.
  Moreover, using \cite[Theorem~7.6]{Roc70}, \revise{we have} for \reviser{$i=1,\dots,m$} that
  \begin{align*}
  &\color{blue}{\rm ri\,}\reviser{\mathfrak{C}_i}\color{black}=\color{blue}{\rm ri\,}\{(x,s,t):\; \reviser{g_i(x,s,t)\le 0\}}\color{black}=\{(x,s,t)\in\color{blue}{\rm ri\,}{\rm dom\,}\reviser{g_i}\color{black}:\;\reviser{g_i(x,s,t)}<0\}\\
  &=\left\{(x,s,t)\in\color{blue}{\rm ri\,}{\rm dom\,}\reviser{f_i}\times \reviser{\R^{m} \times \mathbb{R}} \color{black}:\;\reviser{g_i(x,s,t)}<0\right\},
  \end{align*}
  where \reviser{$g_i(x,s,t)=f_i(x)-s_i$ for each $i$}.
  This together with \eqref{Jan27} shows that \reviser{$\bigcap_{i=0}^m \color{blue}{\rm ri\,}\mathfrak{C}_i\color{black}\neq\emptyset$}.
  Using this, \cite[Theorem~23.8]{Roc70} and the definition of $F$ in \eqref{LMI_f}, we have
  \begin{equation}\label{PFSDPcor}
  \partial F(\bar x,\bar s,f(\bar x))=(0_{n+m},1) + \reviser{\sum_{i=0}^m N_{\mathfrak{C}_i}(\bar x,\bar s,f(\bar x))},
  \end{equation}
  where $0_p$ is the zero vector of dimension $p$, and recall that $\bar s = (f_1(\bar x),\cdots,f_m(\bar x))$.

  We claim that $0\in \color{blue}{\rm ri\,}\partial F(\bar x,\bar s,f(\bar x))\color{black}$. To this end, note first that the assumption $0\in \color{blue}{\rm ri\,}\partial f(\bar x)\color{black}$ and \eqref{Jan27} together with \cite[Theorem~23.8]{Roc70} imply that $\bar x\in \bigcap_i \color{blue}{\rm dom\,}\partial f_i\color{black}$. Hence, we have from \cite[Theorem~23.7]{Roc70} that \reviser{for each $i=1,\ldots,m$},
\begin{equation}\label{NCi}
\begin{aligned}
& \reviser{N_{\mathfrak{C}_i}}(\bar x,\bar s,f(\bar x))=\color{blue}{\rm cl\,}\left[\revise{{\rm cone\,} \reviser{\partial g_i}(\bar x,\bar s,f(\bar x))}\right]={\rm cl} \reviser{\bigcup_{\lambda_i\ge 0}(\lambda_i\partial f_i(\bar x),0_{i-1},-\lambda_i,0_{m+1-i})}\color{black}
\end{aligned}
\end{equation}
where the second equality follows from \cite[Proposition~10.5]{RocWets98} and $\color{blue}{\rm cone\,}\mathfrak{B}\color{black}$ denotes the convex conical hull of $\mathfrak{B}$. Similarly, we also have
\begin{equation}\label{NC0}
\begin{aligned}
N_{\mathfrak{C}_0}(\bar x,\bar s,f(\bar x))={\rm cl}\bigcup_{\lambda_0\ge 0}\left(0_n,\lambda_0\cdot 1_m,-\lambda_0\right),
\end{aligned}
\end{equation}
where $1_m$ is the $m$-dimensional vector of all ones.
   Using \eqref{PFSDPcor}, \eqref{NCi} and \eqref{NC0}, we have
   \begin{align*}
   &\color{blue}{\rm ri\,}\partial F\color{black}(\bar x,\bar s,f(\bar x))\overset{\rm (a)}=(0_{n+m},1) + \reviser{\sum_{i=0}^m} \, \color{blue}{\rm ri\,} \reviser{N_{\mathfrak{C}_i}}\color{black}(\bar x,\bar s,f(\bar x))\\
   &\overset{\rm (b)}=(0_{n+m},1) + \reviser{\sum_{i=1}^m} \,\color{blue}{\rm ri\,}\left[{\rm cl} \reviser{\bigcup_{\lambda_i\ge 0}(\lambda_i\partial f_i(\bar x),0_{i-1},-\lambda_i,0_{m+1-i})}\right]\color{black} + \color{blue}{\rm ri\,}\left[{\rm cl}\bigcup_{\lambda_0\ge 0}\left(0_n,\lambda_0\cdot 1_m,-\lambda_0\right)\right]\color{black}\\
   &\overset{\rm (c)}=(0_{n+m},1) + \reviser{\sum_{i=1}^m\bigcup_{\lambda_i>0}(\lambda_i \, \color{blue}{\rm ri\,}\partial f_i(\bar x)\color{black},0_{i-1},-\lambda_i,0_{m+1-i})} + \bigcup_{\lambda_0>0}\left(0_n,\lambda_0\cdot 1_m,-\lambda_0\right)
 \end{align*}
  where (a) follows from \reviser{\eqref{PFSDPcor} and} \cite[Corollary~6.6.2]{Roc70}, (b) follows from \eqref{NCi} and \eqref{NC0}, and (c) follows from \cite[Theorem~6.3]{Roc70} and \cite[Corollary~6.8.1]{Roc70}. This together with $0\in\revise{{\rm ri\,}\partial f(\bar x)}$ yields
  \begin{align*}
  0&\in(\color{blue}{\rm ri\,}\partial f\color{black}(\bar x),0_m,0)=(  0_n,0_m,1)+\left(\color{blue}{\rm ri\,}\partial f\color{black}(\bar x),-1_m,0\right)+(  0_n,1_m,-1)\\
  &=(0_n,0_m,1)+\left(\sum_{i=1}^m\color{blue}{\rm ri\,}\partial f_i\color{black}(\bar x),-1_m,0\right)+(0_n,1_m,-1)\subseteq \color{blue}{\rm ri\,}\partial F\color{black}(\bar x,\bar s,f(\bar x)),
  \end{align*}
where the second equality follows from \cite[Theorem~23.8]{Roc70} and \cite[Corollary~6.6.2]{Roc70}, thanks to \eqref{Jan27}.
Thus, condition (iii) in Theorem \ref{SDP} is also satisfied. The desired conclusion now follows from Theorem~\ref{SDP}.
\end{proof}
\begin{example}\label{appLMI}
Note that $\ell_1$-norm, $\ell_2$-norm, convex quadratic functions and indicator functions of second-order cones are all LMI-representable. Using these, we can infer from Theorem \ref{corSDP} that the following functions $f$ satisfy the KL property with exponent $\frac12$ at any $\bar x$ that verifies $0 \in \color{blue}{\rm ri\,}\partial f\color{black}(\bar x)$:
  \begin{enumerate}[{\rm (i)}]
  \item {\bf Group Lasso with overlapping blocks of variables:}
  \[
  f(x) = \frac12\|Ax - b\|^2 +  \sum_{i=1}^s w_i\|x_{J_i}\|,
   \]
  where $b\in \R^p$, $A\in \R^{p\times n}$, $J_i \subseteq \{1,\ldots,n\}$ with $\bigcup_{i=1}^s J_i =\{1,\ldots,n\}$, $x_{J_i}$ is the subvector of $x$ indexed by $J_i$, and $w_i \ge 0$, $i=1,\ldots,s$. We emphasize here that $J_i\cap J_j$ can be nonempty when $i\neq j$.
  \item {\bf Least squares with products of second-order cone constraints:}
  \[
  f(x) = \frac12\|Ax - b\|^2 + \delta_{\prod_{i=1}^s \color{blue}{\rm SOC\,}_{n_i}\color{black}}(x),
   \]
  where $b\in \R^p$, $A\in \R^{p\times n}$, $x=(x_1,\ldots,x_s)\in \prod_{i=1}^s\R^{n_i}$ with \reviser{$x_i \in \mathbb{R}^{n_i}$, $i=1,\ldots,s$}, and ${\rm SOC}_{n_i}$ is the second-order cone in $\mathbb{R}^{n_i}$.
  \item {\bf Group fused Lasso \cite{Group_Fused}:}
  \[
  f(x) = \frac12\|Ax - b\|^2 +  \sum_{i=1}^s w_i \|x_{J_i}\|+  \sum_{i=2}^s \nu_i \|x_{J_i}-x_{J_{i-1}}\|,
   \]
   where $b\in \R^p$, $A\in \R^{p\times \reviser{n}}$ \reviser{with $n=rs$ for some $r \in \mathbb{N}$}, $J_i$ is an equi-partition of $\{1,\ldots,n\}$ in the sense that $\bigcup_{i=1}^s J_i =\{1,\ldots,n\}$, $J_i \cap J_{j} =\emptyset$ and $|J_i|=|J_j|=r$ for $i \neq j$, $w_i$, $\nu_i \ge 0$, $i=1,\ldots,s$.
      \end{enumerate}
\end{example}
\subsection{Sum of LMI-representable functions and the nuclear norm}\label{sec4.3.2}
\color{blue}In this section, we apply Theorem~\ref{corSDP} and Corollary~\ref{convex_general} to derive the KL exponent of the function in \eqref{fnuclear} under suitable assumptions.  It is known (see, for example \cite{ReFaPa10}) that the nuclear norm can be expressed as
\begin{equation}\label{nucnorm}
\|X\|_* = \frac12\inf_{U,V}\left\{{\rm tr}(U) + {\rm tr}(V) :\; \begin{bmatrix}
  U&X\\
  X^T&V
\end{bmatrix}\succeq 0, \, U \in {\cal S}^m,V \in {\cal S}^n\right\}
\end{equation}
for any $X\in \R^{m\times n}$.
\reviser{This fact plays an important role for our analysis later on, and shows that the nuclear norm is an SDP representable function. To the best of our knowledge, it is not known that whether the nuclear norm is LMI representable. Our analysis is an attempt to generalize our results on the sum of LMI representable functions (with strict complementarity assumption on the original function) to a large subclass of SDP representable functions that arises in many important areas such as matrix completion \cite{ReFaPa10}.}

%will be used in the proof of the following main theorem of this section.

  \begin{theorem}[{{\bf KL exponent of sum of LMI-representable functions and the nuclear norm}}]\label{nuclearLMI}
 Let $f$ be defined as in \eqref{fnuclear} and let  symmetric matrices  $A_{00}^k$, $A^k_0$, $A_{ij}^k$, \reviser{$i=1,\ldots,m$ and $j=1,\ldots,n$,} be \color{black} such that
  \[
  \color{blue}{\rm epi\,}f_k\color{black}=\left\{(X,t):\;A_{00}^k +\sum_{i=1}^m\sum_{j=1}^n A^k_{ij}X_{ij} + A_0^kt\succeq 0\right\}.
  \]
  Suppose in addition that there exist $X^{s}\in \R^{m\times n}$ and $s^s\in \R^p$ such that for $k=1,\ldots,p$,
  \[
 A_{00}^k + \sum_{i=1}^m\sum_{j=1}^n A^k_{ij}X^{s}_{ij} + A_0^ks^s_k\succ 0.
  \]
  If $\bar X\in \color{blue}{\rm dom\,}\partial f\color{black}$ satisfies $0\in\color{blue}{\rm ri\,}\partial f(\bar X)\color{black}$, then $f$ satisfies the KL property at $\bar X$ with exponent $\frac12$.
  \end{theorem}
  \color{blue}
\begin{remark}
Similar  to Theorem~\ref{corSDP}, the ``ri-condition" here is also imposed on $f$ itself, while  such a condition is   imposed on the $F$ in \eqref{F} in Theorem \ref{SDP}.
\end{remark}
  \begin{proof}
 Let  $F$  be defined as in \eqref{defF} with  the matrix variable $Z\in {\cal S}^{n+m}$ partitioned as in \eqref{Z}. Then $f(X) = \inf_{U,V}F(Z)$, thanks to \eqref{nucnorm}. \color{black} Let $r = {\rm rank}(\bar X)$ and
  \[
  \bar X=\left[P_+\ P_0\right]\left[\begin{matrix}
    \Sigma_+ &0\\
    0           &0
  \end{matrix}\right]\left[Q_+\ Q_0\right]^T \revise{=P_+\Sigma_+Q_+^T},
  \]
  be a singular value decomposition of $\bar X$, where $\Sigma_+\in \R^{r\times r}$ is a diagonal matrix whose diagonal entries are the $r$ positive singular values of $\bar X$, $\left[P_+\ P_0\right]$ is orthogonal with $P_+\in\R^{m\times r}$ and $P_0\in\R^{m\times(m-r)}$, $\left[Q_+\ Q_0\right]$ is orthogonal with $Q_+\in\R^{n\times r}$ and $Q_0\in\R^{n\times(n-r)}$. Define\footnote{\revise{When $r=0$, we set $\bar Z = 0\in {\cal S}^{m+n}$.}}
  \[
 \bar Z:= \begin{bmatrix}
 P_+\Sigma_+P_+^T &\bar X \\
 \bar X^T &Q_+\Sigma_+Q_+^T  \end{bmatrix}.
  \]
  \revise{Then $\bar Z \succeq 0$.} Now, using \cite[Theorem~23.8]{Roc70}, the definition of $F$ and  \cite[Corollary~6.6.2]{Roc70}, we have
  \begin{align}\label{riFnuclear}
  \revise{{\rm ri\,}\partial F(\bar Z)}=\left\{\frac12\begin{bmatrix}
    I_m & \Lambda\\
    \Lambda^T & I_n
  \end{bmatrix} +Y: \Lambda\in\revise{{\rm ri\,}\partial \left(\sum_{k=1}^pf_k\right)}(\bar X) {\rm \ and \ }Y\in \revise{{\rm ri\,}N_{\mathcal{S}_+^{m+n}}(\bar Z)}\right\}.
  \end{align}

  Next, since $0\in\revise{{\rm ri\,}\partial f(\bar X)}$ \revise{and the nuclear norm is continuous}, we see from \cite[Theorem~23.8]{Roc70} and  \cite[Corollary~6.6.2]{Roc70} that
    \begin{equation}\label{rinuclear}
    0\in\revise{{\rm ri\,}\partial f(\bar X)}=\revise{{\rm ri\,}\partial\left(\sum_{k=1}^pf_k\right)}(\bar X)+\revise{{\rm ri\,}\partial\|\bar X\|_*}.
    \end{equation}
  Moreover, recall from \cite[Example~2]{Wa92} and \cite[Corollary~7.6.1]{Roc70} that
  \begin{align}\label{0ri+ri}
  \revise{{\rm ri\,}\partial\|\bar X\|_*}=\left\{\left[P_+\ P_0\right]\left[\begin{matrix}
    I_r &0\\
    0           &W
  \end{matrix}\right]\left[Q_+\ Q_0\right]^T:\;W\in\R^{(m-r)\times(n-r)}, \|W\|_2<1\right\},
  \end{align}
  where $\|W\|_2$ is the operator norm of $W$, that is, the largest singular value of $W$.
  Combining \eqref{rinuclear} and \eqref{0ri+ri}, we conclude that there \reviser{exist} $C\in\revise{{\rm ri\,}\partial \left(\sum_{k=1}^pf_k\right)}(\bar X)$ and $W_0$ with $\|W_0\|_2<1$ such that
  \begin{align}\label{C_relation}
    0=C+\left[P_+\ P_0\right]\begin{bmatrix}
    I_r &0\\
    0           &W_0
  \end{bmatrix}\left[Q_+\ Q_0\right]^T=C+P_0W_0Q_0^T+P_+Q^T_+.
  \end{align}

  On the other hand, using the definition of $\bar Z$ and a direct computation, we have
  \begin{equation}\label{zeig}
  \bar Z = \underbrace{\begin{bmatrix}
  \frac{1}{\sqrt{2}}P_+ &P_0 &0    &\frac{1}{\sqrt{2}}P_+\\
  \frac{1}{\sqrt{2}}Q_+ &0   &Q_0  &-\frac{1}{\sqrt{2}}Q_+
 \end{bmatrix}}_{\widehat P}
 \begin{bmatrix}
   2\Sigma_+ & 0 &0 & 0\\
   0 & 0  &  0& 0   \\
   0 &  0&0 &   0     \\
    0 &  0 & 0 &   0
 \end{bmatrix}
  \begin{bmatrix}
  \frac{1}{\sqrt{2}}P_+ &P_0 &0 &\frac{1}{\sqrt{2}}P_+\\
  \frac{1}{\sqrt{2}}Q_+ &0   &Q_0  &-\frac{1}{\sqrt{2}}Q_+
 \end{bmatrix}^T.
  \end{equation}
  Note that $\widehat P^T\widehat P = \widehat P\widehat P^T \reviser{= I_{m+n}}$, meaning that \eqref{zeig} is an eigenvalue decomposition of $\bar Z$.
  Thus, we can compute that
  \[
  \begin{aligned}
    &\revise{{\rm ri\,}N_{\mathcal{S}_+^{m+n}}(\bar Z)}={\rm ri}\left[(-\mathcal{S}_+^{m+n})\cap\left\{\bar Z\right\}^\perp\right]=\widehat P
 \begin{bmatrix}
   0&0\\
   0&-\revise{{\rm int\,}\mathcal{S}_{+}^{m+n-r}}
 \end{bmatrix}
  \widehat P^T\\&\ni
   \begin{bmatrix}
   \frac{1}{\sqrt{2}}P_+ &P_0 &0 &\frac{1}{\sqrt{2}}P_+\\
   \frac{1}{\sqrt{2}}Q_+ &0   &Q_0  &-\frac{1}{\sqrt{2}}Q_+
 \end{bmatrix}
 \begin{bmatrix}
   0 & 0 &0 & 0\\
   0 & -\frac12 \reviser{I_{m-r}}  & \frac12W_0& 0   \\
   0 &  \frac12W_0^T &-\frac12 \reviser{I_{n-r}}  &   0     \\
    0 &  0 & 0 &   - \reviser{I_r}
 \end{bmatrix}
   \begin{bmatrix}
   \frac{1}{\sqrt{2}}P_+^T & \frac{1}{\sqrt{2}}Q_+^T\\
   P_0^T & 0\\
   0 & Q_0^T\\
   \frac{1}{\sqrt{2}}P_+^T & -\frac{1}{\sqrt{2}}Q_+^T
 \end{bmatrix}\\
 & =\frac12\begin{bmatrix}
   -I_m & -C\\
   -C^T  & -I_n
 \end{bmatrix},
  \end{aligned}
  \]
  where the inclusion holds because $\|W_0\|_2<1$, and the last equality follows from \eqref{C_relation} and a direct computation.
 This together with \eqref{riFnuclear} and the definition of $C$ implies  that $0\in\revise{{\rm ri\,}\partial F(\bar Z)}$. Moreover, one can see that $F$ is the sum of $p+1$ \revise{proper closed} LMI-representable functions and the Slater's condition required in Theorem~\ref{corSDP} holds.
 Thus, we conclude from Theorem~\ref{corSDP} that $F$ in \eqref{defF} has KL property at
  $\bar Z$ with exponent $\frac12$.

Finally, recall that
for the $F$ defined in \eqref{defF}, we have
 \[
 \inf_{U,V} F(Z)=f(X) {\rm \ and\ }\Argmin_{U,V} F\left(\begin{bmatrix}
   U&\bar X\\
   \bar X^T& V
 \end{bmatrix}\right)=\left\{(P_+\Sigma_+P_+^T ,Q_+\Sigma_+Q_+^T)\right\}.\footnote{\revise{When $r = 0$, this set is $\{(0,0)\}$ and $\bar Z = 0$.}}
 \]
\revise{These} together with Corollary~\ref{convex_general} and the fact that the KL exponent of $F$ at $\bar Z$ is $\frac12$ shows that $f$ satisfies the KL property at $\bar X$ with exponent $\frac12$.
 \end{proof}\color{black}

\begin{remark}\label{appLMInu}
In \cite[Proposition~12]{ZhSo2017}, it was shown that if $\ell:\R^p\to \R$ is strongly convex on any compact convex set with locally Lipschitz gradient and ${\cal A}:\R^{m\times n}\to \R^p$ is a linear map, then the function
\[
f(X) = \ell({\cal A}X) + \|X\|_*
\]
satisfies the KL property with exponent $\frac12$ at any $\bar X$ that verifies $0 \in \revise{{\rm ri\,}\partial f(\bar X)}$. In particular, the loss function $X\mapsto \ell({\cal A}X)$ is smooth. The more general case where the nuclear norm is replaced by a general spectral function was considered in \cite[Theorem~3.12]{Cui_Ding_Zhao}, and a sufficient condition involving the relative interior of {\em the subdifferential of the conjugate of the spectral function} was proposed in \cite[Proposition~3.13]{Cui_Ding_Zhao}, which, in general, is different from the regularity condition $0 \in \revise{{\rm ri\,}\partial f(\bar X)}$.

On the other hand, using our Theorem~\ref{nuclearLMI}, we can deduce the KL exponent of functions in the form of \eqref{fnuclear} at points $\bar X$ satisfying the condition $0 \in \revise{{\rm ri\,}\partial f(\bar X)}$, but with a different set of conditions on the loss function. For instance, one can prove using Theorem~\ref{nuclearLMI} that the following functions $f$ satisfy the KL property with exponent $\frac12$ at a point $\bar X$ verifying $0 \in \revise{{\rm ri\,}\partial f(\bar X)}$:
\begin{enumerate}[{\rm (i)}]
\item $f(X) = \frac12\|{\cal A}X - b\|^2 + \mu \sum_{i,j}|X_{ij}| + \nu \|X\|_*$, where $\mu>0$ and $\nu > 0$, $b\in \R^p$ and ${\cal A}:\R^{m\times n}\to \R^p$ is a linear map.
  \item $f(X) = \|{\cal A}X - b\| + \mu \sum_{i,j}|X_{ij}| + \nu \|X\|_*$, where $\mu>0$ and $\nu > 0$, $b\in \R^p$ and ${\cal A}:\R^{m\times n}\to \R^p$ is a linear map.
\end{enumerate}

In view of \cite[Theorem~3.12]{Cui_Ding_Zhao}, it would be of interest to extend Theorem~\ref{nuclearLMI} to cover more general spectral functions. However, since our analysis in this subsection is based on LMI or SDP representability, it is not clear how this can be achieved at this moment. This would be a potential important future research direction.
\end{remark}

\color{blue}
\begin{remark}[{{\bf Discussion of the relative interior conditions}}]\label{newremark4.4}
In Theorems~\ref{SDP}, \ref{corSDP} and \ref{nuclearLMI}, the conclusions of KL exponent being $1/2$ were derived under relative interior conditions. If these relative interior  conditions were dropped, then the corresponding conclusions could fail, in general. For example, in \cite[equation (53)]{ZhSo2017}, the authors provided an example of $\widetilde f(X):=f_1(X)+\|X\|_*$ for $X \in \mathbb{R}^{2 \times 2}$, where $f_1$ is a convex quadratic function on $\mathbb{R}^{2 \times 2}$, and showed that $0 \notin {\rm ri} \, \partial  \widetilde f(\overline{X})$ for some $\overline{X} \in \mathbb{R}^{2 \times 2}$ and the first-order error bound is not satisfied at $\overline{X}$. Recalling \cite[Theorem~5]{BoNgPeSu17} and \cite[Corollary~3.6]{DrLewis18}, this means that $\widetilde f$ cannot have a KL exponent of $\frac12$ at $\overline{X}$.

We also would like to point out that, when the relative interior condition fails, one can follow the approach in Section~\ref{sec4.1} and the general error bound result for ill-posed semidefinite programs \cite{Sturm00,DrLiWo17} to derive a KL exponent that depends on the degree of singularity of a certain semidefinite system in the lifted representation. In general, this KL exponent will approach $1$ quickly as the dimension grows, which can be of less interest. For simplicity, we do not discuss this in detail.
\end{remark}
\color{black}

\subsection{Convex models with $C^2$-cone reducible structure}\label{sec4.2}

In this section, we explore the KL exponent of functions that involve $C^2$-cone reducible structures.
Our first theorem concerns the sum of the {\em support} function of a $C^2$-cone reducible closed convex set and a specially structured smooth convex function. In the theorem, we will also make use of the so-called bounded linear regularity condition \cite[Definition~5.6]{BaBo96}. Recall that $\{\D_1,\D_2\}$ is said to be boundedly linearly regular at $\bar x \in \D_1 \cap \D_2$ if for any bounded neighborhood $\mathfrak{U}$ of $\bar x$, there exists $c>0$ such that
\[
\dist(x,\D_1 \cap \D_2) \le c [\dist(x,\D_1)+\dist(x,\D_2)] \mbox{ for all }
 x \in \mathfrak{U}.
\]
It is known that if $\D_1$ and $\D_2$ are both polyhedral, then $\{\D_1,\D_2\}$ is boundedly linearly regular at any $\bar x \in \D_1\cap \D_2$; moreover, if $\D_1$ is polyhedral and $\D_1 \cap \revise{{\rm ri\,}\D_2} \neq \emptyset$, then $\{\D_1,\D_2\}$ is also boundedly \revise{linearly} regular
at any $\bar x \in \D_1 \cap \D_2$; see \cite[Corollary~3]{BaBoLi99}.

\begin{theorem}[{{\bf Composite convex models with $C^2$-cone reducible structure}}]\label{thmC2reduce}
Let $\ell: \mathbb{Y}\to {\R}$ be a function that is strongly convex on any compact convex set
and has locally Lipschitz gradient, $\mathcal{A}:\mathbb{X} \to \mathbb{Y}$ be a linear map, and $v\in \mathbb{X}$. Consider the function
\[
h(x):=\ell(\mathcal{A}x)+\langle v, x \rangle + \sigma_{\frak D}(x)
\]
with ${\frak D}$ being a \revise{nonempty} $C^2$-cone reducible closed convex set.
Suppose $0\in\partial h(\bar x)$. Then, one has $$\bar x \in N_{\frak D}(-{\cal A}^*\nabla \ell(\mathcal{A} \bar x)-v).$$ If we assume in addition that $\{\mathcal{A}^{-1}\{{\cal A} \bar x\}, \, N_{\frak D}(-{\cal A}^*\nabla \ell(\mathcal{A} \bar x)-v)\}$ is boundedly linearly regular at $\bar x$, then $h$ satisfies the KL property at $\bar x$ with exponent $\frac12$.
\end{theorem}

\begin{proof}
Since $0 \in \partial h(\bar x)$, we see from \cite[Exercise~8.8]{RocWets98} that
\[
\bar w:=-\mathcal{A}^* \nabla \ell (\mathcal{A}\bar x)- v \in \partial \sigma_{\frak D}(\bar x)=\partial \delta_{\frak D}^*(\bar x)=(\partial \delta_{\frak D})^{-1}(\bar x),
\]
where the last equality follows from \cite[Proposition~11.3]{RocWets98}.
This implies  $\bar x \in \partial \delta_{\frak D} (\bar w)=N_{\frak D}(\bar w)$.

We now assume in addition the bounded linear regularity condition and prove the alleged KL property. First, since ${\frak D}$ is a $C^2$-cone reducible closed convex set, there exists $\tilde\rho>0$ and a mapping $\Theta:\mathbb{X} \rightarrow \mathbb{V}$ which is twice continuously differentiable on $B(\bar w, \tilde\rho)$ and a closed convex pointed cone $K \subseteq \mathbb{V}$ such that $\Theta(\bar w)=0$, $D \Theta(\bar w)$ is onto and
${\frak D} \cap B(\bar w, \tilde\rho) = \{w: \Theta(w) \in K\} \cap B(\bar w, \tilde\rho)$.

Fix any $\rho \in (0,\tilde \rho)$ so that $D\Theta(w)$ is onto whenever $w\in B(\bar w,\rho)$.
Then, we have from \cite[Exercise~10.7]{RocWets98} that
\begin{equation}\label{conerel}
N_{\frak D}(w)= D \Theta(w)^* N_K(\Theta(w))\ \ \mbox{ for all } w \in B(\bar w, \rho).
\end{equation}
Now, fix any $\delta>0$. Take $w \in {\frak D} \cap B(\bar w, \rho)$ and $x \in N_{\frak D}(w) \cap B(\bar x,\delta)$.
Then $x=D \Theta(w)^*u_x$ for some $u_x \in N_K(\Theta(w))$ according to \eqref{conerel}. For such a $u_x$, one can observe that
\[
D \Theta(\bar w)^* u_x \in D \Theta(\bar w)^* N_K(\Theta(w)) \subseteq D \Theta(\bar w)^* K^{\circ}=D \Theta(\bar w)^* N_K(\Theta(\bar w))=N_{\frak D}(\bar w),
\]
where $K^\circ$ is the polar of $K$, the set inclusion follows from the definition of normal cone and the fact that $K$ is a closed convex cone, the first equality holds because $\Theta(\bar w) = 0$ and the last equality follows from \eqref{conerel}. Thus, for any $w \in {\frak D} \cap B(\bar w, \rho)$ and $x \in N_{\frak D}(w) \cap B(\bar x,\delta)$, we have
\begin{equation}\label{eq:use1}
\dist(x, N_{\frak D}(\bar w)) \le \|x-D \Theta(\bar w)^* u_x\| = \|D \Theta(w)^*  u_x-D \Theta(\bar w)^* u_x\| \le  L \|u_x\| \|w-\bar w\|,
\end{equation}
where $L$ is the Lipschitz continuity modulus of $D\Theta$ over the set $B(\bar w, \rho)$, which is finite because $\Theta$ is twice continuously differentiable.

Next, for each $z\in B(\bar w, \rho)$, define the linear map
\[
\mathcal{W}(z)=\big(D \Theta(z) D \Theta(z)^*\big)^{-1}D \Theta(z).
\]
Then $\mathcal{W}$ is continuously differentiable on $B(\bar w, \rho)$ because $\Theta$ is twice continuously differentiable on $B(\bar w, \rho)$ with surjective gradient map. Moreover, for any $w \in {\frak D} \cap B(\bar w, \rho)$ and $x \in N_{\frak D}(w) \cap B(\bar x,\delta)$, it follows from the definition of $u_x$ that $[\mathcal{W}(w)](x) = u_x$.
Let $M$ be the Lipschitz continuity modulus of $w\mapsto \mathcal{W}(w)$ on $B(\bar w, \rho)$, which is finite because $\mathcal{W}$ is continuously differentiable on $B(\bar w, \rho)$. Then we have for any $w \in {\frak D} \cap B(\bar w, \rho)$ and $x \in N_{\frak D}(w) \cap B(\bar x,\delta)$ that
\[
\begin{aligned}
&\|u_x - u_{\bar x}\|  =  \|[\mathcal{W}(w)](x)-[\mathcal{W}(\bar w)](\bar x)\| \\
&\le  \|[\mathcal{W}(w)](x)-[\mathcal{W}(\bar w)](x)\|+ \|[\mathcal{W}(\bar w)](x)-[\mathcal{W}(\bar w)](\bar x)\| \\
& \le  M \|x\| \, \|w-\bar w\| + \|\mathcal{W}(\bar w)\| \|x-\bar x\|\\
& \le  M \rho (\|\bar x\|+\|x-\bar x\|)  + \|\mathcal{W}(\bar w)\| \|x-\bar x\|,
\end{aligned}
\]
where the last inequality follows from triangle inequality and the fact that $w\in B(\bar w, \rho)$.
In particular, $\|u_x\| \le \|u_{\bar x}\| + M \rho(\|\bar x\|+\delta) + \|\mathcal{W}(\bar w)\| \delta=:\kappa$.
This together with \eqref{eq:use1} implies that
\[
N_{\frak D}(w)\cap B(\bar x,\delta) \subseteq  N_{\frak D}(\bar w) + \kappa L \, \|w-\bar w\| B(0,1)\ \ \mbox{ for all }  w\in B(\bar w, \rho).
\]
This means that the mapping \revise{$w \mapsto N_{\frak D}(w)$} is calm at $\bar w$ with respect to $\bar x$; see \cite[Page~182]{DonRoc09}. Thus, according to \cite[Theorem~3H.3]{DonRoc09}, the mapping \revise{$x \mapsto (N_{\frak D})^{-1}(x)$} is metrically subregular at $\bar x$ with respect to $\bar w$; see \cite[Page~183]{DonRoc09} for the definition. Noting also that $\partial \sigma_{\frak D} = (N_{\frak D})^{-1}$ according to \cite[Example~11.4]{RocWets98}, we then deduce from \cite[Theorem~3.3]{ArtachoGeo08} that there exist $\delta'\in (0,\delta)$ and $c_0>0$ such that
\begin{equation}\label{eq:metric_sub_sigma}
\sigma_{\frak D}(x)-\sigma_{\frak D}(\bar x)-\langle \bar w, x-\bar x\rangle \ge c_0 \, \dist( x, (\partial \sigma_{\frak D})^{-1}(\bar w))^2=c_0 \, \dist(x, N_{\frak D}(\bar w))^2
\end{equation}
whenever $\|x-\bar x\|\le \delta'$. \reviser{We now follow a similar line of argument used in \cite[Theorem 2]{ZhSo2017} and \cite[Theorem 4.2]{DrLewis18} to show the desired conclusion.} Observe that
\[
\reviser{\begin{aligned}
&\Argmin h=\{z: 0 \in \partial h(z)\} \\
& =  \{z: \mathcal{A}z=\mathcal{A} \bar x \mbox{ and } -\mathcal{A}^* \nabla \ell (\mathcal{A} z)- v \in (N_{\frak D})^{-1}(z)\} \\
%& =  \{z: \mathcal{A}z=\mathcal{A} \bar x \mbox{ and } -\mathcal{A}^* \nabla \ell (\mathcal{A} \bar x)- v \in (N_{\frak D})^{-1}(z)\} \\
& =  \{z: \mathcal{A}z =\mathcal{A} \bar x \mbox{ and } z \in N_{\frak D}(-\mathcal{A}^* \nabla \ell (\mathcal{A} \bar x)- v)\}.
\end{aligned}}
\]
Then it follows that for any bounded convex neighborhood $\mathfrak{U}$ of $\bar x$ with $\mathfrak{U}\subseteq B(\bar x,\delta')$, there exists $c_1>0$ such that for any $z \in \mathfrak{U}$,
\begin{equation}\label{ebineq}
\begin{aligned}
&\dist(z,\Argmin h) = \dist(z, \mathcal{A}^{-1}\{{\cal A}\bar x\} \cap N_{\frak D}(\bar w)) \\
& \overset{\rm (a)}\le \alpha [\dist(z, \mathcal{A}^{-1}\{{\cal A}\bar x\})+ \dist(z,N_{\frak D}(\bar w))] \\
& \overset{\rm (b)}\le \alpha [c_1 \, \|{\cal A}\bar x- \mathcal{A}z\| + \dist(z,N_{\frak D}(\bar w))] \\
& \overset{\rm (c)}\le \alpha \left[c_1 \, \|{\cal A}\bar x- \mathcal{A}z\|+ c_0^{-\frac{1}{2}} \,\sqrt{\sigma_{\frak D}(z)-\sigma_{\frak D}(\bar x)-\langle \bar w, z-\bar x\rangle}\right];
\end{aligned}
\end{equation}
here, (a) holds for some $\alpha > 0$ because of the bounded linear regularity assumption, (b) holds for some $c_1 > 0$ thanks to the Hoffman error bound, and (c) follows from \eqref{eq:metric_sub_sigma}.
Now, as $\ell$ is strongly convex on compact convex sets, there exists $\beta>0$ such that for all $z \in \mathfrak{U}$, we have
\[
\beta \|\mathcal{A} \bar x - \mathcal{A}z\|^2 \le \ell(\mathcal{A}z)-\ell(\mathcal{A}\bar x)- \langle \mathcal{A}^*\nabla \ell(\mathcal{A}\bar x), z-\bar x\rangle.
\]
Combining this with \eqref{ebineq}, we have for any $z \in \mathfrak{U}$ that
\[
\begin{aligned}
&\dist(z,\Argmin h)  \le   \alpha \bigg( c_1 \, \|{\cal A}\bar x- \mathcal{A}z\|+ c_0^{-\frac{1}{2}} \,\sqrt{\sigma_{\frak D}(z)-\sigma_{\frak D}(\bar x)-\langle \bar w, z-\bar x\rangle}  \bigg) \\
& \le  \alpha \bigg( c_1 \, \beta^{-\frac{1}{2}} \sqrt{\ell(\mathcal{A}z)-\ell(\mathcal{A}\bar x)- \langle \mathcal{A}^*\nabla \ell(\mathcal{A}\bar x), z-\bar x\rangle}+ c_0^{-\frac{1}{2}} \,\sqrt{\sigma_{\frak D}(z)-\sigma_{\frak D}(\bar x)-\langle \bar w, z-\bar x\rangle}  \bigg)
\end{aligned}
\]
Note that $\sqrt{a}+\sqrt{b} \le \sqrt{2} \sqrt{a+b}$ for $a,b \ge 0$, and
\[
h(z)-h(\bar x)=\ell(\mathcal{A}z)-\ell(\mathcal{A}\bar x)- \langle \mathcal{A}^*\nabla \ell(\mathcal{A}\bar x), z-\bar x\rangle + \sigma_{\frak D}(z)-\sigma_{\frak D}(\bar x)-\langle \bar w, z-\bar x\rangle.
\]
Thus, there exists $c>0$ such that for all $z \in \mathfrak{U}$,
$\dist(z,\Argmin h)  \le c \, \sqrt{h(z)-h(\bar x)}$.
Combining this with \cite[Theorem~5]{BoNgPeSu17}, we conclude that $h$ satisfies the KL property at $\bar x$ with exponent $\frac12$.
\end{proof}

%\reviser{In passing, we note that the boundedly linear regularity assumption in Theorem 4.4 holds when $\mathcal{A}^{-1}\{{\cal A}\bar X\} \cap  {\rm ri} \,   \left(N_{\frak D}(-\mathcal{A}^*\nabla \ell(\mathcal{A} \bar X)-V)\right) \neq \emptyset$ or $N_{\frak D}(-{\cal A}^*\nabla \ell(\mathcal{A} \bar x)-v)$ is a polyhedral set. The latter condition is automatically true for example when $\frak D$ is a polyhedral or a ball.}

As a corollary of the preceding theorem, we consider the KL exponent of a class of gauge regularized optimization problems. Recall that a convex function $\gamma : \mathbb{X} \rightarrow  \revise{\R\cup\{\infty\}}$ is called a gauge if it is nonnegative, positively homogeneous, and vanishes at the origin. It is clear that any norm is a gauge. In the next corollary, we make explicit use of the gauge structure and replace the relative interior condition in Theorem~\ref{thmC2reduce} by one involving the so-called polar gauge.
Recall from \cite[Proposition~2.1(iii)]{FriMaPong14} that for a gauge $\gamma$, its polar can be given by $\gamma^{\circ}(x)=\sup_z\{\langle x, z\rangle: \gamma(z) \le 1\}$; moreover, polar of norms are their corresponding dual norms.

\begin{corollary}\label{C2gauge}
\revise{Let $f$ be defined as in \eqref{C2h}. }
%Let $\gamma$ be a closed gauge whose polar $\gamma^{\circ}$ is $C^2$-cone reducible. Let $\ell:\mathbb{Y}\to \R$ be a function that is strongly convex on any compact convex set
%and has locally Lipschitz gradient, $\mathcal{A}:\mathbb{X} \to \mathbb{Y}$ be a linear map, and $v\in \mathbb{X}$. Consider the function
%\[
%h(x):=\ell(\mathcal{\mathcal{A}} x)+\langle v,x \rangle+ \gamma(x).
%\]
Suppose that $0\in\partial \revise{f}(\bar x)$ and $\gamma(\bar x) > 0$. Then $\gamma^\circ(-\mathcal{A}^*\nabla \ell(\mathcal{A} \bar x)-v)=1$. Suppose in addition that $-\mathcal{A}^*\nabla \ell(\mathcal{A} \bar x)-v\in \revise{{\rm dom\,} \partial \gamma^\circ}$ and the following relative interior condition holds:
\begin{equation}\label{eq:regularity1}
\revise{ \mathcal{A}^{-1}\{{\cal A}\bar x\} \cap \left( \bigcup_{\lambda > 0} \lambda\left( {\rm ri\,}\partial \gamma^{\circ}(-\mathcal{A}^* \nabla \ell (\mathcal{A}\bar x)- v )\right)\right) \neq \emptyset.}
\end{equation}
Then $\revise{f}$ satisfies the KL property at $\bar x$ with exponent $\frac12$.
\end{corollary}
\begin{proof}
Since $0 \in \partial \revise{f}(\bar x)$, we see from \cite[Exercise~8.8]{RocWets98} that
\[
\bar w:=-\mathcal{A}^* \nabla \ell (\mathcal{A}\bar x)- v \in \partial \gamma(\bar x).
\]
Since we have from \cite[Proposition~2.1(iv)]{FriMaPong14} that $\gamma^* = \delta_{\frak C}$ with ${\frak C}=\{x:\; \gamma^\circ(x)\le 1\}$, we conclude from \eqref{Young} that $\gamma^\circ(\bar w) \le 1$ and $\gamma(\bar x) = \langle\bar x,\bar w\rangle$. Since $\gamma(\bar x)>0$, we also have from $\gamma(\bar x) = \langle\bar x,\bar w\rangle$ and \cite[Proposition~2.1(iii)]{FriMaPong14} that
\[
1 = \frac{\langle\bar x,\bar w\rangle}{\gamma(\bar x)}\le \sup_{z}\{\langle \bar w,z\rangle:\; \gamma(z)\le 1\} = \gamma^\circ(\bar w).
\]
Thus, it holds that $\gamma^\circ(\bar w)=1$.

Next, suppose in addition that $\bar w\in \revise{{\rm dom\,} \partial \gamma^\circ}$ and \eqref{eq:regularity1} holds. \revise{Let $F(x,t)$ be defined as in \eqref{C2F}}. Observe that
\[
F(x,t)=\ell(\tilde{\mathcal{A}}(x,t))+\langle (v,1), (x,t) \rangle  + \sigma_{\mathfrak{D}^{\circ}}(x,t)
\]
where $\tilde{\mathcal{A}}(x,t):=\mathcal{A} x$ and $\mathfrak{D}^\circ$ is the polar of $\mathfrak{D}$, which is given by $\mathfrak{D}^{\circ}=\{(x,t): \gamma^{\circ}(x)+t \le 0\}$ according to the proof of \cite[Theorem~15.4]{Roc70}.
From our assumption, the set $\{(x,t):\; \gamma^{\circ}(x)\le t\}$ is a $C^2$-cone reducible closed convex set, which implies that $\mathfrak{D}^{\circ}$ is also $C^2$-cone reducible. Now, observe from \cite[Theorem~23.7]{Roc70} that for any $(u,s)\in \revise{{\rm dom\,}\partial \gamma^\circ}\times \R$ satisfying $\gamma^\circ(u) + s = 0$, we have
\[
N_{\mathfrak{D}^{\circ}}(u,s)= {\rm cl}\left(\bigcup_{\lambda \ge 0} \lambda \big(\revise{\partial \gamma^{\circ}(u)},1\big)\right),
\]
which together with \cite[Theorem~6.3]{Roc70} and \cite[Corollary~6.8.1]{Roc70} gives
\begin{equation*}%\label{normalcone}
\revise{{\rm ri\,}N_{\mathfrak{D}^{\circ}}(u,s)}= \bigcup_{\lambda > 0} \lambda \big(\revise{{\rm ri\,}\partial \gamma^{\circ}}(u),1\big).
\end{equation*}
Applying this relation with $(u,s) = (\bar w,-\gamma^\circ(\bar w)) = (\bar w,-1)$ together with the relative interior condition \eqref{eq:regularity1} \revise{shows that}
\begin{equation*}
\revise{\big({\mathcal{A}}^{-1}\{{\cal A}\bar x\} \times \R \big) \cap {\rm ri\,}N_{\mathfrak{D}^{\circ}}(\bar w,-1) \neq \emptyset.}
\end{equation*}
In view of this and \cite[Corollary~3]{BaBoLi99}, we obtain that $\{\big({\mathcal{A}}^{-1}\{{\cal A}\bar x\} \times \R \big),N_{\mathfrak{D}^{\circ}}(\bar w,-1)\}$ is boundedly linearly regular.
It follows from Theorem~\ref{thmC2reduce} that $F$ satisfies the KL property at $(\bar x,\gamma(\bar x))$ with exponent $\frac12$. Since
$\revise{f}(x)=\inf_{t \in \R} F(x,t)$, we see from Corollary \ref{convex_general} that $\revise{f}$ satisfies the KL property at $\bar x$ with exponent $\frac12$.
\end{proof}

While checking $C^2$-cone reducibility directly using the definition can be difficult, a sufficient condition related to standard constraint qualifications was given in \cite[Proposition 3.2]{Shapiro}.\footnote{The quoted result is for $C^1$-cone reducibility. However, it is apparent from the proof how to adapt the result for $C^2$-cone reducibility.} Specifically, let
$K \subseteq \mathbb{Y}$ be a $C^2$-cone reducible closed convex set and $G:\mathbb{X} \to \mathbb{Y}$ be a twice continuously differentiable function. If $G(\bar x)\in K$ and $G$ is nondegenerate at $\bar x$ in the sense that
\begin{equation}\label{cond:nondegenerate}
DG(\bar x){\mathbb{X}}+ \left(T_K(G(\bar x))\cap \big[-T_K(G(\bar x))\big]\right)=\mathbb{Y},
\end{equation}
then $G^{-1}(K)$ is a $C^2$-cone reducible set. In particular, if $g_1,\ldots,g_m$ are $C^2$ functions with
$\{\nabla g_i(\bar x): i \in I(\bar x)\}$ being linearly independent, where $I(\bar x):=\{i:\; g_i(\bar x) = 0\}$, then the set $\{x: g_i(x) \le 0, i=1,\ldots,m\}$ is $C^2$-cone reducible at $\bar x$.

We will now present a few concrete examples of functions to which Theorem \ref{thmC2reduce} and Corollary \ref{C2gauge} can be applied, taking advantage of the aforementioned sufficient condition \eqref{cond:nondegenerate} for checking $C^2$-cone reducibility.
\begin{example}\label{ExC2}
Let $\ell:\mathbb{Y}\to \R$ be a function that is strongly convex on any compact convex set
and has locally Lipschitz gradient, $\mathcal{A}:\mathbb{X} \to \mathbb{Y}$ be a linear map, and $v\in \mathbb{X}$.
\begin{itemize}
\item[{\rm (i)}]  {\bf (Entropy-like regularization)} Let $\mathbb{X}=\mathbb{R}^n$ and $\mathbb{Y}=\mathbb{R}^m$. Denote
\[
p(x)=\begin{cases}
\displaystyle \sum_{i=1}^n x_i \log(x_i)- (\sum_{i=1}^n x_i) \log (\sum_{i=1}^nx_i)& \mbox{ if } x \in \mathbb{R}^n_+ ,\\
\reviser{\infty} & \mbox{ else},
\end{cases}
\]
with the convention that $0 \log 0=0$.
This function is proper closed convex and arises in the study of maximum entropy optimization \cite[Example 11.12]{RocWets98}. We claim that
$\revise{f}(x)=\ell(\mathcal{A}x)+\langle v,x\rangle+p(x)$  satisfies the KL property  with exponent $\frac12$ at  any stationary point $\bar x$. To see this,
recall from \cite[Example 11.12]{RocWets98} that
\[
p(x) = \sigma_{\frak D}(x),\ \ {\it where}\ \ {\frak D}=\{x\in \R^n: g(x) \le 0\},
\]
and $g(x)=\log(\sum_{i=1}^n e^{x_i})$. Then we have from Theorem~\ref{thmC2reduce} that $-\mathcal{A}^*\nabla \ell(\mathcal{A} \bar x)-v\in {\frak D}$. Moreover, for all $x\in {\frak D}$, $\nabla g(x)=(\frac{e^{x_1}}{\sum_{i=1}^n e^{x_i}},\ldots,\frac{e^{x_n}}{\sum_{i=1}^n e^{x_i}}) \neq 0$. Thus, in view of the discussion preceding this example, ${\frak D}$ is $C^2$-cone reducible. Finally, notice that for any $x\in {\frak D}$, the set
\[
N_{\frak D}(x)=\begin{cases}
  \bigcup_{\lambda\ge 0} \lambda \{\nabla g(x)\} & \mbox{ if }g(x) = 0,\\
  \{0\} & \mbox{ if }g(x) < 0,
\end{cases}
\]
is polyhedral, and hence, $\{\mathcal{A}^{-1}\{{\cal A}\bar x\},N_{\frak D}(-\mathcal{A}^*\nabla \ell(\mathcal{A} \bar x)-v)\}$ is boundedly linearly regular \cite[Corollary~5.26]{BaBo96}. So, Theorem \ref{thmC2reduce} implies that $\revise{f}$ satisfies the KL property  with exponent $\frac12$ at  any stationary point $\bar x$.

%\item[{\rm (ii)}] {\bf (Rotated second-order cone constraints)}
%Let $n_i \in \mathbb{N}$, $i=1,\ldots,s$ with $\sum_{i=1}^sn_i=n$. Let $\mathbb{X}=\mathbb{R}^n$, $\mathbb{Y}=\mathbb{R}^m$ and $M_i \in \mathbb{R}^{n_i \times n_i}$ be invertible matrices, $i=1,\ldots,s$. Consider  $\revise{f}(x)=\ell(\mathcal{A}x)+\langle v,x\rangle+\delta_{\prod_{i=1}^s {\frak C}_i}(x)$ where ${\frak C}_i=\{M_i^T z: z
% \in {\rm SOC}_{n_i}\}$ and ${\rm SOC}_{n_i}$ is the second-order cone in $\mathbb{R}^{n_i}$. Then $\revise{f}$ satisfies the KL property with exponent $\frac12$ at any stationary point $\bar x$ under the relative interior condition
%\[
% \mathcal{A}^{-1}\{{\cal A}\bar x\} \cap \revise{{\rm ri} \, \left(N_{\frak D}(-\mathcal{A}^*\nabla \ell(\mathcal{A} \bar x)-v)\right)} \neq \emptyset,
%\]
%where $\frak D=\{(x_1,\ldots,x_s) \in \prod_{i=1}^s \mathbb{R}^{n_i}: M_ix_i \in -{\rm SOC}_{n_i}\}$.
%To see this, note that
%$\delta_{\prod_{i=1}^s {\frak C}_i}= \sigma_{\frak D}$. Moreover, notice that $-\prod_{i=1}^s{\rm SOC}_{n_i}$ is $C^2$-cone reducible and each $M_i$ is invertible. Thus, $\frak D$ is also $C^2$-cone reducible at $\bar x$ and the conclusion follows from Theorem \ref{thmC2reduce}.

\item[{\rm \reviser{(ii)}}] {\bf (Positive semidefinite cone constraints)} Let $\mathbb{X}={\cal S}^n$ and $\mathbb{Y}=\mathbb{R}^m$. Using the $C^2$-cone reducibility of ${\cal S}^n_+$, one can see that  $\revise{f}(X)=\ell(\mathcal{A}X)+ \langle V, X\rangle +\delta_{{\cal S}^n_+}(X)$ satisfies the KL property with exponent $\frac12$ at any stationary point $\bar X$ under the relative interior condition
 $\mathcal{A}^{-1}\{{\cal A}\bar X\} \cap \revise{\color{blue} {\rm ri} \,   \left(N_{-S^{n}_+}(-\mathcal{A}^*\nabla \ell(\mathcal{A} \bar X)-V)\right)} \neq \emptyset$. We note that this result has also been derived in \cite{Cui_Sun_Toh} via a different approach.

\item[{\rm \reviser{(iii)}}] {\bf (Schatten $p$-norm regularization)} Let $\mathbb{X}={\cal S}^n$ and $\mathbb{Y}=\mathbb{R}^m$. Let $p \in [1,2] \cup \reviser{\{\infty\}}$ and consider the following optimization model with  Schatten $p$-norm regularization:
\[
\revise{f}(X)=\ell(\mathcal{A} X)+\langle V,X \rangle+ \tau \|X\|_{p}\ \ \  \mbox{    for all } X \in {\cal S}^n,
\]
where $\|X\|_p=\big(\sum_{i=1}^n |\lambda_i(X)|^p\big)^{\frac{1}{p}}$ and $\lambda_n(X)\ge \lambda_{n-1}(X)\ge \cdots\ge \lambda_1(X)$ are eigenvalues of $X$. The dual norm of $\|\cdot\|_p$ is the Schatten $q$-norm with $\frac{1}{p}+\frac{1}{q}=1$ where $q\in\{1\} \cup \reviser{[2,\infty]}$. Let $g(\lambda_1,\ldots,\lambda_n)=\big(\sum_{i=1}^n |\lambda_i|^q\big)^{\frac{1}{q}}$.
It can be directly verified that $g$ is convex, symmetric and $C^2$-cone reducible. So, $\|X\|_q=g(\lambda(X))$ is also $C^2$-cone reducible \cite[Proposition 3.2]{Cui_Ding_Zhao}. Thus, from Corollary \ref{C2gauge}, $\revise{f}$ satisfies the KL property with exponent $\frac12$ at any nonzero stationary point $\bar X$ under the relative interior condition \eqref{eq:regularity1} with $\gamma(X)=\|X\|_p$.
%It is also worth noting that this result can also be obtained by using the facts that the subdifferential mapping of $\|\cdot\|_p$ is metrically subregular for $p \in [1,2] \cup \{+\infty\}$ \cite{Zhou_Zhang_So} and the spectral mapping
%preserves the metric subregularity of the subdifferential mapping \cite[Proposition 3.9]{Cui_Ding_Zhao}.
\end{itemize}
\end{example}

\section{\revise{ KL exponents for some nonconvex models}}\label{sec5}

\subsection{Difference-of-convex functions}
\revise{In this section, we study a relationship between the KL exponents of the difference-of-convex (DC) function $f$ in \eqref{DC} and the auxiliary function $F$ in \eqref{liftedDC}.  In \cite[Theorem~4.1]{LiPoTa18}, it was shown that if $f$ in \eqref{DC} satisfies the KL property at $\bar{x}\in\revise{{\rm dom\,}\partial f }$ with exponent $\frac12$  and $P_2$ has globally Lipschitz gradient, then $F $ in \eqref{liftedDC} satisfies the KL property at $(\bar{x},\nabla P_2({\cal A}\bar{x}))\in \revise{{\rm dom\,}\partial F} $ with exponent $\frac12$. Here we study the converse implication as a corollary to Theorem~\ref{general}.}

\begin{theorem}[{{\bf KL exponent of DC functions}}]\label{DCth}
Suppose that $f $ and $F $ are defined in \eqref{DC} and \eqref{liftedDC} respectively. If $F$ is a KL function with exponent $\alpha\in[0,1)$, then $f$ is a KL function with exponent $\alpha$.
\end{theorem}
\begin{proof}
Let $\bar x\in \revise{{\rm dom\,}\partial f}$. We will show that $f$ satisfies the KL property at $\bar x$ with exponent $\alpha$.

Note that we have $\revise{{\rm dom\,}\partial f} = \revise{{\rm dom\,}\partial P_1}$ thanks to \cite[Corollary~10.9]{RocWets98} and the fact that continuous convex functions are locally Lipschitz continuous. Hence, we actually have $\bar x\in \revise{{\rm dom\,} \partial P_1}$.

Now, using  \cite[Exercise~8.8]{RocWets98} and \cite[Proposition~10.5]{RocWets98}, we have for any $\bar \xi \in \partial P_2({\cal A}\bar x)$ that
  \begin{equation}\label{subinclu}
      \partial F (\bar x,\bar \xi)
                                     =\begin{bmatrix}
                                        \partial P_1(\bar x)- {\cal A}^*\bar \xi\ \\
                                        \partial P_2^*(\bar \xi)- {\cal A}\bar x
                                     \end{bmatrix}
                                     \supseteq\begin{bmatrix}
                                        \partial P_1(\bar x)- {\cal A}^*\bar \xi\ \\
                                        0
                                       \end{bmatrix}.
  \end{equation}
where the inclusion follows from the fact that $\partial P^*_2 = \partial P_2^{-1}$ (see \cite[Proposition~11.3]{RocWets98}).
Since $\bar x\in \revise{{\rm dom}\,\partial P_1}$, we see further from \eqref{subinclu} that $\{\bar x\}\times \partial P_2({\cal A}\bar x)\subseteq \revise{{\rm dom}\,\partial F}$.
Then \revise{condition (i)} of Theorem~\ref{general} holds because one can show using \eqref{Young} that $\Argmin_y F (\bar x,y) = \partial P_2({\cal A}\bar x)$. On the other hand, the assumption on KL property of $F$ shows that \revise{condition (ii)} of Theorem~\ref{general} holds. Now, it remains to prove that $F$ is level-bounded in $y$ locally uniformly in $x$ before we can apply Theorem~\ref{general} to establish the desired KL property.

To this end, we will show that for any $x^*\in \mathbb{X}$ and $\beta\in\R$, the following set is bounded:
\begin{align}\label{fanLi0}
  \{(x,y) :\;\|x-x^*\|\le 1,\ F (x,y)\le \beta\}.
\end{align}
Suppose to the contrary that the above set is unbounded for some $x^*$ and $\beta$.
Then there exists a sequence
\begin{align}\label{fanLi}
  \{(x^k,y^k)\}\subseteq \{(x,y) :\;\|x-x^*\|\le 1,\ F (x,y)\le \beta\}
\end{align}
with $\|y^k\|\to \infty$. Passing to a subsequence if necessary, we may assume without loss of generality that $x^k\to \tilde x$ for some $\tilde x\in B(x^*,1)$ and that $\lim_k \frac{y^k}{\|y^k\|}$ exists. Denote this latter limit by $d$. Then $\|d\|=1$. Next, using the definition of $\{(x^k,y^k)\}$ in \eqref{fanLi} and the definition of $F $, we have for all sufficiently large $k$ that
  \begin{align}
  \beta &\ge F (x^k,y^k)  = P_1(x^k) - \langle {\cal A}x^k,y^k\rangle  + P^*_2(y^k)\ge f(x^k)\label{1stineq}\\
  \Rightarrow \frac{\beta}{\|y^k\|}&\ge \frac{P_1(x^k)}{\|y^k\|} - \left\langle {\cal A}x^k,\frac{y^k}{\|y^k\|}\right\rangle  + \frac{P^*_2(y^k)}{\|y^k\|},\label{2ndineq}
  \end{align}
where the second inequality in \eqref{1stineq} follows from the definition of Fenchel conjugate.
Then we see in particular from \eqref{1stineq} and the closedness of $f$ that $\tilde x\in \revise{{\rm dom\,}f} = \revise{{\rm dom\,}P_1}$. Using this, the closedness of $P_1$ and the definition of $d$, we have upon passing to limit inferior in \eqref{2ndineq} that
  \begin{equation*}
  \begin{aligned}
   0&\ge - \langle {\cal A}\tilde x,d\rangle  +  \liminf_{k\to\infty}\frac{P^*_2(y^k)}{\|y^k\|}\overset{\rm (a)}\ge - \langle {\cal A}\tilde x,d\rangle  + (P_2^*)\color{blue}^\infty\color{black}(d)\\
   &\overset{\rm (b)}=  - \langle {\cal A}\tilde x,d\rangle  + \sigma_{\revise{{\rm dom\,}P_2}}(d) = - \langle {\cal A}\tilde x,d\rangle  +\sup_{x\in\revise{{\rm dom\,}P_2}}\{\langle x,d\rangle \},
   \end{aligned}
  \end{equation*}
   where (a) follows from \cite[Theorem~2.5.1]{AuTe2003} and (b) follows from  \cite[Theorem~2.5.4]{AuTe2003}.
   Since $\revise{{\rm dom\,} P_2 = \mathbb{Y}}$, we deduce from the above inequality that  $d = 0$, which contradicts the fact that $\|d\|=1$. Thus, we have shown that \eqref{fanLi0} is bounded for any $x^*\in \mathbb{X}$ and any $\beta\in \R$, which implies that $F$ is level-bounded in $y$ locally uniformly in $x$. This completes the proof.
\end{proof}

\subsection{Bregman envelope}\label{BreE}

\revise{In this section, we discuss the KL exponent of the Bregman envelope \eqref{Benve} of a proper closed function. We consider the following assumption on $\phi$ in \eqref{Bphi}, which is general enough for the corresponding \eqref{Benve} to include \reviser{the celebrated Moreau envelope and the forward-backward envelope introduced in \cite{StThPa17}} as special cases. Further comments on this assumption will be given in Remark~\ref{remarkEnvelope} below.}

\begin{assumption}\label{Ass11}
  The function $\phi$ \revise{in \eqref{Bphi}} is twice continuously differentiable and there exists $a_1>0$ such that for all $x\in \mathbb{X}$,
 \begin{align}\label{Ass1}
 \nabla^2 \phi(x)- a_1{\cal I}\succeq 0;
 \end{align}
 here ${\cal I}$ is the identity map, and for a linear map ${\cal A}: \mathbb{X}\to \mathbb{X}$, ${\cal A}\succeq 0$ means it is positive semidefinite, i.e., ${\cal A} = {\cal A}^*$ and $\langle h,{\cal A}h\rangle\ge 0$ for all $h \in \mathbb{X}$.
 %where $\nabla^2\phi(x)$ stands for the Hessian matrix of $\phi$ at $x$.
\end{assumption}

Given a proper closed {\color{blue} function} $f$ and a {\color{blue} function} $\phi$ satisfying Assumption~\ref{Ass11}, we first analyze the KL property of the following auxiliary function:
\begin{align}\label{obofBenve}
F(x,y):= f(y) + \B_\phi(y,x)
\end{align}
with $\B_\phi$ defined in \eqref{Bphi}. For this function, applying \cite[Proposition~8.8]{RocWets98} and \cite[Proposition~10.5]{RocWets98}, we have the following formula for $\partial F$ at any $x\in\mathbb{X}$ and $y\in \revise{{\rm dom\,}f}$,
\begin{equation}\label{partialF}
\partial  F (  x,  y)=\begin{bmatrix}
-\nabla^2\phi(  x)(y-  x)\\
\partial f(  y) + \nabla \phi(  y)- \nabla \phi(  x)
\end{bmatrix}.
\end{equation}
This formula will be used repeatedly in our discussion below.

\begin{lemma}\label{OBofBenve}
Let \revise{$f:\mathbb{X}\to\mathbb{R}\cup\{\infty\}$} be a KL function with exponent $\alpha\in[\frac12,1)$. Let $F$ be defined in \eqref{obofBenve} with $\phi$ satisfying Assumption \ref{Ass11}. Then $F$ is a KL function with exponent $\alpha$.
\end{lemma}
\begin{proof}

Thanks to \cite[Lemma~2.1]{LiPong17}, it suffices to show that $F$ satisfies the KL property at any point $(x,y)$ with $0\in\partial F (x,y)$. Let $(\bar x,\bar y)$ be such that  $0\in\partial F(\bar x,\bar y)$. Then in view of \eqref{partialF}, we see that $0\in\partial F(\bar x,\bar y)$ implies that $\nabla^2\phi( \bar x)(\bar y- \bar  x) = 0$.
%\[
%0= -\nabla^2\phi( \bar x)(\bar y- \bar  x){\rm\ and\ } 0\in \partial f(\bar   y) + \nabla \phi( \bar  y)- \nabla \phi(\bar   x).
%\]
Combining this with  \eqref{Ass1} we deduce that $\bar y=\bar x$.

Next, since $f$ is a KL function with exponent $\alpha$, there exist $c,\eta,\epsilon>0$ such that
\begin{align}\label{KLff}
  \frac{1}{c}\dist^{\frac{1}{\alpha}}(0,\partial f(y))\ge f(y)-f(\bar x)
\end{align}
whenever
\revise{$y\in B(\bar x,\epsilon)\cap {\rm dom}\,\partial f$} and $f(y)<f(\bar x) + \eta$.
Since $\phi$ is twice continuously differentiable, by shrinking $\epsilon$ further if necessary, we see that there exists $b_1>a_1$ \color{blue}with $a_1$ being as in \eqref{Ass1} \color{black} such that for any $(x,y)\in B((\bar x,\bar x),\epsilon)$, there exists $x_0\in B(\bar x,\epsilon)$ so that
\begin{equation*}
  \|\nabla \phi(  y)- \nabla \phi(  x)\|\le b_1\|y-x\|\ \ {\rm and}\ \ \langle y-x,\nabla \phi(y)- \nabla \phi(x)\rangle = \langle y-x,[\nabla^2 \phi(x_0)](y-x)\rangle.
\end{equation*}
To the second relation in the above display, apply Cauchy-Schwartz inequality to the left hand side and apply \eqref{Ass1} to the right hand side to obtain $\|y-x\|\|\nabla \phi(x) - \nabla\phi(y)\|\ge a_1\|y-x\|^2$. Combining this with the first relation in the above display, we obtain that
\begin{align}\label{b11}
b_1\|y-x\|\ge\|\nabla \phi(  y)- \nabla \phi(  x)\|\ge a_1\|y-x\|.
\end{align}

Now, combining \eqref{partialF} with \cite[Lemma~2.2]{LiPong17}, we deduce that there exists $C_0>0$ such that for $(x,y)\in B((\bar x,\bar x),\epsilon)$ \revise{ with $y\in {\rm dom}\,\partial f$},
\begin{equation}\label{disttildeF2}
\begin{aligned}
  &\dist^{\frac{1}{\alpha}}(0,\partial F (  x,  y))
  \ge  C_0\left(\|\nabla^2\phi(  x)(y-  x)\|^{\frac{1}{\alpha} } +  \inf_{\xi\in\partial f(y)}\|\xi + \nabla \phi(  y)- \nabla \phi(  x)\|^{\frac{1}{\alpha}}\right)\\
  &\overset{\rm (a)}\ge C_0\left(a_1^{\frac{1}{\alpha}}\|y-x\|^{\frac{1}{\alpha}}  + (a_1b_1^{-1})^{\frac{1}{\alpha}}\inf_{\xi\in\partial f(y)}\|\xi + \nabla \phi(  y)- \nabla \phi(  x)\|^{\frac{1}{\alpha}}\right)\\
  &\overset{\rm (b)}\ge C_0\left(a_1^{\frac{1}{\alpha}}\|y-x\|^{\frac{1}{\alpha}}  +  (a_1b_1^{-1})^{\frac{1}{\alpha}}\inf_{\xi\in\partial f(y)}\eta_1\|\xi\|^{\frac{1}{\alpha}} - (a_1b_1^{-1})^{\frac{1}{\alpha}}\eta_2\|\nabla \phi(  y)- \nabla \phi(  x)\|^{\frac{1}{\alpha}}\right)\\
  &\overset{\rm (c)}\ge C_0\left(a_1^{\frac{1}{\alpha}}\|y-x\|^{\frac{1}{\alpha}}  +  (a_1b_1^{-1})^{\frac{1}{\alpha}}\inf_{\xi\in\partial f(y)}\eta_1\|\xi\|^{\frac{1}{\alpha}} - a_1^{\frac{1}{\alpha}}\eta_2\|y-x\|^{\frac{1}{\alpha}}\right)\\
  &\ge C_1\left(   \inf_{\xi\in\partial f(y)}\|\xi\|^{\frac{1}{\alpha}} + \|y-x\|^{\frac{1}{\alpha}} \right),
  \end{aligned}
\end{equation}
where (a) follows from \eqref{Ass1} and the fact that $\left(\frac{a_1}{b_1}\right)^{\frac{1}{\alpha}}<1$, (b) follows from \cite[Lemma~3.1]{LiPong17} for some $\eta_1>0$ and $\eta_2\in(0,1)$, (c) follows from the first inequality in \eqref{b11}, and the last inequality holds with $C_1:= C_0\min\{(1-\eta_2)a_1^\frac1\alpha,\eta_1(a_1b_1^{-1})^{\frac{1}{\alpha}}\}>0$.

Next, since $\nabla \phi$ is Lipschitz continuous on $\revise{B(\bar x,\epsilon/2)}$ with Lipschitz constant $b_1$ in view of \eqref{b11}, by shrinking $\epsilon$ further,
\revise{we may assume $2b_1\epsilon^2 < 1$ and that for any $(x,y)\in B((\bar x,\bar x),\epsilon)$,}
\begin{align}\label{lessthan1}
0\le \B_\phi(y,x)=\phi(y)-\phi(x)-\langle \nabla \phi(x),y-x\rangle \le \frac{b_1}{2}\|y-x\|^2 \le \frac{b_1}2(2\epsilon)^2<1,
\end{align}
where the first inequality follows from the convexity of $\phi$. Combining this with \eqref{disttildeF2}, we deduce further that for $(x,y)\in B((\bar x,\bar x),\epsilon)$ \revise{ with $y\in {\rm dom}\,\partial f$ and} $F(x,y)<F ( \bar x, \bar  x) +\eta$,
  \begin{align*}
  &\dist^{\frac{1}{\alpha}}(0,\partial F (x,y)) \ge C_1\left(\inf_{\xi\in\partial f(y)}\|\xi\|^{\frac{1}{\alpha}} + \left(2b_1^{-1}\B_\phi(y,x)\right)^{\frac{1}{2\alpha}}\right)\\
  &\overset{\rm (a)}\ge  C_1\left(\inf_{\xi\in\partial f(y)}  \|\xi\|^{\frac{1}{\alpha}}+(2b_1^{-1})^{\frac{1}{2\alpha}}\B_\phi(y,x)\right)
  \overset{\rm (b)}= C_1c\left(\inf_{\xi\in\partial f(y)}  c^{-1}\|\xi\|^{\frac{1}{\alpha}}+(2b_1^{-1})^{\frac{1}{2\alpha}}c^{-1}\B_\phi(y,x)\right)\\
  &\overset{\rm (c)}\ge C_2\left(\inf_{\xi\in\partial f(y)}  c^{-1}\|\xi\|^{\frac{1}{\alpha}}+\B_\phi(y,x)\right)\overset{\rm (d)}\ge C_2\left(f(y)-f(\bar x)+\B_\phi(y,x)\right)\\
  & = C_2\left( F (x,y)- F (\bar x,\bar x)\right)
\end{align*}
where (a) holds because $\frac{1}{2\alpha}\le 1$ and $\B_\phi(y,x)<1$, thanks to \eqref{lessthan1}, the constant $c$ for (b) comes from \eqref{KLff}, (c) holds with $C_2 := C_1c\min\{1,(2b_1^{-1})^{\frac{1}{2\alpha}}c^{-1}\}$, (d) follows from \eqref{KLff} because $(x,y)\in B((\bar x,\bar x),\epsilon)$, \revise{$y\in {\rm dom}\,\partial f$} and $f(y)\le F(x,y)<F(\bar x,\bar x)+\eta=f(\bar x)+\eta$, and the last equality holds
because $f(\bar x)=F (\bar x,\bar x)$. This completes the proof.
\end{proof}

We are now ready to analyze the KL property of the Bregman envelope $F_\phi$ in \eqref{Benve}.
\begin{theorem}[{{\bf KL exponent of Bregman envelope}}]\label{BenvelopeKL}
 Let \revise{$f:\mathbb{X}\to\mathbb{R}\cup\{\infty\}$} be a proper closed function with  $\inf f>-\infty$.  Suppose that $\phi$ satisfies Assumption \ref{Ass11} and  that $f$ is a KL function with exponent $\alpha\in[\frac12,1)$. Then $F_\phi$ defined in \eqref{Benve} is a KL function with exponent $\alpha$.
\end{theorem}
\begin{proof}
  Let $F$ be defined as in \eqref{obofBenve}. We will use Theorem \ref{general} to deduce the KL exponent of $F_\phi$ from that of $F$. To this end, we need to check all the conditions required by Theorem \ref{general}.

  First, we claim that $F$ is level-bounded in $y$ locally uniformly in $x$. To prove this, fix any $x_0\in\mathbb{X}$ and $t\in \R$. Define
  \[
  U_{x_0}:=\{(x,y):\ \|x-x_0\|\le 1,\  F(x,y)\le t\}.
  \]
   \color{blue}Thus, it suffices to \color{black}show that $U_{x_0}$ is bounded. To this end, note that $\phi$ is strongly convex with modulus $a_1$ according to Assumption \ref{Ass11}.
%   we have for any $(x,y)\in U_{x_0}$ that,
%  \[
%  (x-y)^T\nabla^2 \phi(x)(x-y)\ge a_1\|x-y\|^2,
%  \]
%  which by \cite[Theorem~2.1.10]{Ne98}, we know that $\phi$ is strongly convex and
  We have from this and the definition of Bregman distance that for any $(x,y)\in U_{x_0}$,
  \[
  \frac{a_1}{2}\|x-y\|^2\le \B_\phi(\revise{y,x}).
  \]
  Since $\inf f > -\infty$ by assumption, we deduce further that for any $(x,y)\in U_{x_0}$,
\[
\inf f + \frac{a_1}{2}\|x-y\|^2\le \inf f + \B_\phi(\revise{y,x})\le f(y)  + \B_\phi(y,x)=F (x,y)\le t.
\]
Since $x\in B(x_0,1)$, we deduce from the above inequality that $U_{x_0}$ is bounded. Thus, we have shown that $F$ is level-bounded in $y$ locally uniformly in $x$.

%Next, fix any $x_0\in\dom F_\phi$ and any $y_0\in\Argmin_y F (x_0,y)$; such a $y_0$ exists thanks to thanks to \cite[Theorem~1.17(a)]{RocWets98}, which states that $\Argmin_y  F (x_0,y)\neq \emptyset$.  Since $F_\phi$ is level-bounded in $y$ locally uniformly in $x$, and $f(y_0) + \B_\phi(y_0,\cdot)$ is continuous on $\R^n$, we conclude from \cite[Theorem~1.17(c)]{RocWets98} that $F_\phi$ is continuous on $\dom F_\phi$. This shows that condition (iv) in Theorem~\ref{general} is satisfied. Moreover
Next, using \cite[Exercise~8.8]{RocWets98}, we have for any $x\in\revise{{\rm dom\,}\partial F_\phi}$ and any $\bar y\in\Argmin_y  F (x,y) $ that
\[
0\in\partial f(\bar y) + \nabla \B_\phi(\cdot,x)(\bar y),
\]
which implies that $\partial f(\bar y)\neq\emptyset$.
This together with \eqref{partialF} implies that $\partial F (x,\bar y)\neq\emptyset$ for any such $x$ and $\bar y$. In particular, \revise{condition (i)} in Theorem~\ref{general} is satisfied.

Finally, note that \revise{condition (ii)} in Theorem~\ref{general} is also satisfied thanks to Lemma~\ref{OBofBenve}. Thus, we deduce from Theorem~\ref{general} that $F_\phi$ satisfies the KL property with exponent $\alpha$ at any $x\in \revise{{\rm dom\,}\partial F_\phi}$.
\end{proof}

\begin{remark}\label{remarkEnvelope}
The Bregman envelope \eqref{Benve} with $\phi$ satisfying Assumption~\ref{Ass11} covers several envelopes studied in the literature.
\begin{enumerate}[{\rm (i)}]
  \item When $\phi(\cdot)=\frac{1}{2\lambda}\|\cdot\|^2$ with some $\lambda>0$, the function $F_\phi$ in \eqref{Benve} becomes
\[
F_\phi(x)=\inf_y\left\{f(y) + \frac{1}{2\lambda}\|x-y\|^2\right\}=:e_\lambda f(x).
\]
This function is known as the Moreau envelope of $f$. In \cite[Theorem~3.4]{LiPong17}, it was proved that if $f$ is a convex KL function with exponent $\alpha\in(0,\frac23)$ that is continuous on $\revise{{\rm dom\,}\partial f}$,  then $e_\lambda f$ is a  KL function with exponent $\max\left\{\frac12,\frac{\alpha}{2-2\alpha}\right\}$. Here, without the convexity and continuity  assumptions, we can obtain a tighter estimate on the KL exponent of $e_\lambda f$ via Theorem \ref{BenvelopeKL}: if $f$ is a KL function with  exponent $\alpha\in[\frac12,1)$ and $\inf f>-\infty$, then $e_\lambda f$ is a KL function with exponent $\alpha$.

  \item If the function $f$ in \eqref{Benve} takes the form $h + g$, where $g$ is a proper closed function, and $h$ is twice continuously differentiable with Lipschitz gradient whose modulus is less than $\frac1\gamma$, then the function $\phi(x) := \frac1{2\gamma}\|x\|^2 - h(x)$ is convex and satisfies Assumption~\ref{Ass11}. The forward-backward envelope $\psi_\gamma$ of the function $f = h+g$ was defined in \cite{StThPa17} as follows (see also the discussion in \cite[Section~2]{LiuPong16}):
  \[
  %\psi_\gamma(x):=\inf_y\left\{f(x) + \langle \nabla f(x),y-x\rangle  + \frac{1}{\gamma}\|y-x\|^2 + g(y)\right\}
  \psi_\gamma(x)= \inf_y\{h(y)+g(y) + \B_\phi(y,x)\}.
  \]
In \cite[Theorem~3.2]{LiuPong16}, it was shown that if the {\color{blue} first-order error bound condition (or error bound condition in the sense of Luo-Tseng)}  holds for $h+g$, with $h$ being \revise{in addition} analytic and $g$ being  \revise{in addition convex,} continuous on $\revise{{\rm dom\,}\partial g}$, subanalytic and bounded below, then \revise{$\psi_\gamma$} is a KL function with exponent $\frac12$. Here, in view of Theorem \ref{BenvelopeKL}, we can deduce the KL exponent of $\psi_\gamma$ \revise{without the convexity and (sub)analyticity assumptions}: if $f = h+g$ is a KL function with exponent $\alpha\in[\frac12,1)$ and $\inf f>-\infty$, {\color{blue} $g$ is a proper closed function}, and $h$ is twice continuously differentiable with Lipschitz gradient whose modulus is less than $\frac1\gamma$, then $\psi_\gamma$ is a KL function with exponent $\alpha$.
\end{enumerate}

\end{remark}

\subsection{Least squares loss function with rank constraint}\label{Lr}

\revise{In this section, we compute an explicit KL exponent of the function $f$ in \eqref{leastRank}, which can be rewritten as an inf-projection as in \eqref{f4=hatf4}}.
 Now, observe further that one can relax the orthogonality constraint and introduce a penalty function without changing the optimal value in \eqref{f4=hatf4}, i.e.,
\begin{equation}\label{f4=tildef4}
f (X)=\inf_{U}\bigg\{\underbrace{\frac{1}{2}\|{\cal A}X-b\|^2 + \frac12\|U^TU - I_{m-k}\|_F^2 + \delta_{\mathfrak{\tilde D}}(X,U)}_{\tilde f(X,U)}+ \delta_{\mathfrak{\tilde B}}(X,U)\bigg\},
\end{equation}
where
\[
\begin{aligned}
  \mathfrak{\tilde D}&:=\{(X,U)\in\R^{m\times n}\times\R^{m\times (m-k)}:\;U^TX=0\},\\
  \mathfrak{\tilde B}&:=\{(X,U)\in\R^{m\times n}\times\R^{m\times (m-k)}: 0.5I_{m-k}\preceq U^TU\preceq 2I_{m-k}\},
\end{aligned}
\]
where $A\preceq B$ means the matrix $B-A$ is positive semidefinite. In view of \eqref{f4=tildef4}, as another application of Theorem~\ref{general}, we will deduce the KL exponent of $f$ via that of $\tilde f + \delta_{{\frak {\tilde B}}}$.

We start with the following result, which is of independent interest.

\begin{theorem}\label{equalityconstrainLL}
  Let $h:\mathbb{X}\to \R$ and $G:\mathbb{X}\to \mathbb{Y}$ be continuously differentiable. Assume that $G^{-1}\{0\}\neq\emptyset$ and define the functions $g$ and $g_1$ by
  \begin{equation*}\label{Definefunctions}
    g(x) := h(x) + \delta_{G^{-1}\{0\}}(x),\ \ \ \ g_1(x,\lambda):= h(x) + \langle \lambda,G(x)\rangle.
  \end{equation*}
  Let $\bar x\in\revise{{\rm dom\,}\partial g}$ and suppose that the linear map $\nabla G(\bar x): \mathbb{Y}\to \mathbb{X}$ is injective. Then the following statements hold:
  \begin{enumerate}[{\rm (i)}]
    \item There exists $\epsilon>0$ so that for each $x\in B(\bar x,\epsilon)$, the function $\lambda\mapsto \|\nabla h(x) + \nabla G(x)\lambda\|$ has a unique minimizer.
    \item If $g_1$ satisfies the KL property at $(\bar x, \lambda(\bar x))$ with exponent $\alpha$, then $g$ satisfies the KL property at $\bar x$ with exponent $\alpha$, where $\lambda(\bar x)$ is the unique minimizer of $\lambda\mapsto \|\nabla h(\bar x) + \nabla G(\bar x)\lambda\|$.
  \end{enumerate}
\end{theorem}

% Then $\Argmin_\lambda\{\|\nabla F(x) + \nabla G(x)\lambda\|\}$ is unique on ${\rm dom\,}\partial f$. Denote $\lambda(x): = \argmin_\lambda\{\|\nabla F(x) + \nabla G(x)\lambda\|\}$. Then $\lambda(x)$ is continuous around $\bar x$.
\begin{proof}
We first prove (i). Since $\nabla G(\bar x)$ is an injective linear map and $x\mapsto \nabla G(x)$ is continuous,
there exists an $\epsilon > 0$ so that $\nabla G(x)$ is an injective linear map whenever $x\in B(\bar x,\epsilon)$.
Then statement (i) follows immediately because the function $\lambda\mapsto \|\nabla h(x) + \nabla G(x)\lambda\|$ is minimized if and only if the quantity $\|\nabla h(x) + \nabla G(x)\lambda\|^2$ is minimized, and this latter function is a strongly convex function in $\lambda$ whenever $x\in B(\bar x,\epsilon)$, thanks to the fact that $\nabla G(x)$ is an injective linear map \reviser{from $\mathbb{Y}$ to $\mathbb{X}$}.

We now prove (ii). Let $x\in B(\bar x,\epsilon)$ and $\lambda(x)$ denote the unique minimizer of $\lambda\mapsto \|\nabla h(x) + \nabla G(x)\lambda\|$. Then $\lambda(x)$ is also the unique minimizer of $\lambda\mapsto \|\nabla h(x) + \nabla G(x)\lambda\|^2$. Using the first-order optimality condition, we see that $\lambda(x)$ has to satisfy the relation $\nabla G(x)^*\left(\nabla h(x) + \nabla G(x)\lambda(x)\right)= 0$, which gives
\[
\lambda(x) = -(\nabla G(x)^*\nabla G(x))^{-1}(\nabla G(x)^*\nabla h(x));
\]
here the inverse exists because $\nabla G(x)$ is injective. Since $h$ and $G$ are continuously differentiable, we conclude that $\lambda$ is a continuous function on $B(\bar x,\epsilon)$.

Since $g_1$ satisfies the KL property at $(\bar{x},\lambda(\bar x))$ with exponent $\alpha$, there exist $a,\nu,c>0$ such that whenever $(x,\lambda)\in B\left((\bar x,\lambda(\bar x)),\nu \right)$ and $g_1(\bar{x},\lambda(\bar x))<g_1(x,\lambda)< g_1(\bar{x},\lambda(\bar x))+a$, it holds that
 \begin{align}\label{KL22}
   \|\nabla g_1(x,\lambda)\|\ge c\left(g_1(x,\lambda)- g_1(\bar x,\lambda(\bar x) )\right)^{\alpha}.
 \end{align}
Next, using \cite[Exercise~8.8]{RocWets98}, for any  $x\in B(\bar x,\epsilon)\cap \dom\partial g$, we have
\begin{equation*}
    \partial g(x) = \nabla h(x) + N_{G^{-1}\{0\}}(x)\subseteq  \nabla h(x) + \left\{\nabla G(x)\lambda:\;\lambda\in\mathbb{Y}\right\},
\end{equation*}
where the inclusion follows from \cite[Corollary~10.50]{RocWets98} and the injectivity of $\nabla G(x)$.
This implies that for any $x\in B(\bar x,\epsilon)\cap \dom\partial g$,
\begin{align}\label{ooo}
\dist(0,\partial g(x))\ge \inf_{\lambda}\|\nabla h(x) + \nabla G(x)\lambda\|= \|\nabla h(x) + \nabla G(x)\lambda(x)\|,
\end{align}
where the equality follows from the definition of $\lambda(x)$ as the unique minimizer.

On the other hand, we have for any $x\in \revise{{\rm dom\,}\partial g}$ and any $\lambda$ that
  \begin{align}\label{TF}
    \nabla g_1(x,\lambda) = \begin{bmatrix}
      \nabla h(x) + \nabla G(x)\lambda\\
      G(x)
    \end{bmatrix} =
    \begin{bmatrix}
      \nabla h(x) + \nabla G(x)\lambda\\
      0
    \end{bmatrix},
  \end{align}
  where the second equality holds because $G(x) = 0$ whenever $x\in \dom\partial g$. Combining \eqref{TF} with \eqref{ooo}, we then obtain for any  $x\in B(\bar x,\epsilon)\cap \dom\partial g$ that
  \begin{equation}\label{22}
    \begin{split}
      \dist(0,\partial g(x))\ge \|\nabla g_1(x,\lambda(x))\|.
    \end{split}
  \end{equation}

Now, choose $0<\epsilon'<\min\{\epsilon,\frac{\nu}{\sqrt{2}}\}$ small enough so that when $x\in B(\bar x,\epsilon')\cap\dom\partial g$, we have $\|\lambda(x)-\lambda(\bar x)\|\le \frac{\nu}{\sqrt{2}}$; such an $\epsilon'$ exists thanks to the continuity of $\lambda(\cdot)$. This implies that $(x,\lambda(x))\in B\left((\bar x,\lambda(\bar x)),\nu\right)$ whenever $x\in B(\bar x,\epsilon')\cap\dom\partial g$. Therefore, for $x\in B(\bar x,\epsilon')\cap\dom\partial g$ with $g(\bar x)<g(x)<g(\bar x) + a$, we have $(x,\lambda(x))\in B\left((\bar x,\lambda(\bar x)),\nu\right)$ and
\[
g_1(\bar{x},\lambda(\bar x))=g(\bar x)<g(x)=g_1(x,\lambda(x))< g(\bar x) + a=g_1(\bar{x},\lambda(\bar x))+a.
\]
For these $x$, combining \eqref{KL22} with \eqref{22}, we have
\begin{equation*}
    \begin{split}
      &\dist(0,\partial g(x))\ge c \, \reviser{\big(g_1(x,\lambda(x))- g_1(\bar x,\lambda(\bar x) )\big)}^{\alpha} = c \, \left(g(x)- g(\bar x)\right)^{\alpha},
    \end{split}
  \end{equation*}
where the equality holds because $G(x) = 0$ whenever $x\in \dom \partial g$. This completes the proof.
\end{proof}

We now make use of Theorem~\ref{equalityconstrainLL} to deduce the KL exponent of $\tilde f+\delta_{\frak{\tilde B}}$ in \eqref{f4=tildef4} at points $(\bar X,\bar U)\in \revise{{\rm dom\,}\partial(\tilde f+\delta_{\frak{\tilde B}})}$ with $\bar U^T\bar U = I_{m-k}$. For notational simplicity, we write
\begin{align}\label{n_0}
\tau:=mn+m(m-k)+n(m-k)-1.%+\frac{(m-k)(m-k+1)}{2}.
\end{align}

\begin{lemma}\label{tildef4}
 The function $\tilde f+\delta_{\frak{\tilde B}}$ given in \eqref{f4=tildef4} satisfies the KL property with exponent $1-\frac{1}{4\cdot 9^{\tau}}$ at points $(\bar X,\bar U)\in \revise{{\rm dom\,}\partial(\tilde f+\delta_{\frak{\tilde B}})}$ with $\bar U^T\bar U = I_{m-k}$, where $\tau$ is given in \eqref{n_0}.
\end{lemma}
\begin{proof}
  Define the function $G:\R^{m\times n}\times \R^{m\times (m-k)}\to \R^{(m-k)\times n}$ by $G(X,U) := U^TX$, one can rewrite $\tilde f$ as
  \[
  \tilde f (X,U)= \frac12\|{\cal A}X-b\|^2  + \frac12\|U^TU - I_{m-k}\|_F^2  + \delta_{G^{-1}\{0\}}(X,U).
  \]
  Now, for $X\in \R^{m\times n}$, $U\in \R^{m\times (m-k)}$ and $\Lambda\in \R^{(m-k)\times n}$, define
  \[
  \tilde f_1 (X,U,\Lambda) := \frac12\|{\cal A}X-b\|^2  + \frac12\|U^TU - I_{m-k}\|_F^2 + {\rm tr}(\Lambda^T U^TX).
  \]
  Note that \reviser{$\tilde f_1$ is a polynomial of degree $4$ on $\mathbb{R}^{\tau}$ where $\tau$ is given in \eqref{n_0}}. We deduce from \cite[Theorem~4.2]{DAKu05} that $\tilde f_1$ is a KL function with exponent $1-\frac{1}{4\cdot 9^{\tau}}$.

  Next, since $(\bar X,\bar U)\in \revise{{\rm dom\,}\partial (\tilde f + \delta_{\frak{\tilde B}})}$ with $\bar U^T\bar U = I_{m-k}$, we see that $(\bar X,\bar U)$ lies in the interior of ${\frak{\tilde B}}$. Thus, we have $(\bar X,\bar U)\in \revise{{\rm dom\,}\partial \tilde f}$.
  We will now check the conditions in Theorem \ref{equalityconstrainLL} for the functions $\tilde f_1$ and $\tilde f$ (in place of $g_1$ and $g$, \reviser{respectively}) at $(\bar X,\bar U)$. Notice first that the functions $(X,U)\mapsto \frac12\|{\cal A}X-b\|^2  + \frac12\|U^TU - I_{m-k}\|_F^2$ and $G$ are continuously differentiable, and \revise{$G^{-1}\{0\}$} is clearly nonempty. We next claim that the linear map $\nabla G(\bar X,\bar U)$ is injective. To this end, let $Y\in \ker \nabla G(\bar X,\bar U)$. Then, using the definition of the derivative mapping of $G$, for any $(H,K) \in \R^{m\times n}\times \R^{m\times (m-k)}$, we have
  \begin{align*}
  &0=\langle (H,K),[\nabla G(\bar X,\bar U)](Y)\rangle =\langle [DG(\bar X,\bar U)](H,K),Y\rangle\\
  &=\langle \bar U^TH + K^T\bar X,Y\rangle=\langle H,\bar UY\rangle  + \langle \bar XY^T,K\rangle.
  \end{align*}
  Since $H$ and $K$ are arbitrary, we deduce that
  \[
  \bar UY = 0\ \ {\rm and}\ \ \bar XY^T=0.
  \]
  These together with  $\bar U^T\bar U=I_{m-k}$ imply that $Y=0$. Thus, we have $\revise{{\rm ker\,}(\nabla G(\bar X,\bar U))}=\{0\}$, i.e., $\nabla G(\bar X,\bar U)$ is an injective linear map. Now, using Theorem \ref{equalityconstrainLL}, we conclude that $\tilde f$ satisfies the KL property at $(\bar X,\bar U)$ with exponent $1-\frac{1}{4\cdot 9^{\tau}}$.

  Finally, since $(\bar X,\bar U)\in \revise{{\rm int\,}{\frak{\tilde B}}}$, one can verify directly from the definition that, at $(\bar X,\bar U)$, the KL exponent of $\tilde f + \delta_{\frak{\tilde B}}$ is the same as that of $\tilde f$. This completes the proof.
\end{proof}

Now we are ready to compute the KL exponent of $f$ in \eqref{leastRank}. Interestingly, the derived KL exponent can be determined explicitly in terms of the number of rows/columns of the matrix involved and the upper bound constant in the rank constraint.
\begin{theorem}\label{lstsqRankKL}
  The function $f$ given in \eqref{leastRank} is a KL function with exponent $1-\frac{1}{4\cdot 9^{\tau}}$, where $\tau$ is given in \eqref{n_0}.
\end{theorem}
\begin{proof}
 Notice that $f (X)=\inf_{U}(\tilde f + \delta_{\frak{\tilde B}})(X,U)$ and that for any $X\in \revise{{\rm dom\,}\partial f}$,
 \begin{equation}\label{Argminf}
 \Argmin_{U}(\tilde f + \delta_{\frak{\tilde B}})(X,U) = \{U:\; U^TX = 0 \ \mbox{and}\ U^TU = I_{m-k}\},
 \end{equation}
 where $\tilde f+\delta_{\frak{\tilde B}}$ is given in \eqref{f4=tildef4}.
 We will check the conditions in Theorem \ref{general} and apply the theorem to deducing the KL exponent of $f$.

 First, the function $\tilde f + \delta_{\frak{\tilde B}}$ is clearly proper and closed.
 Next, for any fixed $X$, the $U$ with $(X,U)\in\mathfrak{\tilde D}\cap {\frak{\tilde B}}$ satisfies $0.5I_{m-k}\preceq U^TU\preceq 2I_{m-k}$. This shows that $\tilde f + \delta_{\frak{\tilde B}}$ is bounded in $U$ locally uniformly in $X$.
 Furthermore, \revise{for any $X\in{\rm dom\,}\partial f$ and any $U\in \Argmin_{U}(\tilde f + \delta_{\frak{\tilde B}})(X,U)$, we have using \eqref{Argminf} and \cite[Exercise~8.8]{RocWets98} that}
 \[
 \partial (\tilde f + \delta_{\frak{\tilde B}}) (X,U)=({\cal A}^*({\cal A}X - b),   0)+N_{\revise{\mathfrak{\tilde D}\cap \mathfrak{\tilde B}}}(X,U)\neq\emptyset.
 \]
 These together with \eqref{Argminf} and Lemma \ref{tildef4} implies that the conditions required by Theorem \ref{general} are satisfied. Applying Theorem \ref{general}, we conclude that $f$ is a KL function of exponent $1-\frac{1}{4\cdot 9^{\tau}}$.
\end{proof}

\section{Concluding remarks}\label{sec6}

In this paper, we show that the KL exponent is preserved via inf-projection, under mild assumptions. The result is then used for \revise{studying} KL exponents of various convex and nonconvex models, including some SDP-representable functions, convex functions involving $C^2$-cone reducible structures, Bregman envelopes, and more specifically, the sum of the least squares loss function and the indicator function of matrices of rank at most $k$.

Although several important calculus rules have been developed in this manuscript and the previous work \cite{LiPong17}, the KL \color{blue}exponents \color{black}  of some commonly used nonconvex models are still unknown, such as the least squares loss function with $\ell_{1-2}$ regularization \cite{YLHX15}. Estimating the \color{blue}exponents \color{black} for these models is an interesting future research question. Another future research direction will be to look at how KL exponent behaves under other important operations such as taking the maximum of finitely many or the supremum of infinitely many functions, \revise{as discussed in Remark~\ref{newremark3.1}}. Finally, notice that many of our results in this paper for convex models require the strict complementarity condition $0\in \revise{{\rm ri\,}\partial f(x)}$. It will be interesting to identify suitable assumptions (other than polyhedral settings) under which the strict complementarity condition can be relaxed, \revise{as discussed in Remark~\ref{newremark4.4}}.


\begin{thebibliography}{99}
\bibitem{APX17}
M. Ahn, J. S. Pang and J. Xin.
\newblock Difference-of-convex learning: directional stationarity, optimality, and sparsity.
\newblock {\em SIAM J. Optim.} 27:1637--1665, 2017.

\bibitem{Group_Fused}
\color{blue}C. M. Ala\'{i}z, \'{A}. Barbero and J. R.  Dorronsoro.
\newblock Group fused lasso.
\newblock In: Mladenov V., Koprinkova-Hristova P., Palm G., Villa A.E.P., Appollini B., Kasabov N. (eds)  {\em Artificial Neural Networks and Machine Learning--ICANN 2013}.  Lecture Notes in Computer Science, vol 8131, Springer, Berlin, Heidelberg, 2013.\color{black}

\bibitem{ArtachoGeo08}
F. J. Arag\'{o}n Artacho and M. H. Geoffroy.
\newblock Characterization of metric regularity of subdifferentials.
\newblock {\em J. Convex Anal.} 15:365--380, 2008.

\bibitem{AtBo09}
 H. Attouch and J. Bolte.
\newblock  On the convergence of the proximal algorithm for nonsmooth functions involving analytic features.
\newblock {\em Math. Program.} 116:5--16, 2009.

\bibitem{AtBoReSo10}
H. Attouch, J. Bolte, P. Redont and A. Soubeyran.
\newblock Proximal alternating minimization and projection methods for nonconvex problems: an approach based on the Kurdyka-{\L}ojasiewicz inequality.
\newblock {\em Math. Oper. Res.} 35:438--457, 2010.


\bibitem{AtBoSv13}
H. Attouch, J. Bolte and B. F. Svaiter.
\newblock Convergence of descent methods for semi-algebraic and tame problems: proximal algorithms, forward-backward splitting, and regularized Gauss-Seidel methods.
\newblock {\em Math. Program.} 137:91--129, 2013.

\bibitem{AuTe2003}
A. Auslender and M. Teboulle.
\newblock \emph{Asymptotic Cones and Functions in Optimization and Variational Inequalities}.
\newblock {Springer}, 2003.

 \bibitem{BaBo96}
 H. H. Bauschke and J. M. Borwein.
 \newblock On projection algorithms for solving convex feasibility problems.
 \newblock \emph{SIAM Review.} 38:367--426, 1996.

\bibitem{BaBoLi99}
H. H. Bauschke, J. M. Borwein and W. Li.
\newblock Strong conical hull intersection property, bounded linear regularity, Jameson's property (G), and error bounds in convex optimization.
\newblock {\em Math. Program.} 86:135--160, 1999.


\bibitem{BaCoNo2004}
H. H. Bauschke, P. L. Combettes and D. Noll.
\newblock Joint minimization with alternating Bregman proximity operators.
\newblock {\em Pac. J. Optim.} 2:401--424, 2006.

\bibitem{BenTalNemi01}
A. Ben-Tal and A. Nemirovski.
\newblock {\em Lectures on Modern Convex Optimization: Analysis, Algorithms, and Engineering Applications}.
\newblock MPS-SIAM Series on Optimization, 2001.



\bibitem{BoDaLe07}
\revise{J. Bolte, A. Daniilidis and A. Lewis.
\newblock The \L ojasiewicz inequality for nonsmooth subanalytic functions with applications to subgradient dynamical systems.
\newblock {\em SIAM J.  Optim.} 17:1205--1223, 2007.}

\bibitem{BoDaLeSh07}
{\color{blue} J. Bolte, A. Daniilidis, A. Lewis  and  M. Shiota.
\newblock Clarke subgradients of stratifiable functions,
\newblock {\em SIAM J.  Optim.}  18:556--572, 2007.}

\bibitem{BoNgPeSu17}
J. Bolte, T. P. Nguyen, J. Peypouquet and B. W. Suter.
\newblock From error bounds to the complexity of first-order descent methods for convex functions.
\newblock \emph{Math. Program.} 165:471--507, 2017.

\bibitem{BoSaTe14}
J. Bolte, S. Sabach and M. Teboulle.
\newblock Proximal alternating linearized minimization for nonconvex and nonsmooth problems.
\newblock \emph{Math. Program.} 146:459--494, 2014.

\bibitem{BoLe06}
J. Borwein and A. Lewis.
\newblock {\em Convex Analysis and Nonlinear Optimization}.
\newblock 2nd edition, Springer, 2006.

\bibitem{BoPaChPeEc10}
\color{blue}S. Boyd, N. Parikh, E. Chu, B. Peleato and J. Eckstein.
\newblock Distributed optimization and statistical learning via the alternating direction method of multipliers.
\newblock \emph{Found. Trend. in Mach. Learn.} 3:1-122, 2010.\color{black}


\bibitem{Cui_Ding_Zhao}
Y. Cui, C. Ding and X. Zhao.
\newblock Quadratic growth conditions for convex matrix optimization problems associated with spectral functions.
\newblock \emph{SIAM J. Optim.} 27:2332--2355, 2017.

\bibitem{Cui_Sun_Toh}
Y. Cui, D. F. Sun and K. C. Toh.
\newblock On the asymptotic superlinear convergence of the augmented Lagrangian method for semidefinite programming with multiple solutions.
\newblock Preprint 2016. Available at https://arxiv.org/abs/1610.00875.

\bibitem{DAKu05}
D. D'Acunto and K. Kurdyka.
\newblock Explicit bounds for the {\L}ojasiewicz exponent in the gradient inequality for polynomials.
\newblock {\em Ann. Polon. Math. } 87:51--61, 2005.

\bibitem{DonRoc09}
A. L. Dontchev and R. T. Rockafellar.
\newblock {\em Implicit Functions and Solution Mappings}.
\newblock Springer, New York, 2009.



\bibitem{DrIoLewis16}
\revise{D. Drusvyatskiy, A. D. Ioffe and A. S. Lewis.
\newblock Nonsmooth optimization using Taylor-like models: error bounds, convergence, and termination criteria.
\newblock To appear in {\em Math. Program.} https://doi.org/10.1007/s10107-019-01432-w.}

\bibitem{DrLewis18}
D. Drusvyatskiy and A. Lewis.
\newblock Error bounds, quadratic growth, and linear convergence of proximal methods.
\newblock {\em Math. Oper. Res.} 43:919--948, 2018.

\bibitem{DrLiWo17}
D. Drusvyatskiy, G. Li and H. Wolkowicz.
\newblock A note on alternating projections for ill-posed
semidefinite feasibility problems.
\newblock {\em Math. Program.} 162:537--548, 2017.



\bibitem{Fan97}
J. Fan.
\newblock Comments on ``wavelets in statistics: a review" by A. Antoniadis.
\newblock {\em J. Ital. Stat. Soc.} 6:131--138, 1997.

\bibitem{FaPa03}
F. Fachinei and J. -S. Pang.
\newblock \emph{Finite-Dimensional Variational Inequalities and Complementarity Problems.}
\newblock Springer, New York, 2003.


\bibitem{FranGarPey15}
P. Frankel, G. Garrigos and J. Peypouquet.
\newblock Splitting methods with variable metric for Kurdyka-{\L}ojasiewicz functions and general convergence rates.
\newblock {\em J. Optim. Theory Appl.} 165:874--900, 2015.

\bibitem{FriMaPong14}
M. P. Friedlander, I. Mac\^{e}do and T. K. Pong.
\newblock Gauge optimization and duality.
\newblock {\em SIAM J. Optim.} 24:1999--2022, 2014.

\bibitem{HeNi10}
 J. W. Helton and J. Nie.
 \newblock Semidefinite representation of convex sets.
 \newblock \emph{Math. Program.} 122:21--64, 2010.

\bibitem{JiangLi}
{\color{blue} R. Jiang and D. Li.
\newblock Novel reformulations and efficient algorithms for the generalized trust region subproblem.
\newblock {\em SIAM J. Optim.} 29:1603--1633, 2019.}

\bibitem{Kur98}
K. Kurdyka.
\newblock On gradients of functions definable in o-minimal structures.
\newblock {\em Ann. Inst. Fourier} 48:769--783, 1998.

\bibitem{LiMorPham15}
G. Li, B. S. Mordukhovich and T. S. Pham.
\newblock New fractional error bounds for polynomial systems with
applications to H\"{o}lderian stability in optimization and spectral theory of tensors.
\newblock {\em Math. Program.} 153:333--362, 2015.

\bibitem{LiPong16}
G. Li and T. K. Pong.
\newblock Douglas-Rachford splitting for nonconvex optimization with application to
nonconvex feasibility problems.
\newblock {\em Math. Program.} 159:371--401, 2016.


\bibitem{LiPong17}
G. Li and T. K. Pong.
\newblock Calculus of the exponent of Kurdyka-Lojasiewicz inequality and its applications to linear convergence of first-order methods.
\newblock {\em Found. Comput. Math.} 18:1199--1232, 2018.


\bibitem{LiuPong16}
T. Liu and T. K. Pong.
\newblock Further properties of the forward-backward envelope with applications to difference-of-convex programming.
 \newblock \emph{Comput. Optim. Appl.} 67:489--520, 2017.

\bibitem{LiPoTa18}
\revise{T. Liu, T. K. Pong and A. Takeda.
\newblock A refined convergence analysis of pDCA$_e$ with applications to simultaneous sparse recovery and outlier detection.
\newblock \emph{Comput. Optim. Appl.} 73:69--100, 2019.}

\bibitem{LiuWuSo16}
H. Liu, W. Wu and \color{blue}A. M. -C. So.\color{black}
\newblock Quadratic optimization with orthogonality constraints: explicit {\L}ojasiewicz exponent and linear convergence of line-search methods.
\newblock {\em ICML}, 1158--1167, 2016.

\bibitem{Loja63}
S. {\L}ojasiewicz.
\newblock Une propri\'{e}t\'{e} topologique des sous-ensembles analytiques r\'{e}els.
\newblock In {\em Les \'{E}quations aux D\'{e}riv\'{e}es Partielles}, \'{E}ditions du Centre National de la Recherche Scientifique, Paris, 87--89, 1963.


\bibitem{Bruno18}
\reviser{B. F. Louren\c{c}o, M. Muramatsu and T. Tsuchiya.
\newblock Facial reduction and partial polyhedrality,
\newblock \emph{SIAM J. Optim.}, 28:2304–2326, 2018. }

\bibitem{LuoPangRalph96}
Z. Q. Luo, J. S. Pang and D. Ralph.
\newblock {\em Mathematical Programs with Equilibrium Constraints}.
\newblock Cambridge University Press, Cambridge, 1996.

\bibitem{LuTs93}
Z. Q. Luo and P. Tseng.
\newblock Error bounds and convergence analysis of feasible descent methods: a general
approach.
\newblock \emph{Ann. Oper. Res.} 46/47:157--178,




\bibitem{PaBo2012}
N. Parikh and S. P. Boyd.
\newblock Proximal algorithms.
\newblock {\em Found. Trends Optimiz.} 1:123--231, 2013.


\bibitem{Pa00}
\revise{G. Pataki.
\newblock The geometry of semidefinite programming.
\newblock {\em In Handbook of semidefinite programming,  Internat.  Ser. Oper. Res. Management Sci.}  27:29--65. Kluwer Acad. Publ., Boston, MA, 2000.}



\bibitem{ReFaPa10}
B. Recht, M. Fazel and P. Parrilo.
\newblock Guaranteed minimum-rank solutions of linear matrix equations via nuclear norm minimization.
\newblock {\em SIAM Review} 52:471--501, 2010.



\bibitem{Roc70}
R. T. Rockafellar.
\newblock {\em Convex Analysis}.
\newblock Princeton University Press, 1970.

\bibitem{RocWets98}
R. T. Rockafellar and  \color{blue}R. J. -B. Wets. \color{black}
\newblock {\em Variational Analysis}.
\newblock Springer, 1998.

\bibitem{Shapiro}
A. Shapiro.
\newblock Sensitivity analysis of generalized equations.
\newblock \emph{J. Math. Sci.} 115:2554--2565, 2003.

\bibitem{ShSc00}
\revise{A. Shapiro and K. Scheinberg.
\newblock Duality and optimality conditions.
\newblock \emph{In Handbook of semidefinite programming,  Internat. Ser. Oper. Res. Management Sci.} 27:67--110. Kluwer Acad. Publ., Boston, MA, 2000.}


\bibitem{StThPa17}
L. Stella, A. Themelis and P. Patrinos.
\newblock Forward-backward quasi-Newton methods for nonsmooth optimization problems.
\newblock {\em Comput. Optim. Appl.} 67:443--487, 2017.

\bibitem{Sturm00}
J. F. Sturm.
\newblock Error bounds for linear matrix inequalities.
\newblock {\em SIAM J. Optim.} 10:1228--1248, 2000.


\bibitem{Tuy16}
H. Tuy.
\newblock {\em Convex Analysis and Global Optimization}.
\newblock Springer, 2nd edition, 2016.

\bibitem{TuWo12}
\revise{L. Tun\c{c}el and H. Wolkowicz.
\newblock  Strong duality and minimal representations for cone optimization.
\newblock \emph{Comput. Optim. Appl.} 53:619--648, 2012.}


\bibitem{TsYu09}
 P. Tseng and S. Yun.
\newblock A coordinate gradient descent method for nonsmooth separable minimization.
\newblock {\em Math. Program.} 117:387--423, 2009.

\bibitem{UdHoZaBo16}
M. Udell, C. Horn, R. Zadeh and S. Boyd.
\newblock Generalized low rank models.
\newblock {\em Found.   Trends in Mach. Learn.}  9:1--118, 2016.


\bibitem{Wa92}
A. Watson.
\newblock Characterization of the subdifferential of some matrix norms.
\newblock \emph{Linear Algebra Appl.} 170:33--45, 1992.

\bibitem{WCP18}
B. Wen, X. Chen and T. K. Pong.
\newblock A proximal difference-of-convex algorithm with extrapolation.
\newblock {\em Comput. Optim. Appl.} 69:297--324, 2018.




\bibitem{YLHX15}
P. Yin, Y. Lou, Q. He and J. Xin.
\newblock Minimization of $\ell_{1-2}$ for compressed sensing.
\newblock {\em SIAM J. Sci. Comput.} 37:A536--A563, 2015.

\bibitem{YuZhSo19}
\revise{M. Yue, Z. Zhou and A. M. -C. So.
 \newblock A family of inexact SQA methods for non-smooth convex minimization with provable convergence guarantees based on the Luo-Tseng error bound property.
 \newblock \emph{ Math. Program.} 174:327--358, 2019.}

\bibitem{Zhang10}
C.-H. Zhang.
\newblock Nearly unbiased variable selection under minimax concave penalty.
\newblock {\em Ann. Stat.} 38:894--942, 2010.

\bibitem{ZhSo2017}
Z. Zhou and  A. M. -C. So,
\newblock  A unified approach to error bounds for structured convex optimization problems.
\newblock \emph{Math. Program.} 165:689--728, 2017.

\bibitem{Zhou_Zhang_So}
Z. Zhou, Q. Zhang and A. M.-C. So.
\newblock $\ell_{1,p}$-norm regularization: error bounds and convergence rate analysis of first-order methods.
\newblock \emph{ICML 2015}, 1501--1510, 2015.

\end{thebibliography}
\end{document}